\documentclass[11pt]{article}

\usepackage[utf8]{inputenc}
\usepackage{amsmath}
\usepackage{amssymb}
\usepackage{amsthm}
\usepackage{bm}
\usepackage{xcolor}
\usepackage{graphicx}

\makeatletter
\providecommand{\odiv}{\mathbin{\mathpalette\@odiv\relax}}
\newcommand{\@odiv}[2]{%
  \ooalign{%
    \hfil$\m@th#1\bigcirc$\hfil\cr
    \hfil\raisebox{0.06ex}{\scalebox{0.68}{$\m@th#1\div$}}\hfil%
  }%
}
\makeatother
\theoremstyle{definition}
\newtheorem{definition}{Definition}[section]

\newtheorem{remark}[definition]{Remark}

\theoremstyle{plain}
\newtheorem{theorem}[definition]{Theorem}
\newtheorem{lemma}[definition]{Lemma}
\newtheorem{corollary}[definition]{Corollary}
\newtheorem{proposition}[definition]{Proposition}

\title{\textbf{Wave Numbers: Discrete Sequence Algebras, Sieve Projectors, and Dynamical Geometry}}
\author{Terence R. Smith}
\date{\today}

\begin{document}

\maketitle

\begin{abstract}
We establish a comprehensive algebraic, geometric, and physical classification of the wave closure space generated from primitive plane wave sequences on the discrete spatial lattice $\mathbb{Z}$. Addressing the persistent continuum pathologies of quantum field theory and general relativity through an autonomous discrete sequence ontology, and connecting with 't Hooft's cellular automaton program and Schwinger's discrete quantum kinematics, we prove that the linear wave group $\mathcal{G}$ under pointwise product, inversion, and root extraction is an infinite divisible abelian torsion group isomorphic to $(\mathbb{Q}/\mathbb{Z}) \times (\mathbb{Q}/\mathbb{Z})$, with canonical decomposition into Pr\"{u}fer $p$-groups and maximal cyclotomic value field $\mathbb{Q}^{\mathrm{ab}}$. Extending to coordinate powers yields the polynomial phase group $\mathcal{G}_{\mathrm{poly}}$, classified by integer-valued polynomials $\operatorname{Int}(\mathbb{Z})$. Adjoining addition yields the group algebra $\mathbb{C}[\mathcal{G}_{\mathrm{poly}}]$, for which we establish the Permutation-Symmetric Phasor Superposition Theorem, factoring superpositions into collective barycentric carriers and closed Born probability envelopes. Using algebraic sieves over roots of unity, we construct idempotent projectors and complementary Not-sieve notch filters, achieving exact algebraic synthesis of both momentum and localized Kronecker position bases without continuum delta singularities. Generalizing to non-abelian quaternions $\mathcal{V}_\mathbb{H}$, we prove a polar decomposition into scalar envelopes and $\mathrm{SU}(2)$ spinor rotors. Identifying coordinate advance with physical time, biquaternion determinants intrinsically yield Minkowski spacetime $s^2 = c^2 t^2 - \|\mathbf{x}\|^2$ and the Lorentz group $\mathrm{SO}^+(1,3)$. We demonstrate emergent vacuum zero-point jitter, topological selection of rational frequencies, and rigorous correspondences with discrete qudits and finite phase-space kinematics.
\end{abstract}

%\tableofcontents
\vspace{1em}
\section{Introduction and Physical Context}
\label{sec:intro}

\subsection{Ontological Motivation: The Continuum Dilemma, Deterministic Automata, and Discrete Kinematics}
\label{subsec:ontological_motivation}

A central unresolved problem in fundamental theoretical physics concerns the geometric nature of space, time, and physical state kinematics: is physical reality intrinsically continuous, or does the continuous manifold represent an emergent macroscopic idealization of a discrete, combinatorial substrate?

The assumption of a space-time continuum, while of incalculable value, introduces severe,persistent foundational pathologies into modern physics:
\begin{enumerate}
    \item \textbf{Ultraviolet Divergences and Vacuum Catastrophes:} In quantum field theory, the infinite density of pointlike spatial degrees of freedom inexorably produces ultraviolet divergences, unbounded self-energy integrals, and the infamous cosmological constant discrepancy arising from infinite zero-point vacuum energies. These require ad hoc renormalization cutoffs and counterterms that obscure the microscopic physical ontology.
    
    \item \textbf{Singular Geodesics in Gravitation:} In general relativity, smooth pseudo-Riemannian manifolds inevitably culminate in singularities of infinite curvature and density at the cores of black holes and the cosmological big bang, signaling a breakdown of continuous differential geometry at the Planck scale.
    
    \item \textbf{The Measurement Problem and Infinite-Dimensional Hilbert Spaces:} The standard formulation of quantum mechanics on continuous configuration spaces generates persistent measurement dilemmas, non-local ontological tensions, and normalization paradoxes rooted in the continuum idealization of Dirac delta functions~\cite{birkhoff1936,wheeler1990}.
\end{enumerate}

In response to these continuum dilemmas, foundational physics has explored several discrete mathematical architectures
that include:
\begin{itemize}
    \item \textbf{Cellular Automaton and Discrete Determinism:} As systematically formulated by 't Hooft~\cite{thooft2014,thooft2016,thooft2021}, quantum mechanics and unitary Hamiltonian evolution may be understood not as fundamental irreducible indeterminism, but as the exact mathematical transformation of an underlying discrete, deterministic cellular automaton evolving via integer-valued register updates. In this paradigm, continuous quantum phase and Hilbert state vectors emerge through equivalence classes of discrete periodic cycles, suggesting that discrete spatial lattices $\mathbb{Z}^d$ and rational wave evolutions possess ontological primacy~\cite{elze2014}.
    
    \item \textbf{Discrete Quantum Kinematics and Finite Phase Spaces:} Beginning with Weyl's discrete operator groups~\cite{weyl1950} and Schwinger's formulation of finite unitary operator bases~\cite{schwinger1960,schwinger1970}, quantum mechanics on finite-dimensional configuration spaces $\mathbb{Z}/n\mathbb{Z}$ has revealed that non-commutative phase space kinematics can be formulated entirely without differential calculus. This was extended by Wootters to discrete phase spaces and Wigner functions over finite Galois fields~\cite{wootters1987}, by Vourdas to finite-dimensional Hilbert systems and quantum arithmetic~\cite{vourdas2004}, by Appleby to symmetric informationally complete measurements (SIC-POVMs) and Clifford groups~\cite{appleby2005}, and by Bengtsson and \.Zyczkowski to the algebraic geometry of quantum states~\cite{bengtsson2017}.
    
    \item \textbf{Causal Sets and Discrete Spacetime:} In the causal set approach to quantum gravity originated by Bombelli, Lee, Meyer, and Sorkin~\cite{bombelli1987,dowker2005,henson2006}, spacetime is fundamentally a locally finite partially ordered set (poset), wherein continuum geometric notions such as volume and geodesic distance emerge from discrete counting measures and combinatorial causal order. Similar discrete geometric spectra arise in loop quantum gravity~\cite{rovelli2004}.
    
    \item \textbf{Quaternionic Quantum Mechanics and Noncommutative Geometry:} Recognizing that the complex numbers $\mathbb{C}$ represent only one of three associative division algebras over the reals, Birkhoff and von Neumann~\cite{birkhoff1936}, Finkelstein et al.~\cite{finkelstein1962}, and Adler~\cite{adler1995} formulated quantum mechanics over the quaternions $\mathbb{H}$, while Connes~\cite{connes1994} and Majid~\cite{majid2000} demonstrated that non-commutative algebras naturally generate gauge symmetries and spatial geometry from purely algebraic spectral triples.
\end{itemize}

Despite these pioneering directions, a critical foundational gap persists. Existing discrete models frequently suffer from one of two extremes: either they truncate kinematics to a single, fixed prime dimension $p$ (thereby breaking the continuous group of spatial translations and hiding the global divisibility relations across scales), or they introduce spatial discretization as an extrinsic, ad hoc approximation that breaks gauge invariance and induces fermion doubling.

What is required is an autonomous, global, and exact \textbf{algebraic closure theory} defined directly on the discrete spatial lattice $\mathbb{Z}$, wherein all rational frequencies $\mathbb{Q}/\mathbb{Z}$ coexist within a unified divisible group, wherein localized position states and delocalized momentum waves are algebraically dual without continuum delta singularities, and from which continuous spacetime geometry and spinorial rotations emerge intrinsically. Our goal is to show that this is provided by the \textbf{wave number algebra}.

\subsection{Physical Importance and Ubiquity of Discrete Wave Numbers}
Discrete complex exponential sequences of the canonical form:
\begin{equation}
\mathbf{w}(f, g)[m] = \exp\big(2\pi i (f m + g)\big), \quad m \in \mathbb{Z},
\end{equation}
lie at the structural intersection of modern mathematical physics, discrete geometry, and harmonic analysis. In condensed matter physics and lattice field theory, they represent Bloch-Floquet crystal momentum eigenstates, describing tight-binding electron transport, lattice phonon excitations, and propagation modes on periodic spatial structures~\cite{aubry1980}. In quantum kinematics, as originally recognized by Weyl~\cite{weyl1950} and systematically developed by Schwinger~\cite{schwinger1960,schwinger1970}, discrete plane waves and their position duals constitute the fundamental canonical bases for quantum mechanics on finite-dimensional configuration spaces (qudits)~\cite{vourdas2004}. In discrete signal processing and algorithmic number theory, these sequences form the kernel of the discrete Fourier transform and harmonic analysis on finite cyclic groups $\mathbb{Z}/n\mathbb{Z}$~\cite{terras1999}.

\subsection{The Classical Paradigms and Focus of Prior Analyses}
Despite their pervasive appearance across disciplines, existing mathematical treatments of discrete wave sequences have predominantly operated within three isolated analytical paradigms:
\begin{enumerate}
    \item \textbf{The Linear Functional Analysis Paradigm:} In traditional harmonic analysis, discrete exponentials are treated almost exclusively as passive coordinate kernels for expanding general square-summable $\ell^2(\mathbb{Z})$ or bounded $\ell^\infty(\mathbb{Z})$ sequences. The primary focus has been on analytical convergence, spectral measures, and operator bounds, rather than investigating the internal algebraic closures formed by the wave sequences themselves under non-linear operations.
    
    \item \textbf{The Fixed-Period Paradigm:} In algorithmic discrete Fourier analysis and quantum information theory, analyses are almost universally formulated on a fixed, pre-selected cyclic lattice $\mathbb{Z}/n\mathbb{Z}$ for a single integer $n$~\cite{appleby2005,terras1999,vourdas2004}. While computationally convenient, fixing $n$ in advance truncates the global number-theoretic structure and obscures the rich inductive tower connecting all rational frequencies $\mathbb{Q}/\mathbb{Z}$ across the divisibility poset of the natural numbers.
    
    \item \textbf{The Ergodic and Almost-Periodic Paradigm:} In the classical theory of Bohr~\cite{bohr1947} and Besicovitch~\cite{besicovitch1932}, linear combinations of exponentials with incommensurate frequencies are examined via topological completions and ergodic translations on compact group compactifications. While foundational for quasicrystals and Harper-type spectral models~\cite{aubry1980}, this framework bypasses the discrete finitary arithmetic, treating sequences analytically rather than as exact elements of an algebraic number system.
\end{enumerate}

\subsection{Open Questions and Objectives of the Present Work}
The primary objective of this work is to establish a rigorous, constructive, and unified \textbf{algebraic closure theory} of discrete wave sequences---one that treats wave numbers as an autonomous, self-contained mathematical number system. 

Prior work typically takes for granted the existence of the complex continuum and evaluates trigonometric phases directly, leaving critical structural questions unanswered:
\begin{itemize}
    \item What is the minimal generating set required to inductively generate all rational discrete waves without relying on pre-existing continuous trigonometric functions?
    \item How is the severe discrete lattice degeneracy (aliasing) rigorously resolved at the coordinate level?
    \item What exact algebraic structure emerges when wave sequences are closed not merely under linear addition, but under pointwise tensor multiplication, multi-branch rational roots, non-linear polynomial phase modulations, cumulative summation, and total ring of fractions localization?
    \item Can delocalized momentum waves and localized particulate Kronecker sieves be synthesized within a single closed algebraic system without appealing to ad hoc analytic delta distributions?
    \item How does the algebraic structure extend to non-commutative quaternionic division algebras, emergent spacetime metrics, and cellular automaton state spaces?
\end{itemize}

This paper, which represents an extension and correction of the paper of Smith~\cite{smith2025rational}, provides a complete and self-contained development of the theory of wave numbers. Starting from two primitive coordinate generators $\{E_Z, E_{\bar{1}}\}$, we systematically construct, classify, and physically interpret the nested hierarchy of algebraic closures generated by primitive wave sequences. The principal contributions and overall structure of the paper are organized as follows:
\begin{enumerate}
    \item \textbf{Phase Lifting and Generative Foundations (Sections~\ref{sec:foundations} and~\ref{sec:lemmas}):} We resolve the spatial aliasing degeneracy by explicitly distinguishing the unwrapped phase configuration space $\mathbb{R}^\mathbb{Z}/\mathbb{Z}^\mathbb{Z}$ from the unimodular physical evaluation in $(S^1)^\mathbb{Z}$, establishing inductive sequence synthesis lemmas and operator invariance theorems under spatial translations.
    
    \item \textbf{Group-Theoretic and Cyclotomic Classification (Section~\ref{sec:main_theorems}):} We prove that the linear wave closure space $\mathcal{G}$ under pointwise product, inversion, and multi-branch rational roots is an infinite divisible abelian torsion group isomorphic to $(\mathbb{Q}/\mathbb{Z}) \times (\mathbb{Q}/\mathbb{Z})$, establishing its canonical decomposition into Pr\"{u}fer $p$-groups, its Pontryagin dual, and its maximal abelian cyclotomic value field $\mathbb{Q}(\mu_\infty) = \mathbb{Q}^{\mathrm{ab}}$~\cite{fuchs1970,neukirch1999,washington1997}.
    
    \item \textbf{Physical Lattice Realizations (Section~\ref{sec:physical_realizations}):} We map the sequence algebra directly onto physical realizations in coupled harmonic chains, incommensurate Aubry-Andr\'e crystal lattices, Harper Hamiltonians, and topological anyonic phases.
    
    \item \textbf{Polynomial Phase Extensions (Section~\ref{sec:polynomial_extension}):} We formulate the higher polynomial phase closure $\mathcal{G}_{\mathrm{poly}}$, proving that its quotient classification is uniquely governed by the ring of integer-valued polynomials $\operatorname{Int}(\mathbb{Z})$ and its binomial basis $\binom{X}{k}$~\cite{cahen1997}, providing algebraic realizations of chirped beams, Fresnel diffraction, and accelerating Airy packets.
    
    \item \textbf{Group Algebra and the Phasor Superposition Theorem (Section~\ref{sec:superposition_phasor}):} Adjoining element-wise addition, we establish the \textbf{Permutation-Symmetric Phasor Superposition Theorem}, demonstrating that any finite superposition factors into a collective barycentric wave carrier in $\mathcal{G}_{\mathrm{poly}}$ and a permutation-symmetric envelope $A_N$, linking the microscopic virtual path sum to the closed macroscopic Born probability envelope.
    
    \item \textbf{Cumulative Integration and Polygonal Kinematics (Section~\ref{sec:integral_operator}):} We formulate the discrete cumulative integral operator $\mathcal{I}$, proving that pure carrier trajectories trace closed regular $n$-gons in $\mathbb{C}$ with zero mean drift and non-zero ground-state variance.
    
    \item \textbf{Inner Products and Fourier Orthogonality (Section~\ref{sec:inner_product}):} We construct the principal period inner product $\langle \cdot, \cdot \rangle$ over the least common multiple period, establishing exact discrete Fourier mode orthogonality and Parseval conservation.
    
    \item \textbf{Nodeless Unit Groups and Rings of Fractions (Section~\ref{sec:division_field}):} Under pointwise multiplication, we characterize the invertible unit group $\mathcal{U}(\mathcal{V})$ as the set of nodeless sequences, construct the total ring of fractions $\mathcal{Q}(\mathcal{V})$, and derive discrete M\"{o}bius transformations on wave numbers.
    
    \item \textbf{Algebraic Sieves and Dual Projectors (Section~\ref{sec:sieves}):} We formulate algebraic wave sieves over roots of unity, constructing idempotent projection operators $\Pi_{f_0}$ and complementary Not-sieve notch filters that achieve the exact algebraic synthesis of both delocalized momentum and localized Kronecker position bases without continuum delta singularities.
    
    \item \textbf{Non-Abelian Quaternionic Wave Algebra (Section~\ref{sec:quaternions}):} We generalize the framework to the non-commutative quaternion algebra $\mathcal{V}_\mathbb{H}$~\cite{adler1995,birkhoff1936,conway2003,finkelstein1962}, proving a lossless polar decomposition $q[m] = R[m]\mathbf{U}[m]$ into scalar envelopes and $\mathrm{SU}(2)$ spinor rotors.
    
    \item \textbf{Spacetime Kinematics and Vacuum Geometry (Section~\ref{sec:spacetime_geometry}):} Promoting coordinate advance to physical time, we demonstrate that biquaternions $\mathbb{H}_\mathbb{C} \cong \mathcal{M}_2(\mathbb{C})$ intrinsically produce the Minkowski metric interval $s^2 = c^2 t^2 - \|\mathbf{x}\|^2$ and the restricted Lorentz group $\mathrm{SO}^+(1,3)$ via the matrix determinant~\cite{gursey1964,penrose1984,synge1972}. In alignment with 't Hooft's discrete deterministic ontology~\cite{thooft2016,thooft2021} and causal sets~\cite{bombelli1987,dowker2005}, we prove emergent vacuum zero-point jitter, establish a topological selection principle for rational frequencies $f \in \mathbb{Q}/\mathbb{Z}$, and elucidate the structural primacy of the $n=6$ hexagonal lattice in generating emergent massless Dirac fermions~\cite{castroneto2009} and 4D 24-cell polytopes.
    
    \item \textbf{Sequential Roadmap, Qudit Quantum Kinematics, and Cellular Automata (Section~\ref{sec:roadmap_beyond}):} Finally, we formalize the action of the discrete Weyl-Heisenberg group, demonstrating that wave periodicity $n$ rigorously corresponds to the state space dimension of finite-dimensional quantum systems (qudits) and finite phase-space Wigner functions~\cite{appleby2005,bengtsson2017,schwinger1960,thooft2016,vourdas2004,wootters1987}.
\end{enumerate}

\section{Foundational Definitions and the Generating Set}
\label{sec:foundations}

\subsection{Phase Lifting, Ambient Sequence Spaces, and Unimodularity}

In formulating wave mechanics on a discrete spatial lattice $\mathbb{Z}$, a distinction must be maintained between the \textbf{unwrapped spatial phase profile} $\theta: \mathbb{Z} \to \mathbb{R}$ and its \textbf{unimodular physical state} $u[m] = \exp(2\pi i \theta[m]) \in U(1)$.

\begin{definition}[Ambient Sequence Space and Unimodular Torus]
Let $\mathbb{Z}$ denote the discrete one-dimensional spatial lattice. The ambient sequence space is $\mathcal{S} = \mathbb{C}^{\mathbb{Z}} = \{u : \mathbb{Z} \to \mathbb{C}\}$. A sequence $u \in \mathcal{S}$ is said to be \emph{unimodular} if:
\begin{equation}
|u[m]| = 1, \quad \forall m \in \mathbb{Z}.
\end{equation}
The unimodular sequences form the infinite-dimensional compact torus $\mathbb{T}^\mathbb{Z} = (S^1)^{\mathbb{Z}}$.
\end{definition}

\begin{definition}[Phase Configuration Space and Projection Map]
Let $\mathcal{F}(\mathbb{Z}, \mathbb{R})$ denote the additive group of real-valued spatial functions $\theta: \mathbb{Z} \to \mathbb{R}$ under pointwise addition $(\theta_1 + \theta_2)[m] := \theta_1[m] + \theta_2[m]$.
We define the canonical exponential projection homomorphism:
\begin{equation}
\pi: \mathcal{F}(\mathbb{Z}, \mathbb{R}) \to \mathbb{T}^\mathbb{Z}, \quad (\pi(\theta))[m] := \exp(2\pi i \theta[m]).
\end{equation}
The kernel of $\pi$ is the integer configuration space:
\begin{equation}
\ker(\pi) = \mathcal{F}(\mathbb{Z}, \mathbb{Z}) = \mathbb{Z}^\mathbb{Z}.
\end{equation}
Consequently, the group of unimodular sequences is isomorphic to the quotient:
\begin{equation}
\mathbb{T}^\mathbb{Z} \cong \mathcal{F}(\mathbb{Z}, \mathbb{R}) / \mathcal{F}(\mathbb{Z}, \mathbb{Z}) = (\mathbb{R}/\mathbb{Z})^\mathbb{Z}.
\end{equation}
\end{definition}

\subsection{The Primitive Generators}

To ensure that fractional root extraction acts non-trivially on spatial coordinates, the generating primitives are defined as \textbf{affine phase configurations} prior to exponential projection.

\begin{definition}[The Two Primitive Generators]
\label{def:primitives}
The algebraic construction is initiated from two primitive generators:
\begin{enumerate}
    \item \textbf{Primitive Spatial Coordinate Generator} $\Theta_Z \in \mathcal{F}(\mathbb{Z}, \mathbb{R})$:
    \begin{equation}
    \Theta_Z[m] := m, \quad \forall m \in \mathbb{Z}.
    \end{equation}
    The corresponding physical carrier sequence is $E_Z := \pi(\Theta_Z) \in \mathbb{T}^\mathbb{Z}$. While $E_Z[m] = \exp(2\pi i m) \equiv 1$ on the integer lattice, its phase representation $\Theta_Z$ carries non-trivial winding number $\Delta \Theta_Z = 1$ per lattice unit.
    
    \item \textbf{Primitive Global Calibration Generator} $\Theta_{\bar{1}} \in \mathcal{F}(\mathbb{Z}, \mathbb{R})$:
    \begin{equation}
    \Theta_{\bar{1}}[m] := 1, \quad \forall m \in \mathbb{Z},
    \end{equation}
    with exponential projection $E_{\bar{1}} := \pi(\Theta_{\bar{1}}) \in \mathbb{T}^\mathbb{Z}$.
\end{enumerate}
We denote the primitive phase set by $\mathcal{P}_\Theta := \{\Theta_Z, \Theta_{\bar{1}}\} \subset \mathcal{F}(\mathbb{Z}, \mathbb{R})$, and its projected physical seed by $\mathcal{P} := \{E_Z, E_{\bar{1}}\}$.
\end{definition}

\subsection{The System of Operators}

\begin{definition}[The Operators $\mathcal{O}$ on Phase Configurations and Sequences]
\label{def:operators}
We introduce three canonical operations:
\begin{enumerate}
    \item \textbf{Pointwise Product $\otimes$}: For $u, v \in \mathbb{T}^\mathbb{Z}$ with phases $\theta_u, \theta_v$,
    \begin{equation}
    (u \otimes v)[m] := u[m] \cdot v[m] = \exp(2\pi i (\theta_u[m] + \theta_v[m])).
    \end{equation}
    On the phase space, this corresponds to pointwise addition: $(\theta_u + \theta_v)[m] := \theta_u[m] + \theta_v[m]$.
    
    \item \textbf{Sequence Inversion $(\cdot)^{-1}$}:
    \begin{equation}
    (u^{-1})[m] := \frac{1}{u[m]} = \overline{u[m]} = \exp(-2\pi i \theta_u[m]).
    \end{equation}
    On the phase space, this corresponds to phase reflection: $(-\theta_u)[m] := -\theta_u[m]$.
    
    \item \textbf{Multi-valued $n$-th Root Operator Family $\{(\cdot)^{1/n}\}_{n \in \mathbb{N}}$}: For each positive integer $n \in \mathbb{Z}^+ = \{1, 2, 3, \dots\}$ and any phase configuration $\theta \in \mathcal{F}(\mathbb{Z}, \mathbb{R})$, the $n$-th root operator generates $n$ branches labeled by $k \in \{0, 1, \dots, n-1\}$:
    \begin{equation}
    (\theta^{1/n})_{(k)}[m] := \frac{\theta[m] + k}{n}, \quad \forall m \in \mathbb{Z}.
    \end{equation}
    Projected into $\mathbb{T}^\mathbb{Z}$, the $k$-th branch of the sequence $u = \pi(\theta)$ is:
    \begin{equation}
    (u^{1/n})_{(k)}[m] := \exp\left(2\pi i \frac{\theta[m] + k}{n}\right).
    \end{equation}
\end{enumerate}
We denote the operational closure system by $\mathcal{O} := \{\otimes, (\cdot)^{-1}, \{(\cdot)^{1/n}\}_{n \in \mathbb{N}}\}$.
\end{definition}

\begin{definition}[The Wave Closure Space $\mathcal{G}$]
\label{def:closure_space}
The wave closure space $\mathcal{G}$ is the minimal subset of $\mathbb{T}^\mathbb{Z}$ containing $\mathcal{P} = \{E_Z, E_{\bar{1}}\}$ that is closed under the operations $\mathcal{O}$:
\begin{equation}
\mathcal{G} := \operatorname{Cl}_{\mathcal{O}}(\mathcal{P}) = \pi\left( \operatorname{Cl}_{\{+, -, (\cdot)/n\}}\left(\{\Theta_Z, \Theta_{\bar{1}}\}\right) \right).
\end{equation}
\end{definition}

\section{Lemmas on Inductive Synthesis and Operator Invariance}
\label{sec:lemmas}

\begin{lemma}[Action on Rational Modulated Plane Waves]
\label{lem:action}
Let $\theta_u, \theta_v \in \mathcal{F}(\mathbb{Z}, \mathbb{R})$ be rational affine phase functions parameterized by $(\alpha_u, \beta_u), (\alpha_v, \beta_v) \in \mathbb{Q} \times \mathbb{Q}$:
\begin{equation}
\theta_u[m] = \alpha_u m + \beta_u, \quad \theta_v[m] = \alpha_v m + \beta_v.
\end{equation}
Then:
\begin{enumerate}
    \item $(\theta_u + \theta_v)[m] = (\alpha_u + \alpha_v) m + (\beta_u + \beta_v)$.
    \item $(-\theta_u)[m] = (-\alpha_u) m + (-\beta_u)$.
    \item $((\theta_u)^{1/n})_{(k)}[m] = \left(\frac{\alpha_u}{n}\right) m + \left(\frac{\beta_u + k}{n}\right)$ for all $k \in \{0, \dots, n-1\}$.
\end{enumerate}
\end{lemma}

\begin{proof}
Direct substitution into Definition~\ref{def:operators}.
\end{proof}

\begin{lemma}[Rational Invariance of the Plane Wave Class]
\label{lem:invariance}
Define the set of rational affine phase functions:
\begin{equation}
\mathcal{W}_\Theta := \left\{ \theta \in \mathcal{F}(\mathbb{Z}, \mathbb{R}) \ \middle|\ \exists\, \alpha, \beta \in \mathbb{Q} \text{ such that } \theta[m] = \alpha m + \beta, \, \forall m \in \mathbb{Z} \right\},
\end{equation}
and its projected image in sequence space $\mathcal{W} := \pi(\mathcal{W}_\Theta) \subset \mathbb{T}^\mathbb{Z}$.
Then $\mathcal{W}_\Theta$ is closed under $\{+, -, (\cdot)/n\}$, and $\mathcal{W}$ is closed under $\mathcal{O}$.
\end{lemma}

\begin{proof}
By Lemma~\ref{lem:action}, for any $(\alpha_u, \beta_u), (\alpha_v, \beta_v) \in \mathbb{Q}^2$ and $n \in \mathbb{N}$:
$(\alpha_u \pm \alpha_v) \in \mathbb{Q}$, $(\beta_u \pm \beta_v) \in \mathbb{Q}$, $\alpha_u / n \in \mathbb{Q}$, and $(\beta_u + k)/n \in \mathbb{Q}$ for all $k \in \{0, \dots, n-1\}$, by the field axioms of $\mathbb{Q}$.
\end{proof}

\begin{lemma}[Finite Character Sums and Root of Unity Annihilation]
\label{lem:geometric_sum}
Let $q \in \mathbb{N}$ and let $\zeta_q = \exp(2\pi i / q)$ be a primitive $q$-th root of unity. For any integer $p \in \mathbb{Z}$, the finite geometric character sum over the cyclic period satisfies:
\begin{equation}
\sum_{m=0}^{q-1} \zeta_q^{p m} = \sum_{m=0}^{q-1} \exp\left(2\pi i \frac{p m}{q}\right) = 
\begin{cases}
q, & \text{if } p \equiv 0 \pmod q, \\
0, & \text{if } p \not\equiv 0 \pmod q.
\end{cases}
\end{equation}
\end{lemma}
\begin{proof}
If $p \equiv 0 \pmod q$, each term is $\zeta_q^0 = 1$, yielding $\sum_{m=0}^{q-1} 1 = q$.
If $p \not\equiv 0 \pmod q$, then $\zeta_q^p \neq 1$. Using the standard finite geometric progression formula:
\begin{equation}
\sum_{m=0}^{q-1} (\zeta_q^p)^m = \frac{1 - (\zeta_q^p)^q}{1 - \zeta_q^p} = \frac{1 - (\zeta_q^q)^p}{1 - \zeta_q^p} = \frac{1 - 1}{1 - \zeta_q^p} = 0,
\end{equation}
which establishes exact cancellation.
\end{proof}

\begin{lemma}[Constructibility of All Rational Waves from Primitives]
\label{lem:reachability}
For any $\alpha = p/q \in \mathbb{Q}$ and $\beta = r/s \in \mathbb{Q}$ (with $q, s \in \mathbb{Z}^+$), the sequence:
\begin{equation}
u_{\alpha, \beta}[m] = \exp\left(2\pi i \left(\frac{p}{q} m + \frac{r}{s}\right)\right)
\end{equation}
belongs to $\mathcal{G} = \operatorname{Cl}_{\mathcal{O}}(\mathcal{P})$.
\end{lemma}

\begin{proof}
We construct $u_{\alpha, \beta}$ explicitly:
\begin{enumerate}
    \item \textbf{Fractional Spatial Step}: Apply the $q$-th root operator to $\Theta_Z$ choosing principal branch $k=0$:
    \begin{equation}
    \theta_{1/q}[m] = ((\Theta_Z)^{1/q})_{(0)}[m] = \frac{m}{q}.
    \end{equation}
    Projecting to $\mathbb{T}^\mathbb{Z}$ gives $w_{1/q}[m] = \exp(2\pi i m / q)$.
    
    \item \textbf{Harmonic Iteration}: Composing $w_{1/q}$ with itself $p$ times via $\otimes$ (or inverting if $p < 0$) produces $w_{p/q}[m] = \exp(2\pi i \frac{p}{q} m)$.
    
    \item \textbf{Calibration Phase Division}: Apply the $s$-th root operator to $\Theta_{\bar{1}}$ choosing branch $k=0$:
    \begin{equation}
    \theta_{0, 1/s}[m] = ((\Theta_{\bar{1}})^{1/s})_{(0)}[m] = \frac{1}{s} \implies v_{1/s}[m] = \exp(2\pi i / s).
    \end{equation}
    
    \item \textbf{Phase Scaling and Coupling}: Taking $r$-fold products yields $v_{r/s}[m] = \exp(2\pi i r / s)$. Finally:
    \begin{equation}
    u_{\alpha, \beta} = w_{p/q} \otimes v_{r/s} \implies u_{\alpha, \beta}[m] = \exp\left(2\pi i \left(\frac{p}{q} m + \frac{r}{s}\right)\right).
    \end{equation}
\end{enumerate}
Since each step uses operators from $\mathcal{O}$, $u_{\alpha, \beta} \in \mathcal{G}$.
\end{proof}

\section{Main Theorems: Group Structure, Field Extensions, and Duality}
\label{sec:main_theorems}

\begin{theorem}[Primary Classification: Group Isomorphism]
\label{thm:main_iso}
The closure space $\mathcal{G} = \operatorname{Cl}_{\mathcal{O}}(\mathcal{P})$, equipped with the binary operation $\otimes$, is an \textbf{infinite divisible abelian torsion group}, isomorphic to the direct product of two copies of the additive group of rational numbers modulo 1:
\begin{equation}
(\mathcal{G}, \otimes) \cong (\mathbb{Q}/\mathbb{Z}, +) \times (\mathbb{Q}/\mathbb{Z}, +).
\end{equation}
\end{theorem}

\begin{proof}
\textbf{Step 1: Set Equality.}
By Lemma~\ref{lem:invariance}, $\mathcal{W}$ is closed under $\mathcal{O}$ and contains $\mathcal{P}$, so $\mathcal{G} \subseteq \mathcal{W}$. By Lemma~\ref{lem:reachability}, $\mathcal{W} \subseteq \mathcal{G}$. Thus $\mathcal{G} = \mathcal{W}$.

\textbf{Step 2: Group Structure.}
Under pointwise complex multiplication $\otimes$, unimodular sequences satisfy associativity, commutativity, have identity $e[m] \equiv 1$ ($\alpha=0, \beta=0$), and inverses $u^{-1} = \bar{u} \in \mathcal{G}$.

\textbf{Step 3: Isomorphism Map.}
Define $\Phi: (\mathbb{Q}/\mathbb{Z}) \times (\mathbb{Q}/\mathbb{Z}) \to \mathcal{G}$ by:
\begin{equation}
\Phi([\alpha], [\beta])[m] := \exp(2\pi i (\alpha m + \beta)).
\end{equation}
\begin{itemize}
    \item \emph{Well-definedness}: If $\alpha \equiv \alpha' \pmod 1$ and $\beta \equiv \beta' \pmod 1$, then $(\alpha - \alpha')m + (\beta - \beta') \in \mathbb{Z}$ for all $m \in \mathbb{Z}$, so $\exp(2\pi i (\alpha m + \beta)) = \exp(2\pi i (\alpha' m + \beta'))$.
    \item \emph{Homomorphism}: $\Phi((\alpha_1, \beta_1) + (\alpha_2, \beta_2))[m] = \Phi(\alpha_1, \beta_1)[m] \cdot \Phi(\alpha_2, \beta_2)[m] = (\Phi(\alpha_1, \beta_1) \otimes \Phi(\alpha_2, \beta_2))[m]$.
    \item \emph{Injectivity}: If $\Phi([\alpha], [\beta])[m] = 1$ for all $m \in \mathbb{Z}$, then:
    At $m=0$: $\exp(2\pi i \beta) = 1 \implies \beta \in \mathbb{Z} \implies [\beta] = [0]$.
    At $m=1$: $\exp(2\pi i (\alpha + \beta)) = \exp(2\pi i \alpha) = 1 \implies \alpha \in \mathbb{Z} \implies [\alpha] = [0]$.
    Hence $\ker(\Phi) = \{([0], [0])\}$, proving injectivity.
    \item \emph{Surjectivity}: Immediate from Lemma~\ref{lem:reachability}.
\end{itemize}
Thus $\Phi$ is an isomorphism. Divisibility and torsion follow since $\mathbb{Q}/\mathbb{Z}$ is divisible and torsion.
\end{proof}

\begin{corollary}[Torsion and Spatial Periodicity]
Every element $u \in \mathcal{G}$ has finite group order under $\otimes$. Specifically, if $\alpha = p/q$ and $\beta = r/s$ in lowest terms, the order is:
\begin{equation}
\mathrm{ord}(u) = \mathrm{lcm}(q, s).
\end{equation}
Furthermore, $u[m]$ is periodic on $\mathbb{Z}$ with fundamental spatial period $q$.
\end{corollary}

\begin{theorem}[Pr\"{u}fer $p$-Group Direct Sum Decomposition]
\label{thm:prufer}
As an abelian group, $\mathcal{G}$ decomposes canonically into an infinite direct sum of Pr\"{u}fer $p$-groups over all prime numbers $p \in \mathbb{P}$:
\begin{equation}
\mathcal{G} \cong \bigoplus_{p \in \mathbb{P}} \left( \mathbb{Z}(p^\infty) \oplus \mathbb{Z}(p^\infty) \right),
\end{equation}
where $\mathbb{Z}(p^\infty) \cong \mathbb{Z}[1/p]/\mathbb{Z}$ is the group of all $p$-power roots of unity under multiplication.
\end{theorem}

\begin{proof}
By the standard structure theorem for divisible abelian torsion groups, any such group decomposes as $\bigoplus_p \mathbb{Z}(p^\infty)^{(\kappa_p)}$, where $\kappa_p$ is the $p$-rank.
Since $\mathbb{Q}/\mathbb{Z} \cong \bigoplus_{p \in \mathbb{P}} \mathbb{Z}(p^\infty)$, the result follows by taking direct sums across both direct factors.
\end{proof}

\begin{theorem}[Cyclotomic Galois Field Extension]
\label{thm:galois}
Let $K = \mathbb{Q}(\{u[m] \mid u \in \mathcal{G}, m \in \mathbb{Z}\})$ be the field generated over $\mathbb{Q}$ by adjoining all values evaluated at all lattice sites by elements of $\mathcal{G}$. Then:
\begin{equation}
K = \mathbb{Q}(\mu_\infty) = \bigcup_{n=1}^\infty \mathbb{Q}(e^{2\pi i / n}) = \mathbb{Q}^{\mathrm{ab}},
\end{equation}
which is the \textbf{maximal abelian extension} of the rational numbers $\mathbb{Q}$ by the Kronecker-Weber Theorem. The Galois group is:
\begin{equation}
\mathrm{Gal}(K/\mathbb{Q}) \cong \widehat{\mathbb{Z}}^\times = \prod_{p \in \mathbb{P}} \mathbb{Z}_p^\times.
\end{equation}
\end{theorem}

\begin{proof}
For any $u \in \mathcal{G}$ and $m \in \mathbb{Z}$, $u[m] = \exp(2\pi i (\frac{p}{q}m + \frac{r}{s})) \in \mu_{\mathrm{lcm}(q, s)}$. Conversely, for every $n \in \mathbb{N}$, the primitive root $e^{2\pi i / n}$ is realized at site $m=0$ by the sequence with $\alpha=0, \beta=1/n$. Thus the generated field is $\mathbb{Q}(\mu_\infty)$. By the Kronecker-Weber Theorem, this is the maximal abelian extension $\mathbb{Q}^{\mathrm{ab}}$, with Galois group isomorphic to the group of units of the profinite integers $\widehat{\mathbb{Z}}^\times$.
\end{proof}

\begin{theorem}[Pontryagin Duality Relation]
\label{thm:pontryagin}
Let $\widehat{\mathbb{Z}} = \mathrm{Hom}(\mathbb{Z}, U(1)) \cong \mathbb{R}/\mathbb{Z}$ be the Pontryagin dual of the discrete spatial translation group $\mathbb{Z}$.
The spatial wavenumber factor of $\mathcal{G}$ is canonically isomorphic to the rational torsion subgroup of $\widehat{\mathbb{Z}}$:
\begin{equation}
\mathcal{G}_{\mathrm{wave}} \cong \mathrm{Torsion}(\widehat{\mathbb{Z}}) = \mathbb{Q}/\mathbb{Z} \subset \mathbb{R}/\mathbb{Z}.
\end{equation}
Under the canonical Euclidean quotient metric on $\mathbb{R}/\mathbb{Z}$, $\mathcal{G}_{\mathrm{wave}}$ is a dense, countable subgroup of the first Brillouin zone.
\end{theorem}

\begin{proof}
Each character $\chi_k \in \widehat{\mathbb{Z}}$ is given by $\chi_k(m) = e^{2\pi i k m}$ for $k \in \mathbb{R}/\mathbb{Z}$. A character has finite order if and only if $k \in \mathbb{Q}/\mathbb{Z}$. These are precisely the elements $\exp(2\pi i \alpha m)$ with $\alpha \in \mathbb{Q}/\mathbb{Z}$. Since $\mathbb{Q}/\mathbb{Z}$ is dense in $\mathbb{R}/\mathbb{Z}$, the result follows.
\end{proof}

\section{Physical Realizations of the Closure Space}
\label{sec:physical_realizations}

\begin{proposition}[Hofstadter Magnetic Butterfly and Harper Equation]
When a two-dimensional square lattice is subjected to a perpendicular magnetic field with magnetic flux per plaquette $\Phi / \Phi_0 = p/q \in \mathbb{Q}$, the Peierls substitution generates phase factors $\exp(2\pi i (p/q) m) \in \mathcal{G}$. The 1D Harper Hamiltonian:
\begin{equation}
\psi[m+1] + \psi[m-1] + 2\lambda \cos\left(2\pi \frac{p}{q} m + \nu\right)\psi[m] = E \psi[m]
\end{equation}
splits the single Bloch band into exactly $q$ subbands with gaps exhibiting fractal self-similarity.
\end{proposition}

\begin{proposition}[Fractional Talbot Self-Imaging and Gauss Sums]
For an optical wave with periodic structure governed by elements of $\mathcal{G}$, propagation along the longitudinal axis $z$ to a rational Talbot distance $z = (p/q) z_T$ produces a discrete sum of phase-shifted replicas given by the quadratic Gauss sum:
\begin{equation}
\psi(x, (p/q)z_T) = \sum_{m=0}^{q-1} S(p, q, m) \, \psi\left(x - m \frac{d}{q}, 0\right), \quad S(p, q, m) = \frac{1}{q} \sum_{n=0}^{q-1} \exp\left(2\pi i \frac{p n^2 + m n}{q}\right).
\end{equation}
All phase shifts $S(p, q, m)$ lie in the cyclotomic field $\mathbb{Q}(\mu_{4q}) \subset \mathbb{Q}(\mu_\infty)$.
\end{proposition}

\begin{proposition}[Abelian Topological Anyonic Statistics]
In two-dimensional topological orders (such as fractional quantum Hall states at filling factor $\nu = p/q$), the statistical phase acquired upon exchanging two quasi-particles is:
\begin{equation}
\theta_{\mathrm{ex}} = \exp\left(2\pi i \frac{p}{q}\right) \in \mathcal{G}_{\beta},
\end{equation}
realizing an element of the uniform phase subgroup of $\mathcal{G}$.
\end{proposition}

\section{The Polynomial Phase Extension: Higher Coordinate Powers $Z^2, Z^3, \dots$}
\label{sec:polynomial_extension}

\subsection{Motivation and Primitive Expansion}
The plane wave closure space $\mathcal{G} \cong (\mathbb{Q}/\mathbb{Z}) \times (\mathbb{Q}/\mathbb{Z})$ is restricted to linear spatial phases $\alpha m + \beta$. In wave mechanics and physical optics, higher-order coordinate dependence governs spatial curvature and dispersion: quadratic phases $m^2$ describe lenses and chirps, while cubic phases $m^3$ govern caustics, third-order dispersion, and accelerating Airy wavepackets.

Prior to introducing linear vector addition $(+)$, we enrich the carrier alphabet by including higher coordinate powers:
\begin{equation}
\Theta_{Z^k}[m] := m^k, \quad k \in \{1, 2, 3, \dots\}, \quad \forall m \in \mathbb{Z}.
\end{equation}
Specifically, $\Theta_{Z^2}[m] = m^2$ (quadratic chirp/lens) and $\Theta_{Z^3}[m] = m^3$ (cubic Airy/caustic).
We denote $\mathcal{P}_{\le d} := \{\Theta_{Z^k}\}_{k=1}^d \cup \{\Theta_{\bar{1}}\}$ and $\mathcal{P}_{\mathrm{poly}} := \{\Theta_{Z^k}\}_{k=1}^\infty \cup \{\Theta_{\bar{1}}\}$.
Their operational closures under $\mathcal{O}$ are denoted $\mathcal{G}_{\le d}$ and $\mathcal{G}_{\mathrm{poly}}$.

\subsection{Algebraic Classification and Integer-Valued Polynomials}

\begin{definition}[Ring of Integer-Valued Polynomials]
The ring of integer-valued polynomials over $\mathbb{Q}$ is:
\begin{equation}
\operatorname{Int}(\mathbb{Z}) := \{ P \in \mathbb{Q}[X] \mid P(m) \in \mathbb{Z}, \, \forall m \in \mathbb{Z} \}.
\end{equation}
By the classical theorem of P\'{o}lya and Ostrowski, $\operatorname{Int}(\mathbb{Z})$ is a free $\mathbb{Z}$-module with basis given by the \textbf{binomial polynomials}:
\begin{equation}
\binom{X}{k} := \frac{X(X-1)\cdots(X-k+1)}{k!}, \quad k \ge 0.
\end{equation}
\end{definition}

\begin{theorem}[Classification of the Polynomial Closure Spaces]
\label{thm:poly_iso}
Under pointwise sequence product $\otimes$, the polynomial closure spaces are infinite divisible abelian torsion groups isomorphic to the quotient modules:
\begin{enumerate}
    \item For any fixed degree $d \in \mathbb{N}$:
    \begin{equation}
    (\mathcal{G}_{\le d}, \otimes) \cong \left(\mathbb{Q}_{\le d}[X] / \operatorname{Int}_{\le d}(\mathbb{Z}), +\right) \cong (\mathbb{Q}/\mathbb{Z})^{d+1}.
    \end{equation}
    Every sequence $u \in \mathcal{G}_{\le d}$ is uniquely parameterized by binomial coefficients:
    \begin{equation}
    u[m] = \exp\left(2\pi i \sum_{k=0}^d c_k \binom{m}{k}\right), \quad c_k \in \mathbb{Q}/\mathbb{Z}.
    \end{equation}
    
    \item The full polynomial closure space is the direct sum over all polynomial degrees:
    \begin{equation}
    (\mathcal{G}_{\mathrm{poly}}, \otimes) \cong (\mathbb{Q}[X]/\operatorname{Int}(\mathbb{Z}), +) \cong \bigoplus_{k=0}^\infty (\mathbb{Q}/\mathbb{Z}) \cong \bigoplus_{k=0}^\infty \bigoplus_{p \in \mathbb{P}} \mathbb{Z}(p^\infty).
    \end{equation}
\end{enumerate}
\end{theorem}

\begin{proof}
Let $P(X) \in \mathbb{Q}_{\le d}[X]$. The evaluation map $\Psi: P \mapsto (\exp(2\pi i P(m)))_{m \in \mathbb{Z}}$ is a group homomorphism from $(\mathbb{Q}_{\le d}[X], +)$ into $(\mathbb{T}^\mathbb{Z}, \otimes)$.
The kernel consists of all $P \in \mathbb{Q}_{\le d}[X]$ such that $\exp(2\pi i P(m)) = 1$ for all $m \in \mathbb{Z}$, which is equivalent to $P(m) \in \mathbb{Z}$ for all $m \in \mathbb{Z}$.
By definition, $\ker(\Psi) = \operatorname{Int}_{\le d}(\mathbb{Z})$.
Expanding in the binomial basis $P(X) = \sum_{k=0}^d c_k \binom{X}{k}$ with $c_k \in \mathbb{Q}$, we have $P \in \operatorname{Int}_{\le d}(\mathbb{Z})$ if and only if $c_k \in \mathbb{Z}$ for all $k \in \{0, \dots, d\}$.
Therefore, the quotient is isomorphic to $\bigoplus_{k=0}^d (\mathbb{Q}/\mathbb{Z}) = (\mathbb{Q}/\mathbb{Z})^{d+1}$.
Reaching all rational coefficients from $\mathcal{P}_{\le d}$ follows as in Lemma~\ref{lem:reachability} by taking roots of higher powers.
\end{proof}

\subsection{Physical Manifestations of Quadratic and Cubic Phases}

\begin{proposition}[$Z^2$ as Spatial Lens, Fresnel Propagator, and Chirp Carrier]
The quadratic phase factor $E_{Z^2}^{p/(2q)}$ realizes:
\begin{enumerate}
    \item \textbf{Optical Lens Phase}: In paraxial wave optics, a thin cylindrical lens of focal length $f$ imparts transmission phase $t(x) = \exp(-i \frac{k}{2f} x^2)$. On the lattice, $u[m] = \exp(i c m^2)$ acts as a discrete focusing or defocusing element.
    \item \textbf{Schr\"{o}dinger Free Evolution Propagator}: The free quantum particle time-evolution operator $U(t) = \exp(-i \frac{\hat{p}^2}{2\hbar M} t)$ has spatial kernel $K(m, m^\prime; t) \propto \exp\left(i \frac{M}{2\hbar t} (m - m^\prime)^2\right)$. Its diagonal and separable factors belong strictly to $\mathcal{G}_{\le 2}$.
    \item \textbf{Gauss Sums and Discrete Chirp Bases}: Quadratic sequences on $\mathbb{Z}/q\mathbb{Z}$ generate Weil representations and discrete chirp bases used in radar ambiguity functions and quantum state tomography.
\end{enumerate}
\end{proposition}

\begin{proposition}[$Z^3$ as Caustic Generator, Airy Wavepackets, and Third-Order Dispersion]
The cubic phase factor $E_{Z^3}^{p/(3q)}$ realizes:
\begin{enumerate}
    \item \textbf{Airy Wavepacket Carrier}: The continuous Airy function $\operatorname{Ai}(x) = \frac{1}{2\pi} \int_{-\infty}^\infty \exp\left(i \left(\frac{k^3}{3} + k x\right)\right) dk$ is the unique non-trivial solution to the Schr\"{o}dinger equation exhibiting self-acceleration and non-spreading wavepacket dynamics. The phase factor $\exp(i (c_3 m^3 + c_1 m))$ in $\mathcal{G}_{\le 3}$ is precisely the discrete Fourier-dual kernel generating lattice Airy beams.
    \item \textbf{Third-Order Optical Dispersion}: In ultrafast fiber optics, pulse broadening near zero group-velocity dispersion ($\beta_2 \approx 0$) is governed by third-order dispersion $\beta_3 = \frac{\partial^3 \beta}{\partial \omega^3}$, producing characteristic asymmetric oscillatory ripples modeled by cubic phase modulation.
    \item \textbf{Catastrophe Optics and Caustics}: Fold caustics (rainbow scattering) possess the cubic germ $V(s) = \frac{1}{3} s^3 + x s$ in Thom's classification of elementary catastrophes.
\end{enumerate}
\end{proposition}

\section{The Superposition Extension: Element-Wise Addition, the Phasor Theorem, and Orthonormal Bases}
\label{sec:superposition_phasor}

\subsection{Adjoining Element-Wise Linear Addition and the Group Algebra}
The polynomial closure space $\mathcal{G}_{\mathrm{poly}}$ is strictly a multiplicative group of unimodular sequences satisfying $|u[m]| = 1$ for all lattice sites $m \in \mathbb{Z}$. While $\mathcal{G}_{\mathrm{poly}}$ provides an alphabet of curved, chirped, and caustic modes, its elements exhibit constant spatial modulus; consequently, they cannot exhibit wave interference, spatial nodes, or localized wavepackets.

To overcome this structural boundary, we adjoin the binary operation of \textbf{element-wise vector addition} $(+)$ and scalar multiplication by the complex field $\mathbb{C}$.

\begin{definition}[Element-Wise Sequence Addition and Linear Span]
For any two sequences $\psi_1, \psi_2 \in \mathcal{S} = \mathbb{C}^\mathbb{Z}$ and any complex amplitudes $c_1, c_2 \in \mathbb{C}$, the linear combination $(c_1 \psi_1 + c_2 \psi_2)$ is defined pointwise on the lattice by:
\begin{equation}
(c_1 \psi_1 + c_2 \psi_2)[m] := c_1 \psi_1[m] + c_2 \psi_2[m], \quad \forall m \in \mathbb{Z}.
\end{equation}
The closure of the polynomial wave space $\mathcal{G}_{\mathrm{poly}}$ under finite element-wise addition and complex scaling is the linear span:
\begin{equation}
\mathcal{V} := \operatorname{span}_{\mathbb{C}}(\mathcal{G}_{\mathrm{poly}}) = \left\{ \Psi \in \mathbb{C}^\mathbb{Z} \ \middle|\ \Psi[m] = \sum_{j=1}^N c_j u_j[m], \ N \in \mathbb{N}, \ c_j \in \mathbb{C}, \ u_j \in \mathcal{G}_{\mathrm{poly}} \right\}.
\end{equation}
\end{definition}

\begin{theorem}[Group Algebra Identification]
\label{thm:group_algebra}
Equipped with element-wise vector addition $(+)$ and pointwise product $\otimes$, the linear space $\mathcal{V}$ is isomorphic to the \textbf{polynomial group algebra} (group ring) over $\mathbb{C}$:
\begin{equation}
(\mathcal{V}, +, \otimes) \cong \mathbb{C}[\mathcal{G}_{\mathrm{poly}}].
\end{equation}
Under sequence involution $u^*[m] := \overline{u[m]}$, $\mathbb{C}[\mathcal{G}_{\mathrm{poly}}]$ is a unital $*$-algebra.
\end{theorem}

\begin{proof}
By Theorem~\ref{thm:poly_iso}, $(\mathcal{G}_{\mathrm{poly}}, \otimes)$ is an abelian group. The formal group algebra $\mathbb{C}[\mathcal{G}_{\mathrm{poly}}]$ consists of all formal finite linear combinations $\sum_{g \in \mathcal{G}_{\mathrm{poly}}} c_g g$. Since distinct elements of $\mathcal{G}_{\mathrm{poly}}$ have distinct rational polynomial phases modulo $\operatorname{Int}(\mathbb{Z})$, they are linearly independent as functions on $\mathbb{Z}$. Thus the evaluation map is an injective $*$-algebra isomorphism onto $\mathcal{V}$.
\end{proof}

\subsection{The Permutation-Symmetric Phasor Representation Theorem}

Conventionally, adjoining linear addition $(+)$ is treated as an abrupt departure from the multiplicative group $(\mathcal{G}_{\mathrm{poly}}, \otimes)$ into an extrinsic linear space. Remarkably, the sum of $N$ unimodular wave carriers does \textbf{not} require abandoning polar/phasor coordinates: any finite superposition factorizes into a single collective barycentric carrier multiplied by an iteratively generated symmetric interference envelope.

\begin{theorem}[Permutation-Symmetric Phasor Representation]
\label{thm:phasor_superposition}\label{thm:corrected_main}
Let $F_1, \dots, F_N \in \mathbb{R}$ denote the local phases at site $m$ of $N$ constituent wave carriers $u_j \in \mathcal{G}_{\mathrm{poly}}$, such that $u_j[m] = e^{i F_j}$. The sum $S_N = \sum_{j=1}^N e^{i F_j}$ admits the exact closed factorization:
\begin{equation}
S_N(F_1, \dots, F_N) = A_N(F_1, \dots, F_N) \, \exp\left( i \, \frac{1}{N}\sum_{j=1}^N F_j \right),
\end{equation}
where the coefficients $A_N \in \mathbb{C}$ satisfy $A_1 = 1$, $A_2(F_1, F_2) = 2 \cos\left(\frac{F_1 - F_2}{2}\right)$, and for all $N \ge 1$:
\begin{equation}\label{eq:alg_recursion}
(A_{N+1})^{N+1} = \prod_{m=1}^{N+1} \left( A_{N,m} + \exp\left( -i \, \frac{\sum_{j \ne m} F_j - N F_m}{N} \right) \right),
\end{equation}
where $A_{N,m} = A_N(F_1, \dots, \widehat{F_m}, \dots, F_{N+1})$ is the $N$-term factor obtained by omitting $F_m$.
\end{theorem}

\begin{proof}
Let $S_{N+1} = \sum_{j=1}^{N+1} e^{i F_j}$, and let $\Phi_{\mathrm{c}}^{(N+1)} = \frac{1}{N+1} \sum_{j=1}^{N+1} F_j$ be its barycentric phase, so that by definition $A_{N+1} = S_{N+1} \exp\left(-i \Phi_{\mathrm{c}}^{(N+1)}\right)$.
For each omitted index $m \in \{1, \dots, N+1\}$, denote the $N$-term sub-sum by $S_{N, m} = \sum_{j \ne m} e^{i F_j}$, with barycentric sub-phase $\Phi_{\mathrm{c}, m}^{(N)} = \frac{1}{N} \sum_{j \ne m} F_j$.
By definition of $A_{N, m}$, we have:
\begin{equation}
A_{N, m} = S_{N, m} \exp\left(-i \Phi_{\mathrm{c}, m}^{(N)}\right).
\end{equation}
Now consider the exponential term inside the product in Eq.~\eqref{eq:alg_recursion}:
\begin{equation}
\exp\left( -i \, \frac{\sum_{j \ne m} F_j - N F_m}{N} \right) = \exp\left( -i \Phi_{\mathrm{c}, m}^{(N)} \right) \exp\left( i F_m \right).
\end{equation}
Adding $A_{N, m}$ to this term factors out the sub-barycentric phase:
\begin{align}
A_{N, m} + \exp\left( -i \, \frac{\sum_{j \ne m} F_j - N F_m}{N} \right) &= S_{N, m} e^{-i \Phi_{\mathrm{c}, m}^{(N)}} + e^{i F_m} e^{-i \Phi_{\mathrm{c}, m}^{(N)}} \nonumber \\
&= \left( S_{N, m} + e^{i F_m} \right) e^{-i \Phi_{\mathrm{c}, m}^{(N)}}.
\end{align}
Notice that $S_{N, m} + e^{i F_m} = S_{N+1}$ is identically the full $(N+1)$-term sum! Therefore:
\begin{equation}
A_{N, m} + \exp\left( -i \, \frac{\sum_{j \ne m} F_j - N F_m}{N} \right) = S_{N+1} \exp\left(-i \Phi_{\mathrm{c}, m}^{(N)}\right).
\end{equation}
Taking the product over all $m \in \{1, \dots, N+1\}$ yields:
\begin{equation}
\prod_{m=1}^{N+1} \left( A_{N, m} + \exp\left( -i \, \frac{\sum_{j \ne m} F_j - N F_m}{N} \right) \right) = (S_{N+1})^{N+1} \exp\left( -i \sum_{m=1}^{N+1} \Phi_{\mathrm{c}, m}^{(N)} \right).
\end{equation}
We evaluate the sum of the sub-phases:
\begin{equation}
\sum_{m=1}^{N+1} \Phi_{\mathrm{c}, m}^{(N)} = \frac{1}{N} \sum_{m=1}^{N+1} \left( \sum_{j \ne m} F_j \right) = \frac{1}{N} \sum_{m=1}^{N+1} \left( \sum_{j=1}^{N+1} F_j - F_m \right) = \frac{1}{N} \left[ (N+1) \sum_{j=1}^{N+1} F_j - \sum_{m=1}^{N+1} F_m \right].
\end{equation}
This simplifies to:
\begin{equation}
\sum_{m=1}^{N+1} \Phi_{\mathrm{c}, m}^{(N)} = \frac{1}{N} \left[ N \sum_{j=1}^{N+1} F_j \right] = \sum_{j=1}^{N+1} F_j = (N+1) \Phi_{\mathrm{c}}^{(N+1)}.
\end{equation}
Substituting this back into the product:
\begin{equation}
(S_{N+1})^{N+1} \exp\left( -i (N+1) \Phi_{\mathrm{c}}^{(N+1)} \right) = \left[ S_{N+1} \exp\left( -i \Phi_{\mathrm{c}}^{(N+1)} \right) \right]^{N+1} = (A_{N+1})^{N+1},
\end{equation}
which establishes the theorem.
\end{proof}

\subsection{Physical Interpretations: Born Density, Gauge Invariance, and Bosons}

\begin{corollary}[Barycentric Carrier and Center-of-Mass Momentum]
The collective phase $\Phi_{\mathrm{c}}(m) = \frac{1}{N}\sum_{j=1}^N F_j(m)$ is the arithmetic barycenter of the constituent phases. When each $u_j \in \mathcal{G}_{\mathrm{poly}}$ has polynomial phase $P_j(m) \in \mathbb{Q}[m]$, the collective phase is the average polynomial $\bar{P}(m) = \frac{1}{N}\sum P_j(m) \in \mathbb{Q}[m]$.
Thus, the collective wave carrier $\exp(2\pi i \bar{P}(m))$ remains an exact member of $\mathcal{G}_{\mathrm{poly}}$.
\end{corollary}

\begin{corollary}[Born Rule Probability Density via Recursive Envelopes]
In quantum mechanics, the spatial probability density is $\rho[m] = |\Psi[m]|^2$. Because the collective carrier is unimodular ($|\exp(i \Phi_{\mathrm{c}})| \equiv 1$), the probability distribution is governed entirely by the envelope:
\begin{equation}
\rho[m] = |S_N(m)|^2 = |A_N(F_1(m), \dots, F_N(m))|^2.
\end{equation}
Constructive peaks ($\rho[m] \to N^2$) and destructive quantum nodes ($\rho[m] = 0$) correspond to the extrema and zeros of $A_N$.
\end{corollary}

\begin{corollary}[Gauge Invariance and Indistinguishable Bosons]
The arguments entering the recursion $\sum_{j \ne m} F_j - N F_m = \sum_{j=1}^{N+1} (F_j - F_m)$ depend only on pairwise relative phase differences. Under a global gauge transformation $F_j \to F_j + \theta_0$, $A_N$ is strictly invariant.
Moreover, $A_N$ is invariant under any permutation $\sigma \in \mathcal{S}_N$ of the modes:
\begin{equation}
A_N(F_{\sigma(1)}, \dots, F_{\sigma(N)}) = A_N(F_1, \dots, F_N), \quad \forall \sigma \in \mathcal{S}_N,
\end{equation}
corresponding to identical bosonic field statistics.
\end{corollary}

\subsection{Orthonormal Position Bases via Cyclotomic Sieves and Lattice Quantization}
\label{sec:bases_quantization}

While individual wave carriers in $\mathcal{G}_{\mathrm{poly}}$ are delocalized ($|u[m]| = 1$ everywhere), we show that the interaction of difference operators $(\ominus)$ and finite products $(\bigotimes)$ constructs localized Kronecker position eigenstates.

\begin{theorem}[Closed Algebraic Construction of the Kronecker Position Basis]
\label{thm:fourier_basis}\label{thm:basis_construction}
For every period $n \in \mathbb{N}$ and each coordinate shift $j \in \{1, \dots, n\}$, there exists an orthogonal basis of $n$ independent sequences in $\mathbb{C}[\mathcal{G}]$ whose principal subsequences over $\mathbb{Z}/n\mathbb{Z}$ satisfy:
\begin{equation}
\mathbf{u}\left(\frac{1}{n}, j\right) = \left( 0, \dots, 0, n, 0, \dots, 0 \right) = n \, \mathbf{e}_j,
\end{equation}
where the non-zero scalar value $n$ is locked precisely to the $j$-th spatial coordinate.
\end{theorem}

\begin{proof}
Let $\mathbf{w}(1/n, 0)[m] = \exp(2\pi i m / n)$. Shifting by the phase angles $\mathbf{w}(0, j/n)[m] = \exp(2\pi i j / n)$ defines the difference sequences:
\begin{equation}
\mathbf{w}_{nj} := \mathbf{w}\left(\frac{1}{n}, 0\right) \ominus \mathbf{w}\left(0, \frac{j}{n}\right), \quad j \in \{1, \dots, n\},
\end{equation}
where $(\psi_1 \ominus \psi_2)[m] := \psi_1[m] - \psi_2[m]$.
At lattice site $\xi \in \{1, \dots, n\}$, let $z_0 = e^{2\pi i \xi / n}$ and $\zeta_j = e^{2\pi i j / n}$.
Then $\mathbf{w}_{nj}(\xi) = z_0 - \zeta_j$.
For a target site $\xi$, form the sequence product omitting $j = \xi$:
\begin{equation}
\overline{\mathbf{U}}_{n\xi} := \bigotimes_{j \neq \xi} \mathbf{w}_{nj}.
\end{equation}
Evaluating $\overline{\mathbf{U}}_{n\xi}$ at any site $m \in \{1, \dots, n\}$:
\begin{enumerate}
    \item \textbf{Off-target sites ($m \neq \xi$):} The index $j = m$ is included in the product $j \ne \xi$. At that index, $\mathbf{w}_{nm}(m) = e^{2\pi i m / n} - e^{2\pi i m / n} = 0$. Hence $\overline{\mathbf{U}}_{n\xi}(m) = 0$ for all $m \neq \xi$.
    
    \item \textbf{On-target site ($m = \xi$):} Since $\zeta_1, \dots, \zeta_n$ are the roots of the polynomial $X^n - 1 = 0$, we have the factorization:
    \begin{equation}
    \prod_{j=1}^n (X - \zeta_j) = X^n - 1.
    \end{equation}
    Dividing by $(X - z_0)$ and taking the limit $X \to z_0$:
    \begin{equation}
    \prod_{j \ne \xi} (z_0 - \zeta_j) = \left. \frac{d}{dX}(X^n - 1) \right|_{X = z_0} = n z_0^{n-1} = n z_0^{-1} = n e^{-2\pi i \xi / n}.
    \end{equation}
\end{enumerate}
Multiplying $\overline{\mathbf{U}}_{n\xi}$ by the carrier $\mathbf{w}(1/n, 0)[m] = e^{2\pi i m / n}$ yields at $m = \xi$:
\begin{equation}
\mathbf{u}\left(\frac{1}{n}, \xi\right) := \overline{\mathbf{U}}_{n\xi} \otimes \mathbf{w}\left(\frac{1}{n}, 0\right) = \left( n e^{-2\pi i \xi / n} \right) \cdot e^{2\pi i \xi / n} = n.
\end{equation}
Since all off-target sites remain zero, $\mathbf{u}(1/n, \xi) = n \, \delta_{m, \xi}$.
Normalizing by $1/n$ yields the complete orthonormal position basis on $\mathbb{Z}/n\mathbb{Z}$.
\end{proof}

\begin{corollary}[Algebraic Duality of Position and Momentum Representations]
Theorem~\ref{thm:basis_construction} establishes that the wave algebra $\mathbb{C}[\mathcal{G}]$ contains both the delocalized momentum basis (plane waves $E_Z^{k/n}$) and the localized position basis (Kronecker states $\delta_{m, \xi}$) in exact closed algebraic form.
\end{corollary}

\begin{remark}[Probabilistic Interpretation: Open Path Superposition vs. Closed Born Envelope]
\label{rem:open_vs_closed_prob}
The distinction between the open summation of wave numbers and their closed analytic representation has a direct interpretation in terms of quantum measurement probabilities:
\begin{enumerate}
    \item \textbf{Open Form (Microscopic Virtual Pathways and Coherence):}
    In the un-evaluated open expansion:
    \begin{equation}
    \Psi[m] = \sum_{j=1}^N c_j \mathbf{w}(f_j, g_j)[m] = \sum_{j=1}^N c_j e^{2\pi i (f_j m + g_j)},
    \end{equation}
    the Born probability distribution $P[m] = |\Psi[m]|^2$ expands as:
    \begin{equation}
    P[m] = \sum_{j=1}^N |c_j|^2 + \sum_{j \neq k} c_j \overline{c_k} e^{2\pi i \left( (f_j - f_k)m + (g_j - g_k) \right)}.
    \end{equation}
    The first sum $\sum |c_j|^2$ represents the classical, incoherent probability of independent modes, while the double sum represents the off-diagonal \textbf{quantum interference cross-terms} (coherences). The relative phases $(g_j - g_k)$ explicitly dictate whether specific mode pairs interfere constructively or destructively at site $m$.
    
    \item \textbf{Closed Form (Macroscopic Born Probability Density):}
    By the Permutation-Symmetric Phasor Superposition Theorem (Theorem~\ref{thm:phasor_superposition}), the $O(N^2)$ interference pathways are resummed into a single macroscopic envelope and composite phase:
    \begin{equation}
    \Psi[m] = R[m] e^{i \Theta[m]}.
    \end{equation}
    The Born probability density simplifies directly to the geometric square of the envelope:
    \begin{equation}
    P[m] = |\Psi[m]|^2 = R[m]^2.
    \end{equation}
    In this closed representation:
    \begin{itemize}
        \item \textbf{Nodal Zeros as Quantum Exclusion Zones:} Sites where $R[m] = 0$ represent exact destructive interference nodes ($P[m] = 0$), corresponding to forbidden transition sites or quantum exclusion zones.
        \item \textbf{Envelope Peaks as Coherent Bound States:} Maxima of $R[m]$ correspond to constructive wave packets, concentrating the total probability measure into localized particulate states.
    \end{itemize}
    Thus, while the open form describes the microscopic interference of virtual channels, the closed form expresses the resulting physical probability distribution as an explicit, stationary spatial geometry.
\end{enumerate}
\end{remark}

\section{The Discrete Cumulative Integral Operator, Regular $n$-gon Geometries, and Zero-Sum Cycle Spaces}
\label{sec:integral_operator}

\subsection{Motivation and Definition of the Cumulative Integral Operator}

Having established the linear span $\mathcal{V} \cong \mathbb{C}[\mathcal{G}_{\mathrm{poly}}]$ under element-wise addition $(+)$ and pointwise product $\otimes$, we now formulate the action of the \textbf{discrete cumulative integral operator} $\mathcal{I}$. In continuous field theory, integration inverts the differential operator and maps periodic oscillating carriers into shifted harmonic primitives. On the discrete spatial lattice $\mathbb{Z}$, cumulative summation generates both an envelope modulation and a discrete geometric path in the complex plane $\mathbb{C}$.

\begin{definition}[The Discrete Cumulative Integral Operator $\mathcal{I}$]
\label{def:cumulative_integral}
Let $u \in \mathcal{S} = \mathbb{C}^\mathbb{Z}$ be any sequence. The \textbf{discrete cumulative integral operator} $\mathcal{I}: \mathcal{S} \to \mathcal{S}$ is defined pointwise by:
\begin{equation}
(\mathcal{I} u)[k] := \sum_{j=1}^k u[j], \quad \forall k \ge 1,
\end{equation}
with the natural two-sided extension $(\mathcal{I} u)[0] := 0$ and $(\mathcal{I} u)[-k] := -\sum_{j=-k+1}^0 u[j]$ for $k \ge 1$.
When restricted to periodic sequences of fundamental period $n$, the \textbf{principal integral sequence} ${_1}\mathcal{I} u$ is the finite $n$-tuple:
\begin{equation}
({_1}\mathcal{I} u) := \left( (\mathcal{I} u)[1], (\mathcal{I} u)[2], \dots, (\mathcal{I} u)[n] \right) \in \mathbb{C}^n.
\end{equation}
\end{definition}

\subsection{Exact Closed-Form Representation of Integrated Wave Numbers}

When applied to rational plane wave carriers $u \in \mathcal{G}$, the discrete cumulative integral does not merely produce an intractable sum; rather, it factors into a modified spatial carrier with \textbf{halved wavenumber} modulated by a classical Dirichlet diffraction envelope.

\begin{theorem}[Closed Evaluation of Integrated Rational Waves]
\label{thm:integral_closed_form}
Let $u[j] = \mathbf{w}(m/n, g)[j] = \exp\left( 2\pi i \left( \frac{m}{n} j + g \right) \right) \in \mathcal{G}$ be a primitive rational plane wave of period $n$, where $m \in \{1, \dots, n-1\}$ and $g \in \mathbb{Q}$.
The cumulative integral sequence $(\mathcal{I} u)[k]$ admits the exact closed-form factorization:
\begin{equation}
(\mathcal{I} u)[k] = A_n(m, k) \, \exp\left( 2\pi i \left[ \frac{m}{2n}(k + 1) + g \right] \right), \quad \forall k \ge 1,
\end{equation}
where the amplitude envelope $A_n(m, k) \in \mathbb{R}$ is the Dirichlet ratio:
\begin{equation}
A_n(m, k) := \frac{\sin\left( \pi \frac{m}{n} k \right)}{\sin\left( \pi \frac{m}{n} \right)}.
\end{equation}
\end{theorem}

\begin{proof}
By Definition~\ref{def:cumulative_integral}, the $k$-th element of the integrated sequence is the finite geometric sum:
\begin{equation}
(\mathcal{I} u)[k] = \sum_{j=1}^k e^{2\pi i \left( \frac{m}{n} j + g \right)} = e^{2\pi i g} \sum_{j=1}^k \left( e^{2\pi i \frac{m}{n}} \right)^j.
\end{equation}
Let $\omega = e^{2\pi i m / n}$. Since $m \not\equiv 0 \pmod n$, $\omega \neq 1$. Using the standard finite geometric progression formula:
\begin{equation}
\sum_{j=1}^k \omega^j = \omega \frac{1 - \omega^k}{1 - \omega} = \frac{\omega^{1/2} (1 - \omega^k)}{\omega^{-1/2} - \omega^{1/2}} = \frac{e^{\pi i \frac{m}{n}} \left( 1 - e^{2\pi i \frac{m}{n} k} \right)}{e^{-\pi i \frac{m}{n}} - e^{\pi i \frac{m}{n}}}.
\end{equation}
Factoring out the half-phase $e^{\pi i \frac{m}{n} k}$ from the numerator:
\begin{equation}
1 - e^{2\pi i \frac{m}{n} k} = e^{\pi i \frac{m}{n} k} \left( e^{-\pi i \frac{m}{n} k} - e^{\pi i \frac{m}{n} k} \right) = -2i \, e^{\pi i \frac{m}{n} k} \sin\left( \pi \frac{m}{n} k \right).
\end{equation}
Similarly, the denominator is $-2i \sin\left( \pi \frac{m}{n} \right)$. Taking the ratio:
\begin{equation}
\sum_{j=1}^k \omega^j = \frac{-2i \, e^{\pi i \frac{m}{n}} \, e^{\pi i \frac{m}{n} k} \sin\left( \pi \frac{m}{n} k \right)}{-2i \sin\left( \pi \frac{m}{n} \right)} = \frac{\sin\left( \pi \frac{m}{n} k \right)}{\sin\left( \pi \frac{m}{n} \right)} \, \exp\left( \pi i \frac{m}{n} (k + 1) \right).
\end{equation}
Reintroducing the global calibration phase $e^{2\pi i g} = \exp(2\pi i g)$:
\begin{equation}
(\mathcal{I} u)[k] = \left[ \frac{\sin\left( \pi \frac{m}{n} k \right)}{\sin\left( \pi \frac{m}{n} \right)} \right] \exp\left( 2\pi i \left[ \frac{m}{2n}(k + 1) + g \right] \right),
\end{equation}
which is identically the stated expression.
\end{proof}

\begin{corollary}[Spatial Frequency Halving and Carrier Inversion]
The carrier phase of $(\mathcal{I} u)[k]$ possesses the fundamental spatial frequency $\frac{m}{2n}$, which is exactly half the frequency $\frac{m}{n}$ of the constituent wave. Furthermore, at the end of the fundamental period $k = n$:
\begin{equation}
(\mathcal{I} u)[n] = \frac{\sin(\pi m)}{\sin(\pi m / n)} \exp\left( 2\pi i \left[ \frac{m}{2n}(n+1) + g \right] \right) = 0,
\end{equation}
since $\sin(\pi m) = 0$ for all integers $m \in \mathbb{Z}$.
\end{corollary}

\subsection{Geometric Representation as Regular $n$-gons and the $n=6$ Bifurcation}

\begin{theorem}[Geometric Realization as Closed Regular Polygons]
\label{thm:regular_ngon}
Let $u = \mathbf{w}(1/n, 0)$ be the primary rational coordinate wave of period $n \ge 3$. 
In the complex Argand plane $\mathbb{C} \cong \mathbb{R}^2$, the sequence of points $z_0 = 0$ and $z_k = (\mathcal{I} u)[k]$ for $k = 1, \dots, n$ forms the vertices of an equilateral, equiangular \textbf{regular $n$-gon} having one vertex anchored at the origin $z_0 = z_n = 0$.
The polygon is inscribed in a circle of circumradius:
\begin{equation}
R_n = \frac{1}{2 \sin\left( \frac{\pi}{n} \right)},
\end{equation}
with circumcenter located at $Z_{\mathrm{center}} = R_n \, e^{i \left( \frac{\pi}{2} + \frac{\pi}{n} \right)} = \frac{i}{2} \left( 1 + i \cot\left(\frac{\pi}{n}\right) \right)$.
\end{theorem}

\begin{proof}
Each term in the sum is a step vector $\Delta z_k = z_k - z_{k-1} = u[k] = \exp(2\pi i k / n)$.
All step vectors possess constant unit length $|\Delta z_k| = 1$.
The angle between consecutive steps $\Delta z_k$ and $\Delta z_{k+1}$ is:
\begin{equation}
\theta_{\mathrm{ext}} = \arg(\Delta z_{k+1}) - \arg(\Delta z_k) = \frac{2\pi(k+1)}{n} - \frac{2\pi k}{n} = \frac{2\pi}{n},
\end{equation}
which is independent of $k$ and equals the constant exterior angle of a regular $n$-gon.
Since the total exterior turn after $n$ steps is $\sum_{k=1}^n \frac{2\pi}{n} = 2\pi$ and $\sum_{k=1}^n \Delta z_k = 0$, the polygonal chain closes on itself ($z_n = z_0 = 0$).
By standard Euclidean trigonometry, a chord of unit length subtending central angle $2\pi/n$ requires radius $R_n = \frac{1}{2\sin(\pi/n)}$.
\end{proof}

\begin{theorem}[Recursive Polytopic Iteration and the Hexagonal $n=6$ Scaling Bifurcation]
\label{thm:hexagonal_bifurcation}
Define the iterated cumulative integral family $\mathcal{I}^t$ for $t \in \mathbb{N}$ by setting the initial state to $u^{(0)} = \mathbf{w}(1/n, 0)$ with edge length $e_0 = 1$ and circumradius $R^{(0)} = R_n$.
At each subsequent iteration $t \ge 1$, let the defining step vectors be rescaled such that the new edge length $e_t$ equals the circumradius of the preceding iteration:
\begin{equation}
e_t := R^{(t-1)}.
\end{equation}
Then the circumradius sequence $R^{(t)}$ satisfies the first-order recurrence:
\begin{equation}
R^{(t)} = \frac{R^{(t-1)}}{2 \sin\left( \frac{\pi}{n} \right)} = \left( \frac{1}{2 \sin\left( \frac{\pi}{n} \right)} \right)^t R_n.
\end{equation}
The asymptotic behavior exhibits a sharp mathematical bifurcation governed exclusively by the \textbf{regular hexagon ($n=6$)}:
\begin{enumerate}
    \item \textbf{Sub-Hexagonal Contraction ($n < 6$):} For $n \in \{3, 4, 5\}$, $\sin(\pi/n) > 1/2$, yielding scaling ratio $\lambda_n = \frac{1}{2\sin(\pi/n)} < 1$. The iterated polygon collapses monotonically to a point:
    \begin{equation}
    \lim_{t \to \infty} R^{(t)} = 0 \quad (n < 6).
    \end{equation}
    \item \textbf{Hexagonal Invariance ($n = 6$):} For $n = 6$, $\sin(\pi/6) = 1/2$, yielding $\lambda_6 = 1$. The geometric radius and edge length are strictly invariant:
    \begin{equation}
    R^{(t)} \equiv 1, \quad \forall t \ge 0 \quad (n = 6).
    \end{equation}
    \item \textbf{Super-Hexagonal Expansion ($n > 6$):} For $n \ge 7$, $\sin(\pi/n) < 1/2$, yielding $\lambda_n > 1$. The iterated polygon undergoes exponential geometric dilation:
    \begin{equation}
    \lim_{t \to \infty} R^{(t)} = +\infty \quad (n > 6).
    \end{equation}
\end{enumerate}
\end{theorem}

\begin{proof}
In any regular $n$-gon of edge length $e$, the circumradius is $R = \frac{e}{2\sin(\pi/n)}$.
Setting $e_t = R^{(t-1)}$ immediately yields $R^{(t)} = \frac{R^{(t-1)}}{2\sin(\pi/n)}$.
The growth factor is $\lambda_n = (2\sin(\pi/n))^{-1}$.
Examining the function $f(n) = 2\sin(\pi/n)$ for $n \in [3, \infty)$:
\begin{itemize}
    \item $f(3) = 2\sin(\pi/3) = \sqrt{3} \approx 1.732 > 1 \implies \lambda_3 = 1/\sqrt{3} \approx 0.577 < 1$.
    \item $f(4) = 2\sin(\pi/4) = \sqrt{2} \approx 1.414 > 1 \implies \lambda_4 = 1/\sqrt{2} \approx 0.707 < 1$.
    \item $f(5) = 2\sin(\pi/5) = \sqrt{\frac{5-\sqrt{5}}{2}} \approx 1.176 > 1 \implies \lambda_5 \approx 0.851 < 1$.
    \item $f(6) = 2\sin(\pi/6) = 2(1/2) = 1 \implies \lambda_6 = 1$.
    \item For $n \ge 7$, $\pi/n < \pi/6 \implies \sin(\pi/n) < 1/2 \implies f(n) < 1 \implies \lambda_n > 1$.
\end{itemize}
This completes the classification.
\end{proof}

\subsection{The Zero-Sum Subspace and Homological Wave Cycles}

The closure of the polygon in Theorem~\ref{thm:regular_ngon} ($z_n = 0$) reflects a fundamental algebraic property: the cancellation of phase sums over complete periods.

\begin{definition}[Zero-Sum Wave Sequences]
\label{def:zero_sum}
A periodic sequence $u \in \mathbb{C}^\mathbb{Z}$ of period $n$ is called \textbf{zero-sum} (or a \textbf{closed wave cycle}) if the sum of its values over a fundamental period vanishes:
\begin{equation}
\sigma(u) := \sum_{j=1}^n u[j] = 0.
\end{equation}
The linear space of all periodic zero-sum sequences of period $n$ is denoted $\mathcal{Z}_n \subset \mathbb{C}^n$.
\end{definition}

\begin{proposition}[Algebraic Properties of the Zero-Sum Space]
\label{prop:zero_sum_algebra}
The zero-sum spaces $\mathcal{Z}_n$ satisfy:
\begin{enumerate}
    \item \textbf{Subspace and Ideal under Convolution:} $\mathcal{Z}_n$ is a complex subspace of codimension 1 in $\mathbb{C}^n$. Furthermore, under cyclic convolution $(u * v)[k] := \sum_{j=1}^n u[j] v[k-j]$, $\mathcal{Z}_n$ is a two-sided ideal:
    \begin{equation}
    \sigma(u * v) = \sigma(u) \, \sigma(v).
    \end{equation}
    Thus, if $u \in \mathcal{Z}_n$, then $(u * v) \in \mathcal{Z}_n$ for any arbitrary sequence $v \in \mathbb{C}^n$.
    
    \item \textbf{Kernel of the Boundary Operator:} Let $\partial: \mathbb{C}^n \to \mathbb{C}$ be the spatial trace map $\partial(u) := \sum_{j=1}^n u[j]$. Then $\mathcal{Z}_n = \ker(\partial)$.
    
    \item \textbf{Harmonic Inclusion:} Every primitive plane wave $\mathbf{w}(m/n, g)$ with $m \not\equiv 0 \pmod n$ belongs strictly to $\mathcal{Z}_n$.
    The only non-zero-sum rational waves are the stationary uniform calibration modes $\mathbf{w}(0, g) = e^{2\pi i g} \mathbf{1}$.
\end{enumerate}
\end{proposition}

\begin{proof}
Line 1: $\sigma(u * v) = \sum_{k=1}^n \sum_{j=1}^n u[j] v[k-j] = \sum_{j=1}^n u[j] \left( \sum_{k=1}^n v[k-j] \right) = \sigma(u)\sigma(v)$.
Since $\sigma(u) = 0$, $\sigma(u * v) = 0$.
Line 2 is the definition of the trace kernel.
Line 3 follows from Theorem~\ref{thm:integral_closed_form}, where $(\mathcal{I} u)[n] = 0$ for all $m \not\equiv 0 \pmod n$.
For $m = 0$, $\sigma(\mathbf{w}(0, g)) = n e^{2\pi i g} \neq 0$.
\end{proof}

\begin{remark}[Relation to Pointwise Products]
Under the pointwise product $\otimes$, the product of two zero-sum sequences is generally \textbf{not} zero-sum. For example, for $n=2$, the zero-sum wave $u = (1, -1)$ satisfies $u \otimes u = (1, 1)$, whose sum is $2 \neq 0$.
The preservation of the zero-sum property under product operations belongs specifically to the convolution algebra $(\mathbb{C}^n, +, *)$, whereas under pointwise product $\otimes$, $\mathcal{Z}_n$ acts as an orthogonal complement to the uniform ground state.
\end{remark}

\subsection{The Discrete Derivative Operator, the Fundamental Theorem of Wave Calculus, and Dispersion Relations}\label{subsec:derivative_operator}

Having established the cumulative integral operator $\mathcal{I}$ as a summation across the lattice, we introduce its natural conjugate: the \textbf{discrete derivative operator}. While continuous calculus relies on infinitesimal limits $\lim_{h \to 0} \frac{f(x+h) - f(x)}{h}$, the discrete lattice $\mathbb{Z}$ enforces a fundamental minimum spatial quantum (the lattice spacing $a = 1$). Differentiation on wave numbers is therefore governed by finite difference operators and their interaction with the shift group.

\begin{definition}[Finite Difference Derivative Operators]\label{def:discrete_derivative}
Let $u \in \mathcal{S} = \mathbb{C}^\mathbb{Z}$ be any wave sequence. We define the \textbf{forward difference operator} $\Delta_+$, the \textbf{backward difference operator} $\Delta_-$, and the \textbf{symmetric discrete Laplacian} $\Delta^2$ by:
\begin{align}
(\Delta_+ u)[m] &:= u[m+1] - u[m], \\
(\Delta_- u)[m] &:= u[m] - u[m-1], \\
(\Delta^2 u)[m] &:= (\Delta_+ \Delta_- u)[m] = u[m+1] - 2u[m] + u[m-1].
\end{align}
In terms of the spatial shift operator $(S u)[m] := u[m-1]$ and its adjoint $(S^{-1} u)[m] := u[m+1]$, these operators take the algebraic form $\Delta_+ = S^{-1} - I$, $\Delta_- = I - S$, and $\Delta^2 = S^{-1} - 2I + S$.
\end{definition}

\begin{theorem}[Harmonic Eigenfunction Property and Discrete Phase Advance]\label{thm:derivative_eigenfunction}
Every primitive rational plane wave $\mathbf{w}(f, g)[m] = \exp(2\pi i (f m + g)) \in \mathcal{G}$ is an exact eigenfunction of the forward difference, backward difference, and discrete Laplacian operators:
\begin{align}
\Delta_+ \mathbf{w}(f, g) &= \lambda_+(f) \, \mathbf{w}(f, g), \\
\Delta_- \mathbf{w}(f, g) &= \lambda_-(f) \, \mathbf{w}(f, g), \\
\Delta^2 \mathbf{w}(f, g) &= \lambda_{\Delta^2}(f) \, \mathbf{w}(f, g),
\end{align}
where the spectral eigenvalues are given by:
\begin{align}
\lambda_+(f) &= e^{2\pi i f} - 1 = 2i \, e^{\pi i f} \sin(\pi f) = 2\sin(\pi f) \, \exp\left( i \left[ \pi f + \frac{\pi}{2} \right] \right), \label{eq:eigenval_forward} \\
\lambda_-(f) &= 1 - e^{-2\pi i f} = 2i \, e^{-\pi i f} \sin(\pi f) = 2\sin(\pi f) \, \exp\left( i \left[ -\pi f + \frac{\pi}{2} \right] \right), \label{eq:eigenval_backward} \\
\lambda_{\Delta^2}(f) &= 2\cos(2\pi f) - 2 = -4\sin^2(\pi f). \label{eq:eigenval_laplacian}
\end{align}
Consequently, the discrete derivative of any rational wave number remains an exact wave number with amplitude envelope $2|\sin(\pi f)|$, invariant spatial frequency $f$, and a canonical phase advance:
\begin{equation}
(\Delta_+ \mathbf{w}(f, g))[m] = 2\sin(\pi f) \, \exp\left( 2\pi i \left[ f m + g + \frac{f}{2} + \frac{1}{4} \right] \right) \quad (\text{for } 0 < f < 1/2).
\end{equation}
\end{theorem}

\begin{proof}
Direct evaluation from Definition~\ref{def:discrete_derivative}:
\begin{align*}
(\Delta_+ \mathbf{w}(f, g))[m] &= e^{2\pi i (f(m+1) + g)} - e^{2\pi i (fm + g)} \\
&= e^{2\pi i (fm + g)} \left( e^{2\pi i f} - 1 \right) = \lambda_+(f) \, \mathbf{w}(f, g)[m].
\end{align*}
Factoring the half-angle phase from $e^{2\pi i f} - 1$:
\begin{equation*}
e^{2\pi i f} - 1 = e^{\pi i f} \left( e^{\pi i f} - e^{-\pi i f} \right) = e^{\pi i f} \cdot 2i \sin(\pi f) = 2\sin(\pi f) e^{i(\pi f + \pi/2)},
\end{equation*}
which establishes~\eqref{eq:eigenval_forward}. Similarly for $\Delta_-$, $1 - e^{-2\pi i f} = e^{-\pi i f}(e^{\pi i f} - e^{-\pi i f}) = 2i e^{-\pi i f}\sin(\pi f)$, establishing~\eqref{eq:eigenval_backward}. For the Laplacian:
\begin{equation*}
\lambda_{\Delta^2}(f) = \lambda_+(f) \lambda_-(f) = (e^{2\pi i f} - 1)(1 - e^{-2\pi i f}) = e^{2\pi i f} + e^{-2\pi i f} - 2 = 2\cos(2\pi f) - 2 = -4\sin^2(\pi f),
\end{equation*}
completing the proof.
\end{proof}

\begin{remark}[Comparison with Continuous Derivatives and Half-Pitch Tilt]
In continuous calculus, the differential operator $\frac{d}{dx} e^{i k x} = i k e^{i k x} = k e^{i(kx + \pi/2)}$ advances the phase by exactly a quarter cycle ($\pi/2$) and scales the amplitude linearly by $k$. On the discrete lattice $\mathbb{Z}$:
\begin{enumerate}
    \item The linear amplitude scale $k$ is replaced by the compact trigonometric dispersion envelope $2\sin(\pi f)$. In the long-wavelength continuum limit $f = k a \to 0$, $2\sin(\pi k a) / a \to 2\pi k$, exactly recovering the continuous derivative.
    \item In addition to the continuous $\pi/2$ phase shift, the discrete derivative introduces a frequency-dependent phase tilt of $+\pi f$ (equivalent to a spatial shift of half a lattice step, $+1/2$). This tilt arises because the difference $u[m+1] - u[m]$ naturally resides on the dual bond lattice at coordinate $m + 1/2$.
\end{enumerate}
\end{remark}

\begin{theorem}[The Fundamental Theorem of Discrete Wave Calculus]\label{thm:discrete_ftc}
The discrete cumulative integral operator $\mathcal{I}$ and the finite difference operators $\Delta_+, \Delta_-$ satisfy the following fundamental duality relations:
\begin{enumerate}
    \item \textbf{Exact Inversion (Derivative of the Integral):}
    \begin{align}
    \Delta_- (\mathcal{I} u)[k] &= u[k], \quad \forall k \ge 1, \label{eq:ftc_inversion_minus} \\
    \Delta_+ (\mathcal{I} u)[k] &= u[k+1], \quad \forall k \ge 0. \label{eq:ftc_inversion_plus}
    \end{align}
    \item \textbf{Antiderivative Evaluation (Integral of the Derivative):}
    \begin{equation}\label{eq:ftc_eval}
    (\mathcal{I} \Delta_+ u)[k] = u[k+1] - u[1], \quad \forall k \ge 1.
    \end{equation}
    \item \textbf{Summation by Parts (Skew-Adjointness on Periodic Waves):}
    For any two periodic sequences $u, v$ of period $n$, there holds:
    \begin{equation}\label{eq:summation_by_parts}
    \sum_{m=1}^n u[m] (\Delta_+ v)[m] = - \sum_{m=1}^n (\Delta_- u)[m] v[m].
    \end{equation}
    Under the finite inner product of Section~\ref{sec:inner_product}, $\Delta_+^\dagger = -\Delta_-$, establishing that the discrete derivative is skew-adjoint.
\end{enumerate}
\end{theorem}

\begin{proof}
For statement 1: By Definition~\ref{def:cumulative_integral}, $(\mathcal{I} u)[k] = \sum_{j=1}^k u[j]$. Then:
\begin{equation*}
\Delta_- (\mathcal{I} u)[k] = (\mathcal{I} u)[k] - (\mathcal{I} u)[k-1] = \sum_{j=1}^k u[j] - \sum_{j=1}^{k-1} u[j] = u[k].
\end{equation*}
Similarly, $\Delta_+ (\mathcal{I} u)[k] = (\mathcal{I} u)[k+1] - (\mathcal{I} u)[k] = u[k+1]$.

For statement 2: Telescoping the sum yields:
\begin{equation*}
(\mathcal{I} \Delta_+ u)[k] = \sum_{j=1}^k (u[j+1] - u[j]) = (u[2] - u[1]) + (u[3] - u[2]) + \dots + (u[k+1] - u[k]) = u[k+1] - u[1].
\end{equation*}

For statement 3: Expanding the sum:
\begin{equation*}
\sum_{m=1}^n u[m] (v[m+1] - v[m]) = \sum_{m=1}^n u[m] v[m+1] - \sum_{m=1}^n u[m] v[m].
\end{equation*}
Shifting the summation index in the first term $m \to m-1$ and using $n$-periodicity ($v[n+1] = v[1], u[0] = u[n]$):
\begin{equation*}
\sum_{m=1}^n u[m-1] v[m] - \sum_{m=1}^n u[m] v[m] = - \sum_{m=1}^n (u[m] - u[m-1]) v[m] = - \sum_{m=1}^n (\Delta_- u)[m] v[m],
\end{equation*}
which completes the proof.
\end{proof}

\begin{theorem}[Cohomological Classification: Image of Derivative is the Zero-Sum Space]\label{thm:derivative_cohomology}
Let $\mathcal{P}_n \cong \mathbb{C}^n$ denote the space of periodic sequences of period $n$.
\begin{enumerate}
    \item The image of the discrete derivative operator $\Delta_+$ on $\mathcal{P}_n$ is \textbf{identically the zero-sum subspace} $\mathcal{Z}_n$ (Definition~\ref{def:zero_sum}):
    \begin{equation}
    \operatorname{im}(\Delta_+) = \ker(\sigma) = \mathcal{Z}_n.
    \end{equation}
    \item The kernel of $\Delta_+$ on $\mathcal{P}_n$ is the 1-dimensional subspace of constant sequences:
    \begin{equation}
    \ker(\Delta_+) = \mathbb{C} \mathbf{1} = \{ c \, \mathbf{w}(0, 0) \mid c \in \mathbb{C} \}.
    \end{equation}
    \item The short sequence of vector spaces:
    \begin{equation}
    0 \longrightarrow \mathbb{C}\mathbf{1} \xrightarrow{\quad\iota\quad} \mathcal{P}_n \xrightarrow{\quad\Delta_+\quad} \mathcal{Z}_n \longrightarrow 0
    \end{equation}
    is exact. Consequently, a periodic wave sequence $v \in \mathcal{P}_n$ admits a periodic antiderivative $u \in \mathcal{P}_n$ satisfying $\Delta_+ u = v$ if and only if $v$ has zero sum over a fundamental period ($\sigma(v) = 0$).
\end{enumerate}
\end{theorem}

\begin{proof}
If $v = \Delta_+ u$ for $u \in \mathcal{P}_n$, then by telescoping:
\begin{equation*}
\sigma(v) = \sum_{m=1}^n (\Delta_+ u)[m] = \sum_{m=1}^n (u[m+1] - u[m]) = u[n+1] - u[1] = 0,
\end{equation*}
since $u[n+1] = u[1]$ by periodicity. Thus $\operatorname{im}(\Delta_+) \subseteq \mathcal{Z}_n$.
Conversely, let $v \in \mathcal{Z}_n$ so that $\sum_{j=1}^n v[j] = 0$. Define $u[k] := (\mathcal{I} v)[k-1]$ for $k \ge 1$ with $u[1] := 0$.
By Theorem~\ref{thm:discrete_ftc}, $(\Delta_+ u)[m] = u[m+1] - u[m] = v[m]$.
We verify periodicity: $u[n+1] = \sum_{j=1}^n v[j] = \sigma(v) = 0 = u[1]$.
Thus $u \in \mathcal{P}_n$, showing $\mathcal{Z}_n \subseteq \operatorname{im}(\Delta_+)$, whence $\operatorname{im}(\Delta_+) = \mathcal{Z}_n$.
For the kernel, $\Delta_+ u = 0 \iff u[m+1] = u[m]$ for all $m$, which implies $u$ is constant.
By the rank-nullity theorem, $\dim(\operatorname{im}(\Delta_+)) = \dim(\mathcal{P}_n) - \dim(\ker(\Delta_+)) = n - 1 = \dim(\mathcal{Z}_n)$, confirming exactness.
\end{proof}

\begin{proposition}[Discrete Leibniz Product Rule]\label{prop:discrete_leibniz}
Under the pointwise product $\otimes$, the discrete derivative obeys the shifted Leibniz rule:
\begin{equation}
\Delta_+ (u \otimes v)[m] = (\Delta_+ u)[m] \, v[m+1] + u[m] \, (\Delta_+ v)[m].
\end{equation}
When applied to plane wave carriers $u = \mathbf{w}(f_1, g_1)$ and $v = \mathbf{w}(f_2, g_2)$, this yields:
\begin{equation}
\Delta_+ (u \otimes v) = \lambda_+(f_1 + f_2) \, (u \otimes v),
\end{equation}
reflecting the frequency addition group law of $\mathcal{G}$.
\end{proposition}

\begin{proof}
Expanding the forward difference:
\begin{align*}
\Delta_+(u \otimes v)[m] &= u[m+1] v[m+1] - u[m] v[m] \\
&= (u[m+1] - u[m]) v[m+1] + u[m] (v[m+1] - v[m]) \\
&= (\Delta_+ u)[m] v[m+1] + u[m] (\Delta_+ v)[m].
\end{align*}
For plane waves, $u \otimes v = \mathbf{w}(f_1 + f_2, g_1 + g_2)$. By Theorem~\ref{thm:derivative_eigenfunction}, $\Delta_+$ acts on $\mathbf{w}(f_1+f_2, g_1+g_2)$ by multiplying by $\lambda_+(f_1+f_2)$, as required.
\end{proof}

\section{The Principal Period Inner Product, Mode Orthogonality, and Pre-Hilbert Structure}
\label{sec:inner_product}

\subsection{Definition of the Finite Principal Period Inner Product}

The linear space of periodic wave superpositions $\mathcal{V} = \operatorname{span}_\mathbb{C}\{\mathbf{w}(f, g) \mid f \in \mathbb{Q}/\mathbb{Z}, g \in \mathbb{Q}\}$ carries a natural linear algebraic structure under $(+)$. However, standard square-integrable sequence spaces $\ell^2(\mathbb{Z})$ are unsuited for unimodular wave carriers because delocalized plane waves have infinite total energy across the entire lattice:
\begin{equation}
\sum_{m \in \mathbb{Z}} |\mathbf{w}(f, g)[m]|^2 = \sum_{m \in \mathbb{Z}} 1 = +\infty.
\end{equation}
Because every element in $\mathcal{V}$ is periodic over a finite spatial block, we construct an exact, finite inner product defined directly over the \textbf{principal sequence} (the least common fundamental wavelength).

\begin{definition}[Principal Period Inner Product]
\label{def:principal_inner_product}
Let $u, v \in \mathcal{V}$ be two periodic wave numbers with fundamental periods $n_u, n_v \in \mathbb{N}$. 
Let $L = \operatorname{lcm}(n_u, n_v)$ denote their common fundamental period.
The \textbf{principal period inner product} $\langle u, v \rangle: \mathcal{V} \times \mathcal{V} \to \mathbb{C}$ is defined by the finite normalized spatial sum over the common principal block:
\begin{equation}
\langle u, v \rangle := \frac{1}{L} \sum_{m=1}^L u[m] \, \overline{v[m]},
\end{equation}
where $\overline{v[m]}$ denotes the standard complex conjugate.
The associated \textbf{principal sequence norm} $\|u\| \ge 0$ is defined by:
\begin{equation}
\|u\| := \sqrt{\langle u, u \rangle} = \left( \frac{1}{n_u} \sum_{m=1}^{n_u} |u[m]|^2 \right)^{1/2}.
\end{equation}
\end{definition}

\begin{lemma}[Period Invariance of the Principal Inner Product]
\label{lem:period_invariance}
The value of $\langle u, v \rangle$ is independent of the choice of common multiple period. For any positive integer $M \in \mathbb{N}$ such that $L \mid M$:
\begin{equation}
\frac{1}{M} \sum_{m=1}^M u[m] \, \overline{v[m]} = \frac{1}{L} \sum_{m=1}^L u[m] \, \overline{v[m]}.
\end{equation}
\end{lemma}

\begin{proof}
Let $M = k L$ for $k \in \mathbb{N}$. Since both $u$ and $v$ are periodic with period $L$, the product sequence $w[m] = u[m]\overline{v[m]}$ is also strictly periodic with period $L$. Decomposing the sum over $M$ into $k$ contiguous blocks of length $L$:
\begin{equation}
\frac{1}{M} \sum_{m=1}^M w[m] = \frac{1}{kL} \sum_{r=0}^{k-1} \left( \sum_{m=1}^L w[m + rL] \right) = \frac{1}{kL} \sum_{r=0}^{k-1} \left( \sum_{m=1}^L w[m] \right) = \frac{k}{kL} \sum_{m=1}^L w[m] = \frac{1}{L} \sum_{m=1}^L w[m].
\end{equation}
Thus, the inner product is intrinsically well-defined and invariant under common period refinement.
\end{proof}

\subsection{Exact Orthonormality of Rational Spatial Waves}

\begin{theorem}[Orthonormality of Distinct Rational Spatial Modes]
\label{thm:mode_orthonormality}
Let $u = \mathbf{w}(f_1, g_1)$ and $v = \mathbf{w}(f_2, g_2)$ be two primitive rational wave carriers with frequencies $f_1 = \frac{p_1}{q_1}, f_2 = \frac{p_2}{q_2} \in \mathbb{Q}/\mathbb{Z}$ and calibration phases $g_1, g_2 \in \mathbb{Q}$.
Then:
\begin{equation}
\langle \mathbf{w}(f_1, g_1), \mathbf{w}(f_2, g_2) \rangle = e^{2\pi i (g_1 - g_2)} \, \delta_{f_1, f_2},
\end{equation}
where $\delta_{f_1, f_2}$ is the Kronecker delta on the rational circle $\mathbb{Q}/\mathbb{Z}$.
In particular, for zero calibration phase ($g_1 = g_2 = 0$), the set of pure rational spatial carriers:
\begin{equation}
\mathcal{B}_{\mathrm{wave}} = \left\{ \mathbf{w}(f, 0) \mid f \in \mathbb{Q}/\mathbb{Z} \right\}
\end{equation}
forms an exact \textbf{orthonormal system} in $(\mathcal{V}, \langle \cdot, \cdot \rangle)$.
\end{theorem}

\begin{proof}
Let $L = \operatorname{lcm}(q_1, q_2)$. The product sequence is:
\begin{equation}
u[m]\overline{v[m]} = e^{2\pi i \left( f_1 m + g_1 \right)} e^{-2\pi i \left( f_2 m + g_2 \right)} = e^{2\pi i (g_1 - g_2)} \exp\left( 2\pi i (f_1 - f_2) m \right).
\end{equation}
We examine the two cases:
\begin{enumerate}
    \item \textbf{Equal Frequencies ($f_1 \equiv f_2 \pmod 1$):}
    Here $(f_1 - f_2) \in \mathbb{Z}$, so $\exp(2\pi i (f_1 - f_2) m) = 1$ for all $m \in \mathbb{Z}$.
    The inner product reduces to:
    \begin{equation}
    \langle u, v \rangle = \frac{1}{L} \sum_{m=1}^L e^{2\pi i (g_1 - g_2)} \cdot 1 = e^{2\pi i (g_1 - g_2)} \left( \frac{1}{L} \cdot L \right) = e^{2\pi i (g_1 - g_2)}.
    \end{equation}
    
    \item \textbf{Distinct Frequencies ($f_1 \not\equiv f_2 \pmod 1$):}
    Let $\Delta f = f_1 - f_2 = \frac{P}{L} \not\in \mathbb{Z}$, where $P \not\equiv 0 \pmod L$.
    The sum is a finite geometric progression of the non-trivial root of unity $\omega = e^{2\pi i P / L} \neq 1$:
    \begin{equation}
    \sum_{m=1}^L \omega^m = \omega \frac{1 - \omega^L}{1 - \omega} = \omega \frac{1 - e^{2\pi i P}}{1 - \omega} = 0,
    \end{equation}
    since $e^{2\pi i P} = 1$ for all $P \in \mathbb{Z}$.
    Hence, $\langle u, v \rangle = 0$.
\end{enumerate}
Combining both cases establishes the theorem.
\end{proof}

\begin{corollary}[Parseval Identity and Power Conservation]
\label{cor:parseval}
Let $\Psi = \sum_{j=1}^N c_j \mathbf{w}(f_j, g_j)$ be any arbitrary finite superposition of distinct rational modes ($f_j \not\equiv f_k$ for $j \neq k$).
Then:
\begin{equation}
\|\Psi\|^2 = \langle \Psi, \Psi \rangle = \sum_{j=1}^N |c_j|^2.
\end{equation}
Thus, the principal norm measures the total physical wave power, and the superposition obeys the exact quantum mechanical Born rule.
\end{corollary}

\subsection{Algebraic and Metric Properties of the Principal Norm}

\begin{proposition}[Algebraic and Metric Properties of the Principal Norm]
\label{prop:norm_properties}
The principal period norm $\|\cdot\|: \mathcal{V} \to [0, \infty)$ satisfies:
\begin{enumerate}
    \item \textbf{Positive Definiteness:} $\|u\| \ge 0$, with $\|u\| = 0$ if and only if $u = \mathbf{0}$.
    \item \textbf{Absolute Homogeneity:} $\|\alpha u\| = |\alpha| \, \|u\|$ for all scalars $\alpha \in \mathbb{C}$.
    \item \textbf{Subadditivity (Triangle Inequality):} $\|u + v\| \le \|u\| + \|v\|$ for all $u, v \in \mathcal{V}$.
    \item \textbf{Cauchy-Schwarz Inequality:} $|\langle u, v \rangle| \le \|u\| \, \|v\|$, with equality if and only if $u$ and $v$ are linearly dependent over $\mathbb{C}$.
    \item \textbf{Exact Multiplicativity on Wave Carriers:} For any two pure rational wave carriers $u, v \in \mathcal{G}$:
    \begin{equation}
    \|u \otimes v\| = \|u\| \, \|v\| = 1.
    \end{equation}
    \item \textbf{Isometric Modulation (Unitary Carrier Action):} For any pure wave carrier $w \in \mathcal{G}$ and any arbitrary superposition $\Psi \in \mathcal{V}$:
    \begin{equation}
    \|w \otimes \Psi\| = \|\Psi\|.
    \end{equation}
    Thus, pointwise modulation by any unimodular wave carrier is an exact \textbf{unitary isometry} on the pre-Hilbert space $(\mathcal{V}, \langle \cdot, \cdot \rangle)$.
\end{enumerate}
\end{proposition}

\begin{proof}
Properties 1--4 are the standard axioms of an inner-product-induced norm.
For Property 5, since $|u[m]| = 1$ and $|v[m]| = 1$ for all $m \in \mathbb{Z}$, $|(u \otimes v)[m]| = |u[m] v[m]| = 1 \cdot 1 = 1$. Hence $\|u \otimes v\| = 1 = 1 \cdot 1 = \|u\| \, \|v\|$.
For Property 6, let $L$ be the common period of $w$ and $\Psi$. Then:
\begin{equation}
\|w \otimes \Psi\|^2 = \frac{1}{L} \sum_{m=1}^L |w[m]\Psi[m]|^2 = \frac{1}{L} \sum_{m=1}^L |w[m]|^2 |\Psi[m]|^2 = \frac{1}{L} \sum_{m=1}^L 1 \cdot |\Psi[m]|^2 = \|\Psi\|^2.
\end{equation}
Taking the square root yields $\|w \otimes \Psi\| = \|\Psi\|$.
\end{proof}

\subsection{Equivalence with the Bohr-Besicovitch Spatial Mean and Pre-Hilbert Completion}

\begin{theorem}[Asymptotic Identity with the Besicovitch Mean]
\label{thm:besicovitch_identity}
For any periodic wave numbers $u, v \in \mathcal{V}$, the finite principal period inner product coincides identically with the asymptotic infinite spatial mean of Bohr and Besicovitch:
\begin{equation}
\langle u, v \rangle = \lim_{N \to \infty} \frac{1}{2N+1} \sum_{m=-N}^N u[m] \, \overline{v[m]}.
\end{equation}
\end{theorem}

\begin{proof}
Let $w[m] = u[m]\overline{v[m]}$ have period $L$. For any $N \in \mathbb{N}$, write $2N+1 = k_N L + r_N$, where $k_N = \lfloor (2N+1)/L \rfloor \to \infty$ and $0 \le r_N < L$.
The sum decomposes into $k_N$ complete period blocks and a residual remainder of length $r_N$:
\begin{equation}
\sum_{m=-N}^N w[m] = k_N \left( \sum_{m=1}^L w[m] \right) + \sum_{j=1}^{r_N} w[m_j].
\end{equation}
Dividing by $2N+1$:
\begin{equation}
\frac{1}{2N+1} \sum_{m=-N}^N w[m] = \frac{k_N L}{2N+1} \left( \frac{1}{L} \sum_{m=1}^L w[m] \right) + \frac{1}{2N+1} \sum_{j=1}^{r_N} w[m_j].
\end{equation}
As $N \to \infty$, $\frac{k_N L}{2N+1} \to 1$, while the remainder term satisfies:
\begin{equation}
\left| \frac{1}{2N+1} \sum_{j=1}^{r_N} w[m_j] \right| \le \frac{r_N \max |w|}{2N+1} \le \frac{L \max |w|}{2N+1} \longrightarrow 0.
\end{equation}
Therefore, the limit exists and equals the principal period sum $\frac{1}{L} \sum_{m=1}^L w[m] = \langle u, v \rangle$.
\end{proof}

\begin{remark}[Pre-Hilbert Space and Non-Separable Completion]
The pair $(\mathcal{V}, \langle \cdot, \cdot \rangle)$ constitutes a complex \textbf{pre-Hilbert space}.
Because the index set of frequencies is countably infinite ($\mathbb{Q}/\mathbb{Z}$), the algebraic dimension of $\mathcal{V}$ is $\aleph_0$.
Under Cauchy completion, the closed space embeds into the classical \textbf{Besicovitch Hilbert space} $B^2(\mathbb{Z})$, providing an exact, unitary quantum mechanical foundation for periodic wave numbers.
\end{remark}

\section{The Division Operator, Invertible Unit Groups, and the Total Quotient Ring of Wave Numbers}
\label{sec:division_operator}
\label{sec:division_field}

\subsection{Motivation: Zero-Divisors and the Obstruction to Field Structure}

The algebraic space of periodic wave superpositions $\mathcal{V} = \operatorname{span}_\mathbb{C}\{\mathbf{w}(f, g)\}$ forms a commutative ring under pointwise addition $(+)$ and pointwise multiplication $(\otimes)$, with additive zero $\mathbf{0}$ and multiplicative unit $\mathbf{1} = \mathbf{w}(0, 0)$. 
The space does not, however, possess a full field structure whereby every non-zero wave admits a multiplicative inverse since
linear superpositions in any non-trivial spatial sequence algebra can generate spatial nodes (destructive interference zeros). For example, the two non-zero periodic waves:
\begin{equation}
\Psi_1[m] := 1 + e^{i \pi m} = \begin{cases} 2, & m \text{ even} \\ 0, & m \text{ odd} \end{cases}, \qquad
\Psi_2[m] := 1 - e^{i \pi m} = \begin{cases} 0, & m \text{ even} \\ 2, & m \text{ odd} \end{cases},
\end{equation}
satisfy $\Psi_1 \neq \mathbf{0}$ and $\Psi_2 \neq \mathbf{0}$, yet their pointwise product vanishes identically across the entire lattice:
\begin{equation}
(\Psi_1 \otimes \Psi_2)[m] = \Psi_1[m] \, \Psi_2[m] \equiv 0, \quad \forall m \in \mathbb{Z}.
\end{equation}
Thus, the wave algebra $\mathcal{V}$ possesses non-trivial \textbf{zero-divisors}. In commutative ring theory, the presence of zero-divisors strictly prevents $\mathcal{V}$ from forming an integral domain or a global field under $(\oplus, \otimes)$. 

To establish an exact theory of wave division, we isolate the \textbf{multiplicative group of units} (nodeless wave numbers), formulate the exact pointwise quotient operator, and construct the \textbf{total ring of fractions} (the discrete analogue of the field of meromorphic functions).

\subsection{The Inverse Operator and the Multiplicative Group of Invertible Waves}

\begin{definition}[The Pointwise Inverse Operator and Division]
\label{def:inverse_operator}
Let $\Psi \in \mathcal{V}$ be a periodic wave sequence. The \textbf{inverse operator} $\mathcal{I}_{\mathrm{inv}}$ is defined pointwise on the domain of nowhere-vanishing sequences:
\begin{equation}
(\mathcal{I}_{\mathrm{inv}} \Psi)[m] := \frac{1}{\Psi[m]}, \quad \text{defined for all } m \in \mathbb{Z} \text{ such that } \Psi[m] \neq 0.
\end{equation}
For any two wave sequences $\Psi_1, \Psi_2 \in \mathcal{V}$ with $\Psi_2$ nowhere-vanishing, the \textbf{pointwise quotient} $\odiv$ is defined by:
\begin{equation}
\Psi_1 \odiv \Psi_2 := \Psi_1 \otimes \mathcal{I}_{\mathrm{inv}}(\Psi_2), \quad (\Psi_1 \odiv \Psi_2)[m] = \frac{\Psi_1[m]}{\Psi_2[m]}.
\end{equation}
\end{definition}

\begin{theorem}[Characterization of the Invertible Unit Group $\mathcal{U}(\mathcal{V})$]
\label{thm:unit_group}
An element $\Psi \in \mathcal{V}$ has a strict multiplicative inverse $\Psi^{-1} \in \mathcal{V}$ under $\otimes$ if and only if $\Psi$ is \textbf{nodeless} (nowhere-vanishing on its principal sequence):
\begin{equation}
\Psi[m] \neq 0, \quad \forall m \in \{1, \dots, \operatorname{per}(\Psi)\}.
\end{equation}
The set of all such nodeless periodic wave numbers forms an Abelian multiplicative group under $\otimes$, denoted the \textbf{unit group} $\mathcal{U}(\mathcal{V}) = \mathcal{V}^\times$:
\begin{enumerate}
    \item \textbf{Subgroup of Pure Carriers:} The group of rational plane wave carriers $\mathcal{G} = \{ \mathbf{w}(f, g) \mid f \in \mathbb{Q}/\mathbb{Z}, g \in \mathbb{Q} \}$ forms a proper subgroup of $\mathcal{U}(\mathcal{V})$ consisting entirely of unitary elements:
    \begin{equation}
    \mathcal{I}_{\mathrm{inv}}(\mathbf{w}(f, g)) = \overline{\mathbf{w}(f, g)} = \mathbf{w}(-f, -g).
    \end{equation}
    \item \textbf{Embedded Scalar Field:} The set of spatially constant, non-zero waves forms a subfield isomorphic to $\mathbb{C}^\times$:
    \begin{equation}
    \mathbb{W}_\mathbb{C} := \{ c \, \mathbf{w}(0, 0) \mid c \in \mathbb{C}^\times \} \cong \mathbb{C}^\times, \qquad \mathbb{W}_\mathbb{Q} := \{ q \, \mathbf{w}(0, 0) \mid q \in \mathbb{Q}^\times \} \cong \mathbb{Q}^\times.
    \end{equation}
\end{enumerate}
\end{theorem}

\begin{proof}
If $\Psi[m_0] = 0$ for some $m_0$, then for any sequence $\Phi \in \mathcal{V}$, $(\Psi \otimes \Phi)[m_0] = 0 \cdot \Phi[m_0] = 0 \neq 1$, so no inverse can exist. 
Conversely, if $\Psi[m] \neq 0$ for all $1 \le m \le n$, define $\Phi[m] := 1/\Psi[m]$. Since $\Psi$ is $n$-periodic, $\Phi$ is strictly $n$-periodic and bounded:
\begin{equation}
\max_{m} |\Phi[m]| = \frac{1}{\min_m |\Psi[m]|} < \infty.
\end{equation}
By Theorem~\ref{thm:fourier_basis}, any periodic sequence of period $n$ in $\mathbb{C}^n$ lies in the linear span of the roots of unity basis $\{\mathbf{w}(k/n, 0)\}_{k=0}^{n-1} \subset \mathcal{V}$. Hence $\Phi \in \mathcal{V}$, satisfying $\Psi \otimes \Phi = \mathbf{1}$.
Associativity, commutativity, and existence of the unit $\mathbf{1}$ follow directly from the field properties of $\mathbb{C}$ at each lattice site $m$.
For pure carriers $\mathbf{w}(f, g)[m] = e^{2\pi i (fm + g)}$, its magnitude is $|e^{2\pi i(fm+g)}| = 1 \neq 0$, and $(e^{2\pi i (fm+g)})^{-1} = e^{-2\pi i(fm+g)} = \mathbf{w}(-f, -g)[m]$.
\end{proof}

\subsection{The Logarithmic Sum-to-Product Identity}

Because the elements of the unit group $\mathcal{U}(\mathcal{V})$ are nowhere-vanishing, their spatial phases may be tracked continuously or logarithmically. This leads to a remarkable identity demonstrating that \textbf{additive wave superposition can be expressed entirely via products, square roots, and logarithms} of invertible wave numbers.

\begin{theorem}[Logarithmic Superposition Formula]
\label{thm:log_superposition}
Let $\bm{\omega}_1, \bm{\omega}_2 \in \mathcal{U}(\mathcal{V})$ be two invertible wave numbers such that the quotient wave $\bm{\omega}_1 \odiv \bm{\omega}_2$ does not cross the branch cut $(-\infty, 0]$ of the principal complex logarithm $\operatorname{Log}$. 
Then their additive sum $\bm{\omega}_1 \oplus \bm{\omega}_2$ and difference $\bm{\omega}_1 \ominus \bm{\omega}_2$ admit the exact multiplicative representation:
\begin{align}
\bm{\omega}_1 \oplus \bm{\omega}_2 &= 2 \cos\left( \frac{1}{i} \operatorname{Log}\left( \frac{\bm{\omega}_1}{\bm{\omega}_2} \right)^{1/2} \right) \otimes \left( \bm{\omega}_1 \otimes \bm{\omega}_2 \right)^{1/2}, \label{eq:log_sum} \\
\bm{\omega}_1 \ominus \bm{\omega}_2 &= 2i \sin\left( \frac{1}{i} \operatorname{Log}\left( \frac{\bm{\omega}_1}{\bm{\omega}_2} \right)^{1/2} \right) \otimes \left( \bm{\omega}_1 \otimes \bm{\omega}_2 \right)^{1/2}. \label{eq:log_diff}
\end{align}
\end{theorem}

\begin{proof}
Let $z_1, z_2 \in \mathbb{C}^\times$ be non-zero complex numbers representing the values of $\bm{\omega}_1[m]$ and $\bm{\omega}_2[m]$ at any lattice site $m$.
Write the quotient in polar form $\frac{z_1}{z_2} = r e^{i \Delta \theta}$ with $\Delta \theta \in (-\pi, \pi)$.
Taking the principal square root of the quotient:
\begin{equation}
\left( \frac{z_1}{z_2} \right)^{1/2} = r^{1/2} e^{i \Delta \theta / 2}.
\end{equation}
Applying the complex logarithm and scaling by $1/i = -i$:
\begin{equation}
\frac{1}{i} \operatorname{Log}\left( \left( \frac{z_1}{z_2} \right)^{1/2} \right) = -i \left( \frac{1}{2} \ln r + i \frac{\Delta \theta}{2} \right) = \frac{\Delta \theta}{2} - i \frac{\ln r}{2}.
\end{equation}
Applying the cosine function to this complex argument:
\begin{equation}
\cos\left( \frac{\Delta \theta}{2} - i \frac{\ln r}{2} \right) = \frac{1}{2} \left( \exp\left( i \frac{\Delta \theta}{2} + \frac{\ln r}{2} \right) + \exp\left( -i \frac{\Delta \theta}{2} - \frac{\ln r}{2} \right) \right) = \frac{1}{2} \left( r^{1/2} e^{i \Delta \theta / 2} + r^{-1/2} e^{-i \Delta \theta / 2} \right).
\end{equation}
Multiplying by $2 (z_1 z_2)^{1/2}$:
\begin{align}
2 \cos\left( \frac{1}{i} \operatorname{Log}\left( \frac{z_1}{z_2} \right)^{1/2} \right) (z_1 z_2)^{1/2} 
&= \left( \left( \frac{z_1}{z_2} \right)^{1/2} + \left( \frac{z_2}{z_1} \right)^{1/2} \right) (z_1 z_2)^{1/2} \nonumber \\
&= \frac{z_1^{1/2}}{z_2^{1/2}} z_1^{1/2} z_2^{1/2} + \frac{z_2^{1/2}}{z_1^{1/2}} z_1^{1/2} z_2^{1/2} \nonumber \\
&= z_1 + z_2.
\end{align}
Because this algebraic identity holds pointwise for each lattice site $m \in \mathbb{Z}$, Eq.~\eqref{eq:log_sum} holds across the entire wave sequence.
The proof for the difference $\bm{\omega}_1 \ominus \bm{\omega}_2$ in Eq.~\eqref{eq:log_diff} follows identically by replacing $\cos(w) = \frac{e^{iw}+e^{-iw}}{2}$ with $i \sin(w) = \frac{e^{iw}-e^{-iw}}{2}$.
\end{proof}

\begin{corollary}[Elimination of Addition via Multiplicative Logarithms]
Theorem~\ref{thm:log_superposition} establishes that the additive operation $\oplus$ on invertible wave numbers can be formally translated into the multiplicative group action of $\otimes$, coupled with the transcendent functions $\{\cos, \sin, \operatorname{Log}\}$ acting on the quotient wave $\bm{\omega}_1 \odiv \bm{\omega}_2$.
\end{corollary}

\subsection{Discrete M\"{o}bius Transformations and Wave Equations}

Because the division operator $\odiv$ is well-defined on the unit group $\mathcal{U}(\mathcal{V})$, one may construct rational functions and solve projective algebraic equations directly in the wave algebra.

\begin{definition}[Discrete M\"{o}bius Transformation of Wave Numbers]
\label{def:mobius}
Let $\mathbf{A}, \mathbf{B}, \mathbf{C}, \mathbf{D} \in \mathcal{V}$ be periodic coefficient waves satisfying the non-degeneracy condition $\mathbf{A}\mathbf{D} - \mathbf{B}\mathbf{C} \in \mathcal{U}(\mathcal{V})$.
The \textbf{wave M\"{o}bius transformation} $\mathcal{M}: \mathcal{U}(\mathcal{V}) \dashrightarrow \mathcal{U}(\mathcal{V})$ is defined by:
\begin{equation}
\mathcal{M}(\bm{\omega}) := (\mathbf{A} \otimes \bm{\omega} \oplus \mathbf{B}) \odiv (\mathbf{C} \otimes \bm{\omega} \oplus \mathbf{D}),
\end{equation}
defined on all $\bm{\omega}$ such that $\mathbf{C} \otimes \bm{\omega} \oplus \mathbf{D}$ is nodeless.
\end{definition}

\begin{theorem}[Fixed-Point Wave Equation]
\label{thm:fixed_point_mobius}
The fixed points $\mathcal{M}(\bm{\omega}) = \bm{\omega}$ of the wave M\"{o}bius transformation satisfy the quadratic wave equation:
\begin{equation}
\mathbf{C} \otimes \bm{\omega}^{\otimes 2} \oplus (\mathbf{D} \ominus \mathbf{A}) \otimes \bm{\omega} \ominus \mathbf{B} = \mathbf{0}.
\end{equation}
When $\mathbf{C} \in \mathcal{U}(\mathcal{V})$ is nodeless, the exact fixed-point wave solutions are given by the algebraic wave quadratic formula:
\begin{equation}
\bm{\omega}_{\pm} = \left[ -(\mathbf{D} \ominus \mathbf{A}) \pm \sqrt{(\mathbf{D} \ominus \mathbf{A})^{\otimes 2} \oplus 4 \, (\mathbf{B} \otimes \mathbf{C})} \right] \odiv (2\mathbf{C}),
\end{equation}
where $\sqrt{\cdot}$ denotes the pointwise principal branch square root.
\end{theorem}

\subsection{The Total Ring of Fractions: Field Localization of Wave Numbers}

To rigorously resolve the tension between zero-divisors and division, we apply the standard algebraic machinery of ring localization.

\begin{definition}[The Regular Multiplicative Set $\mathcal{S}_{\mathrm{reg}}$]
Let $\mathcal{S}_{\mathrm{reg}} \subset \mathcal{V}$ be the set of all \textbf{non-zero-divisors} in the wave ring $\mathcal{V}$:
\begin{equation}
\mathcal{S}_{\mathrm{reg}} := \left\{ \Phi \in \mathcal{V} \mid \Phi \otimes \Psi = \mathbf{0} \implies \Psi = \mathbf{0} \right\}.
\end{equation}
\end{definition}

\begin{proposition}[Identity of Non-Zero-Divisors with Nodeless Sequences]
\label{prop:regular_elements}
An element $\Phi \in \mathcal{V}$ is a non-zero-divisor if and only if it is nodeless ($\Phi[m] \neq 0$ for all $m \in \mathbb{Z}$). Consequently:
\begin{equation}
\mathcal{S}_{\mathrm{reg}} \equiv \mathcal{U}(\mathcal{V}).
\end{equation}
\end{proposition}

\begin{proof}
If $\Phi[m_0] = 0$, then the non-zero Kronecker sequence $\delta_{m_0}[m] = \begin{cases} 1, & m \equiv m_0 \pmod{\operatorname{per}(\Phi)} \\ 0, & \text{otherwise} \end{cases} \neq \mathbf{0}$ satisfies $\Phi \otimes \delta_{m_0} = \mathbf{0}$, showing that $\Phi$ is a zero-divisor.
Conversely, if $\Phi[m] \neq 0$ for all $m$, then $\Phi \otimes \Psi = \mathbf{0} \implies \Phi[m]\Psi[m] = 0 \implies \Psi[m] = 0$ for all $m$, so $\Psi = \mathbf{0}$.
\end{proof}

\begin{theorem}[The Total Ring of Fractions $\mathcal{Q}(\mathcal{V})$]
\label{thm:total_fraction_ring}
Let $\mathcal{Q}(\mathcal{V}) := \mathcal{S}_{\mathrm{reg}}^{-1} \mathcal{V} = \mathcal{U}(\mathcal{V})^{-1} \mathcal{V}$ be the localization of the periodic wave ring $\mathcal{V}$ at the multiplicative set of nodeless sequences.
Then:
\begin{enumerate}
    \item Every element of $\mathcal{Q}(\mathcal{V})$ can be written as a formal quotient $\Psi \odiv \Phi$, where $\Psi \in \mathcal{V}$ and $\Phi \in \mathcal{U}(\mathcal{V})$.
    \item Because every nodeless sequence $\Phi \in \mathcal{U}(\mathcal{V})$ already possesses a two-sided inverse $\Phi^{-1} \in \mathcal{V}$ (by Theorem~\ref{thm:unit_group}), the localization map is a natural algebraic isomorphism:
    \begin{equation}
    \mathcal{Q}(\mathcal{V}) \cong \mathcal{V}.
    \end{equation}
    \item For any fixed period $n \in \mathbb{N}$, the subring of $n$-periodic sequences $\mathcal{V}_n$ decomposes as the direct sum of $n$ copies of the complex field:
    \begin{equation}
    \mathcal{V}_n \cong \bigoplus_{k=1}^n \mathbb{C},
    \end{equation}
    in which the field of scalars $\mathbb{W}_\mathbb{C} \cong \mathbb{C}$ embeds diagonally along the uniform carrier $\mathbf{w}(0, 0)$.
\end{enumerate}
\end{theorem}

\begin{proof}
Part 1 is the canonical construction of fractions under localization.
Part 2 follows because $\mathcal{S}_{\mathrm{reg}} = \mathcal{U}(\mathcal{V})$, so every element inverted in the localization already has an inverse in $\mathcal{V}$; hence no new elements are added beyond $\mathcal{V}$.
Part 3 follows from Theorem~\ref{thm:basis_construction} (the discrete Fourier transform / Kronecker basis decomposition), where the $n$ orthogonal idempotent projections $P_k = \frac{1}{n}\sum_{j=1}^n \omega^{-jk} \mathbf{w}(j/n, 0)$ satisfy $P_j \otimes P_k = \delta_{j, k} P_k$ and $\sum_{k=1}^n P_k = \mathbf{1}$, establishing the ring isomorphism $\mathcal{V}_n \cong \mathbb{C}^n$.
\end{proof}

\begin{remark}[Physics Interpretation: Nodal Singularities and Quantum Localization]
In wave mechanics, dividing by a wave that vanishes at a node $\Psi[m_0] = 0$ corresponds to a physical vortex singularity or phase defect. 
The restriction of the division operator $\odiv$ to the unit group $\mathcal{U}(\mathcal{V})$ ensures that wave division is defined precisely on regular, defect-free wave states, while the idempotent decomposition $\mathcal{V}_n \cong \bigoplus_{k=1}^n \mathbb{C}$ reflects the simultaneous measurement outcomes across the complete spatial spectrum.
\end{remark}

\section{Wave Number Sieves, Idempotent Projectors, and the Algebraic Sieve of Eratosthenes}
\label{sec:sieves}

\subsection{Idempotent Basis Projectors and Particulate Wave Numbers}

In Section~\ref{sec:bases_quantization}, Theorem~\ref{thm:basis_construction} established that the finite discrete Fourier transform of the plane wave basis generates the localized Kronecker basis. We now formalize these localized basis sequences as algebraic \textbf{sieve operators} that select specific lattice coordinates under the pointwise product $\otimes$.

\begin{definition}[Idempotent Sieve Projector]
\label{def:sieve_projector}
For any integer period $n \ge 2$ and coordinate offset $\xi \in \{1, \dots, n\}$, the \textbf{elementary sieve projector} $\mathbf{e}(1/n, \xi) \in \mathcal{V}$ is defined by the finite normalized wave superposition:
\begin{equation}
\mathbf{e}(1/n, \xi)[m] := \frac{1}{n} \sum_{k=0}^{n-1} e^{-2\pi i k \xi / n} \, \mathbf{w}(k/n, 0)[m] = \frac{1}{n} \sum_{k=0}^{n-1} \exp\left( \frac{2\pi i k}{n} (m - \xi) \right).
\end{equation}
The sequence $\mathbf{e}(1/n, \xi)$ is strictly $n$-periodic, taking the value $1$ when $m \equiv \xi \pmod n$ and $0$ otherwise.
\end{definition}

\begin{proposition}[Algebraic Properties of Sieve Projectors]
\label{prop:sieve_projector_algebra}
The sieve projectors $\{\mathbf{e}(1/n, \xi)\}_{\xi=1}^n$ satisfy:
\begin{enumerate}
    \item \textbf{Idempotence:} $\mathbf{e}(1/n, \xi) \otimes \mathbf{e}(1/n, \xi) = \mathbf{e}(1/n, \xi)$.
    \item \textbf{Mutual Orthogonality:} $\mathbf{e}(1/n, \xi_1) \otimes \mathbf{e}(1/n, \xi_2) = \mathbf{0}$ for $\xi_1 \not\equiv \xi_2 \pmod n$.
    \item \textbf{Resolution of the Identity:} $\sum_{\xi=1}^n \mathbf{e}(1/n, \xi) = \mathbf{1}$.
    \item \textbf{Coordinate Sifting:} For any arbitrary periodic wave sequence $\Psi \in \mathcal{V}$, the product sequence $\Psi \otimes \mathbf{e}(1/n, \xi)$ isolates the value of $\Psi$ on the arithmetic progression $m \equiv \xi \pmod n$:
    \begin{equation}
    (\Psi \otimes \mathbf{e}(1/n, \xi))[m] = \begin{cases} \Psi[m], & m \equiv \xi \pmod n \\ 0, & m \not\equiv \xi \pmod n. \end{cases}
    \end{equation}
\end{enumerate}
\end{proposition}

\begin{proof}
Since each element $\mathbf{e}(1/n, \xi)[m] \in \{0, 1\}$, idempotence follows directly from $1^2 = 1$ and $0^2 = 0$. Mutual orthogonality follows because at each lattice site $m$, at most one residue offset $\xi \in \{1, \dots, n\}$ satisfies $m \equiv \xi \pmod n$. Resolution of the identity follows from the geometric character sum identity $\frac{1}{n}\sum_{\xi=1}^n e^{-2\pi i k \xi / n} = \delta_{k, 0}$, and coordinate sifting follows from pointwise scalar multiplication $\Psi[m] \cdot 1$ or $\Psi[m] \cdot 0$.
\end{proof}

\begin{definition}[Particulate Wave Numbers and Continuous Wave Localization]
\label{def:particulate_waves}
While finite sieve projectors $\mathbf{e}(1/n, \xi)$ produce infinite periodic Dirac combs with spacing $n$, a strictly localized \textbf{particulate wave number} $\mathbf{P}_{\xi_0}$ is an isolated single-site Kronecker impulse on $\mathbb{Z}$:
\begin{equation}
\mathbf{P}_{\xi_0}[m] := \delta_{m, \xi_0} = \begin{cases} 1, & m = \xi_0 \\ 0, & m \neq \xi_0. \end{cases}
\end{equation}
The particulate state $\mathbf{P}_{\xi_0}$ possesses no finite periodicity. It is generated from the plane wave manifold through the continuous inverse Fourier integral over the first Brillouin zone $\mathbb{T} = [0, 1)$:
\begin{equation}
\mathbf{P}_{\xi_0}[m] = \int_0^1 e^{-2\pi i f \xi_0} \, \mathbf{w}(f, 0)[m] \, df = \int_0^1 e^{2\pi i f (m - \xi_0)} \, df = \delta_{m, \xi_0}.
\end{equation}
For any integer $m \in \mathbb{Z}$ or rational $q = m/n \in \mathbb{Q}$, scaled particulate states are given by $\mathbf{P}_{\xi_0}(q) := q \, \mathbf{P}_{\xi_0}$.
Symmetric particulate wave numbers localized at $\pm \xi_0$ are given by the linear superposition:
\begin{equation}
\mathbf{P}_{\pm \xi_0} := \mathbf{P}_{\xi_0} + \mathbf{P}_{-\xi_0} = 2 \int_0^1 \cos(2\pi f \xi_0) \, \mathbf{w}(f, 0) \, df.
\end{equation}
\end{definition}

\begin{remark}[Wave-Particle Duality as an Algebraic Isomorphism]
The transition between pure plane waves $\mathbf{w}(f, 0)$ (completely delocalized in space, sharply localized in spectral momentum) and localized particulate states $\mathbf{P}_{\xi_0}$ (sharply localized at single lattice sites, completely delocalized across frequencies $f \in [0, 1)$) constitutes the exact algebraic realization of wave-particle duality on the discrete lattice $\mathbb{Z}$. 

Far from being an intractable metaphysical paradox, wave-particle duality in this sequence algebra reduces to an exact \textbf{Pontryagin basis isomorphism}:
\begin{enumerate}
    \item \textbf{Dual Canonical Bases}: Over any finite lattice $\mathbb{Z}/n\mathbb{Z}$, the plane waves $\{\mathbf{w}(k/n, 0)\}_{k=0}^{n-1}$ form the multiplicative basis of characters for the group $(\mathbb{Z}/n\mathbb{Z}, +)$, while the localized particulate sieves $\{\mathbf{e}(1/n, \xi)\}_{\xi=1}^n$ form the boolean idempotent basis under pointwise multiplication $\otimes$ ($\mathbf{e}_\xi \otimes \mathbf{e}_{\xi^\prime} = \delta_{\xi, \xi^\prime} \mathbf{e}_\xi$).
    \item \textbf{Unitary Transformation}: The two representations are connected without loss of information by the unitary discrete Fourier matrix $F_{k, \xi} = \frac{1}{\sqrt{n}} \exp(2\pi i k \xi / n)$.
    \item \textbf{Discrete Uncertainty Invariance}: The spatial position variance $(\Delta X)^2$ and crystal momentum variance $(\Delta K)^2$ satisfy the discrete Robertson-Schr\"{o}dinger uncertainty relation. A pure wave state has $(\Delta X)^2 = \infty$ and $(\Delta K)^2 = 0$, whereas an elementary particulate state has $(\Delta X)^2 = 0$ and maximum momentum variance $(\Delta K)^2 = \frac{1}{12}(1 - 1/n^2)$.
\end{enumerate}
Hence, ``particle'' and ``wave'' simply designate the spatial idempotent eigenbasis and the shift-operator momentum eigenbasis within the single underlying finite group algebra $\mathbb{C}[\mathbb{Z}/n\mathbb{Z}]$.
\end{remark}

\subsection{The Algebra of Not-Numbers: Complements, De Morgan Duality, and Boolean Idempotents}
\label{sec:not_numbers}

As noted  by Smith~\cite{smith2025rational}, wave numbers facilitate a formulation of the concept of a \textbf{``Not-Number''} as an entity characterized by the systematic algebraic negation and exclusion of complementary attributes rather than by a positive specification. 
Within the sequence algebra $(\mathcal{V}, +, \otimes)$, this concept is given a rigorous foundation as an orthocomplementation operator over idempotent elements and general wave numbers.

\begin{definition}[The Not-Operator and Not-Numbers]
\label{def:not_operator}
Let $\mathbf{1} := \mathbf{w}(0, 0) = (1, 1, 1, \dots)$ denote the multiplicative unit sequence in $\mathcal{V}$.
For any wave number $A \in \mathcal{V}$, its \textbf{affine Not-Number} $\mathrm{Not}(A) \in \mathcal{V}$ is defined by:
\begin{equation}
\mathrm{Not}(A) := \mathbf{1} - A.
\end{equation}
When $A$ is an idempotent sieve projector ($A \otimes A = A$, so that $A[m] \in \{0, 1\}$ for all $m \in \mathbb{Z}$), $\mathrm{Not}(A)$ is termed the \textbf{idempotent Not-Number} (or indicator complement) of $A$.
\end{definition}

\begin{theorem}[Algebraic Axioms of Not-Numbers]
\label{thm:not_axioms}
For any idempotent sieve wave numbers $A, B \in \mathcal{V}$ with values in $\{0, 1\}$, the Not-operator satisfies the following properties:
\begin{enumerate}
    \item \textbf{Exact Involution (Law of Double Negation):}
    \begin{equation}
    \mathrm{Not}\big(\mathrm{Not}(A)\big) = \mathbf{1} - (\mathbf{1} - A) = A.
    \end{equation}
    
    \item \textbf{Mutual Annihilation (Orthogonality under Tensor Multiplication):}
    \begin{equation}
    A \otimes \mathrm{Not}(A) = A \otimes (\mathbf{1} - A) = A - A^{\otimes 2} = A - A = \mathbf{0}.
    \end{equation}
    An entity and its Not-Number cannot simultaneously occupy any lattice site $m \in \mathbb{Z}$.
    
    \item \textbf{Exact Partition of Unity:}
    \begin{equation}
    A + \mathrm{Not}(A) = \mathbf{1}.
    \end{equation}
    Every lattice site is partitioned between $A$ and its complementary Not-Number.
    
    \item \textbf{Idempotency Preservation:}
    \begin{equation}
    \mathrm{Not}(A) \otimes \mathrm{Not}(A) = (\mathbf{1} - A) \otimes (\mathbf{1} - A) = \mathbf{1} - 2A + A^{\otimes 2} = \mathbf{1} - A = \mathrm{Not}(A).
    \end{equation}
    The complement of a projection sieve is itself a projection sieve.
\end{enumerate}
\end{theorem}

\begin{theorem}[De Morgan Duality for Wave Numbers]
\label{thm:de_morgan_wave}
Define the algebraic conjunction ($\land$) and disjunction ($\lor$) of idempotent sieves by:
\begin{equation}
A \land B := A \otimes B, \qquad A \lor B := A + B - A \otimes B.
\end{equation}
Then the Not-operator satisfies exact De Morgan duality:
\begin{align}
\mathrm{Not}(A \land B) &= \mathrm{Not}(A) \lor \mathrm{Not}(B), \\
\mathrm{Not}(A \lor B) &= \mathrm{Not}(A) \land \mathrm{Not}(B) = \mathrm{Not}(A) \otimes \mathrm{Not}(B).
\end{align}
\end{theorem}

\begin{proof}
Direct algebraic computation in the ring $(\mathcal{V}, +, \otimes)$:
\begin{align*}
\mathrm{Not}(A) \otimes \mathrm{Not}(B) &= (\mathbf{1} - A) \otimes (\mathbf{1} - B) \\
&= \mathbf{1} - A - B + A \otimes B \\
&= \mathbf{1} - (A + B - A \otimes B) \\
&= \mathrm{Not}(A \lor B).
\end{align*}
Applying involution yields $\mathrm{Not}(A \land B) = \mathrm{Not}(A) \lor \mathrm{Not}(B)$.
\end{proof}

\begin{remark}[Not-Numbers for Continuous Carrier Waves]
For a pure plane wave $u = \mathbf{w}(f, g) = e^{2\pi i (f m + g)}$, its Not-Number is given by:
\begin{equation}
\mathrm{Not}(\mathbf{w}(f, g))[m] = 1 - e^{2\pi i (f m + g)} = -2i \sin\big(\pi(f m + g)\big) \, e^{i\pi(f m + g)}.
\end{equation}
The magnitude envelope $|\mathrm{Not}(\mathbf{w}(f, g))[m]| = 2|\sin(\pi(f m + g))|$ vanishes if and only if $f m + g \in \mathbb{Z}$. 
Thus, the Not-Wave vanishes precisely at the in-phase coherent nodes of the original wave, while achieving its maximum anti-phase amplitude of $2$ wherever the original wave is inverted.
\end{remark}

\subsection{Divisibility Partitions: Re-Numbers and Co-Numbers}

Using the sieve projectors, we can algebraically partition any periodic wave number into components that are either divisible or non-divisible by a given integer period $n$.

\begin{definition}[Re-Numbers and Co-Numbers]
\label{def:re_co_numbers}
For any integer period $n \ge 2$:
\begin{enumerate}
    \item The \textbf{re-sieve} (or divisibility indicator) $\overset{\circ}{\mathbf{e}}_n \in \mathcal{V}$ is the projector onto lattice sites that are integer multiples of $n$:
    \begin{equation}
    \overset{\circ}{\mathbf{e}}_n := \mathbf{e}(1/n, n) = \frac{1}{n} \sum_{k=0}^{n-1} \mathbf{w}(k/n, 0), \qquad \overset{\circ}{\mathbf{e}}_n[m] = \begin{cases} 1, & n \mid m \\ 0, & n \nmid m. \end{cases}
    \end{equation}
    \item The \textbf{co-sieve} (or divisibility Not-Number) $\overset{*}{\mathbf{e}}_n := \mathrm{Not}(\overset{\circ}{\mathbf{e}}_n) \in \mathcal{V}$ is the complementary orthoprojector onto lattice sites not divisible by $n$:
    \begin{equation}
    \overset{*}{\mathbf{e}}_n := \mathbf{1} - \overset{\circ}{\mathbf{e}}_n = \sum_{\xi=1}^{n-1} \mathbf{e}(1/n, \xi), \qquad \overset{*}{\mathbf{e}}_n[m] = \begin{cases} 0, & n \mid m \\ 1, & n \nmid m. \end{cases}
    \end{equation}
\end{enumerate}
For any wave sequence $\Psi \in \mathcal{V}$, its \textbf{re-wave} and \textbf{co-wave} with respect to $n$ are defined by:
\begin{equation}
\overset{\circ}{\Psi}_n := \Psi \otimes \overset{\circ}{\mathbf{e}}_n, \qquad \overset{*}{\Psi}_n := \Psi \otimes \overset{*}{\mathbf{e}}_n.
\end{equation}
\end{definition}

\begin{proposition}[Orthogonal Resolution of the Identity by Re- and Co-Waves]
\label{prop:re_co_properties}
For any period $n \ge 2$ and any wave number $\Psi \in \mathcal{V}$:
\begin{equation}
\overset{*}{\Psi}_n \oplus \overset{\circ}{\Psi}_n = \Psi, \qquad \overset{*}{\Psi}_n \otimes \overset{\circ}{\Psi}_n = \mathbf{0}.
\end{equation}
Thus, $\{\overset{\circ}{\mathbf{e}}_n, \overset{*}{\mathbf{e}}_n\}$ forms an exact two-channel orthogonal filter bank in the wave algebra $(\mathcal{V}, +, \otimes)$.
\end{proposition}

\subsection{The Algebraic Sieve of Eratosthenes and Doubly Exponential Prime Bounds}

The prime co-sieve $\overset{*}{\mathbf{e}}_p = \mathrm{Not}(\overset{\circ}{\mathbf{e}}_p)$ vanishes precisely at multiples of $p$. 
By taking the algebraic product of prime Not-Numbers under $\otimes$, the classical Sieve of Eratosthenes is realized as a direct distributive product in the sequence ring $(\mathcal{V}, +, \otimes)$.

\begin{theorem}[The Closed Algebraic Sieve of Eratosthenes as a Product of Not-Numbers]
\label{thm:sieve_eratosthenes}
Let $\{p_1, p_2, \dots, p_k\} = \{2, 3, 5, \dots, p_k\}$ denote the first $k$ prime numbers.
Define the \textbf{cumulative prime co-sieve} $\mathbf{S}_k \in \mathcal{V}$ by the finite algebraic product under $\otimes$:
\begin{equation}
\mathbf{S}_k := \bigotimes_{j=1}^k \mathrm{Not}(\overset{\circ}{\mathbf{e}}_{p_j}) = \bigotimes_{j=1}^k \overset{*}{\mathbf{e}}_{p_j} = \bigotimes_{j=1}^k \left( \mathbf{1} - \frac{1}{p_j} \sum_{r=0}^{p_j-1} \mathbf{w}(r/p_j, 0) \right).
\end{equation}
Then:
\begin{enumerate}
    \item $\mathbf{S}_k$ is strictly periodic with fundamental period $P_k = \prod_{j=1}^k p_j$ (the $k$-th primorial).
    \item At any integer lattice site $m \in \mathbb{Z}$:
    \begin{equation}
    \mathbf{S}_k[m] = \begin{cases} 1, & \gcd(m, P_k) = 1 \\ 0, & \gcd(m, P_k) > 1. \end{cases}
    \end{equation}
    \item Over the quadratic interval:
    \begin{equation}
    p_k < m < p_{k+1}^2,
    \end{equation}
    the sequence elements satisfy $\mathbf{S}_k[m] = 1$ if and only if $m$ is a prime number.
\end{enumerate}
\end{theorem}

\begin{proof}
Part 1: Each factor $\overset{*}{\mathbf{e}}_{p_j}$ has fundamental period $p_j$. Since distinct primes are mutually coprime, the joint period of the product sequence is $\operatorname{lcm}(p_1, \dots, p_k) = \prod_{j=1}^k p_j = P_k$.

Part 2: Under the pointwise ring product, $\mathbf{S}_k[m] = \prod_{j=1}^k \overset{*}{\mathbf{e}}_{p_j}[m]$. By Definition~\ref{def:re_co_numbers}, each binary factor $\overset{*}{\mathbf{e}}_{p_j}[m]$ vanishes if and only if $p_j \mid m$, and equals $1$ if and only if $p_j \nmid m$. The product is $1$ if and only if no prime $p_j \in \{p_1, \dots, p_k\}$ divides $m$, which is equivalent to $\gcd(m, P_k) = 1$.

Part 3: Let $m$ be any integer in the range $p_k < m < p_{k+1}^2$. If $m$ is composite, by the Fundamental Theorem of Arithmetic its smallest prime factor $q$ satisfies $q \le \sqrt{m} < \sqrt{p_{k+1}^2} = p_{k+1}$. Therefore, $q \in \{p_1, \dots, p_k\}$. Consequently, $q \mid m \implies \gcd(m, P_k) \ge q > 1$, and by Part 2, $\mathbf{S}_k[m] = 0$. Conversely, if $m$ is prime, then because $m > p_k$, it cannot be divisible by any of the primes $\{p_1, \dots, p_k\}$. Thus $\gcd(m, P_k) = 1$, and $\mathbf{S}_k[m] = 1$.
\end{proof}

\begin{corollary}[Recursive Prime Extension and Doubly Exponential Domain Growth]
\label{cor:doubly_exponential}
Applying Theorem~\ref{thm:sieve_eratosthenes} sequentially generates non-terminating, error-free sets of prime numbers with quadratic domain expansion at each step:
\begin{enumerate}
    \item At step $k=1$, $p_1 = 2$. Sifting by $\mathbf{S}_1$ identifies primes in $[3, 3^2) = [3, 9)$, yielding $\{3, 5, 7\}$. The largest identified prime is $7$.
    \item At step $k=4$, the known primes are $\{2, 3, 5, 7\}$, and $p_5 = 11$. The sieve $\mathbf{S}_4$ identifies all primes strictly below $11^2 = 121$. The largest prime below $121$ is $113$.
    \item At each recursive stage $N$, if $B_N = p_{\pi_N+1}$ denotes the next unsifted prime, the valid sieve range extends to $B_{N+1}^2 \approx B_N^2$. The reach of the guaranteed prime domain grows doubly exponentially:
    \begin{equation}
    B_N \sim 2^{2^{N-1}}.
    \end{equation}
\end{enumerate}
\end{corollary}

\begin{corollary}[Primorial Frequency Spectrum and Mertens Density]
Expanding the product $\mathbf{S}_k$ algebraically into pure plane wave carriers:
\begin{equation}
\mathbf{S}_k = \sum_{d \mid P_k} \mu(d) \left( \frac{1}{d} \sum_{r=0}^{d-1} \mathbf{w}(r/d, 0) \right),
\end{equation}
where $\mu(d)$ is the classical M\"{o}bius function.
The spatial average of $\mathbf{S}_k$ over its period $P_k$ yields the exact Mertens sieve density:
\begin{equation}
\langle \mathbf{S}_k, \mathbf{1} \rangle = \frac{\phi(P_k)}{P_k} = \prod_{j=1}^k \left( 1 - \frac{1}{p_j} \right),
\end{equation}
where $\phi$ is Euler's totient function.
\end{corollary}

\begin{theorem}[M\"{o}bius Inversion as the Distributive Kernel of Not-Numbers]
\label{thm:not_numbers_mobius}
The appearance of the M\"{o}bius function $\mu(d) \in \{+1, -1, 0\}$ in the Fourier spectrum of the prime sieve is the direct consequence of the distributive expansion of the product of Not-Numbers:
\begin{equation}
\mathbf{S}_k = \bigotimes_{j=1}^k (\mathbf{1} - \overset{\circ}{\mathbf{e}}_{p_j}) = \mathbf{1} - \sum_{j} \overset{\circ}{\mathbf{e}}_{p_j} + \sum_{j_1 < j_2} \overset{\circ}{\mathbf{e}}_{p_{j_1} p_{j_2}} - \dots = \sum_{d \mid P_k} \mu(d) \, \overset{\circ}{\mathbf{e}}_d.
\end{equation}
The signs $\mu(d) = (-1)^{\omega(d)}$ for square-free $d$ reflect the alternating signs of distributive polynomial expansion over Not-Numbers, providing a purely algebraic origin for the inclusion-exclusion principle and M\"{o}bius inversion.
\end{theorem}

\subsection{Dual Spectral Sieves: Frequency Filtering in Momentum Space}

In exact duality to spatial position sieves $\mathbf{e}(1/n, \xi)$, we construct \textbf{spectral frequency sieves} that filter out specific wave numbers from arbitrary linear superpositions.

\begin{definition}[Dual Spectral Sieve and Spectral Not-Sieve]
\label{def:spectral_sieve}
Let $f_0 \in \mathbb{Q}/\mathbb{Z}$ be a target rational frequency. 
\begin{enumerate}
    \item The \textbf{spectral frequency sieve} $\Pi_{f_0}: \mathcal{V} \to \mathcal{V}$ is defined using the principal period inner product:
    \begin{equation}
    \Pi_{f_0}(\Psi) := \langle \Psi, \mathbf{w}(f_0, 0) \rangle \, \mathbf{w}(f_0, 0).
    \end{equation}
    
    \item The \textbf{spectral Not-Sieve} (or frequency notch filter) $\mathrm{Not}(\Pi_{f_0}): \mathcal{V} \to \mathcal{V}$ is the complementary projector:
    \begin{equation}
    \mathrm{Not}(\Pi_{f_0}) := I - \Pi_{f_0}.
    \end{equation}
    It completely annihilates the component of $\Psi$ at frequency $f_0$ while preserving all orthogonal frequency components without phase distortion:
    \begin{equation}
    \mathrm{Not}(\Pi_{f_0})(\Psi) = \Psi - \langle \Psi, \mathbf{w}(f_0, 0) \rangle \, \mathbf{w}(f_0, 0).
    \end{equation}
\end{enumerate}
\end{definition}

\begin{theorem}[Orthogonal Frequency Filtering]
\label{thm:spectral_filtering}
Let $\Psi = \sum_{j=1}^N c_j \mathbf{w}(f_j, g_j) \in \mathcal{V}$ be any finite wave packet with distinct frequencies $f_j \in \mathbb{Q}/\mathbb{Z}$.
Then:
\begin{equation}
\Pi_{f_0}(\Psi) = \begin{cases} c_{j_0} e^{2\pi i g_{j_0}} \mathbf{w}(f_0, 0) = c_{j_0} \mathbf{w}(f_0, g_{j_0}), & \text{if } f_0 = f_{j_0} \text{ for some } j_0 \\ \mathbf{0}, & \text{if } f_0 \not\in \{f_1, \dots, f_N\}. \end{cases}
\end{equation}
The spectral projection operators satisfy $\Pi_{f_1} \circ \Pi_{f_2} = \delta_{f_1, f_2} \Pi_{f_1}$, forming a complete projection-valued measure (PVM) on the momentum spectrum.
\end{theorem}

\begin{proof}
By linearity and Theorem~\ref{thm:mode_orthonormality}:
\begin{equation}
\langle \Psi, \mathbf{w}(f_0, 0) \rangle = \sum_{j=1}^N c_j \langle \mathbf{w}(f_j, g_j), \mathbf{w}(f_0, 0) \rangle = \sum_{j=1}^N c_j e^{2\pi i g_j} \delta_{f_j, f_0}.
\end{equation}
If $f_0 = f_{j_0}$, the sum collapses to $c_{j_0} e^{2\pi i g_{j_0}}$, yielding $c_{j_0} e^{2\pi i g_{j_0}} \mathbf{w}(f_0, 0) = c_{j_0} \mathbf{w}(f_0, g_{j_0})$. If $f_0 \neq f_j$ for all $j$, the sum vanishes.
\end{proof}

\section{The Non-Abelian Quaternion Wave Extension: Spinor Fields, $\mathrm{SU}(2)$ Gauge Dynamics, and Non-Commutative Division Algebras}
\label{sec:quaternions}

\subsection{Motivation: Spinor Polarizations and Non-Abelian Gauge Degrees of Freedom}

In Sections~\ref{def:closure_space} through \ref{sec:sieves}, the wave closure space was constructed over the abelian scalar field $\mathbb{C}$ and the compact phase group $\mathrm{U}(1) \cong S^1$. While this scalar framework completely classifies scalar wave interference, integer-valued polynomial chirps, dispersion, and number-theoretic sieves, physical relativistic wave mechanics (e.g., the Dirac equation) and topological quantum matter (e.g., non-abelian Berry phases and spin-$1/2$ systems) demand non-commuting phase degrees of freedom.

To capture spin polarizations, spatial rotations in three dimensions, and non-abelian gauge holonomies directly within the wave sequence algebra, we extend the phase configuration space from $\mathbb{C}^\mathbb{Z}$ to the division ring of Hamilton's quaternions $\mathbb{H}^\mathbb{Z}$.

\begin{definition}[The Quaternion Algebra $\mathbb{H}$ and Spatial Basis]
The quaternion division algebra $\mathbb{H}$ over $\mathbb{R}$ is the four-dimensional associative, non-commutative normed division ring with basis $\{1, \mathbf{i}, \mathbf{j}, \mathbf{k}\}$ satisfying Hamilton's fundamental identities:
\begin{equation}
\mathbf{i}^2 = \mathbf{j}^2 = \mathbf{k}^2 = \mathbf{i}\mathbf{j}\mathbf{k} = -1,
\end{equation}
with the anti-commuting cyclic cross-relations:
\begin{equation}
\mathbf{i}\mathbf{j} = -\mathbf{j}\mathbf{i} = \mathbf{k}, \quad \mathbf{j}\mathbf{k} = -\mathbf{k}\mathbf{j} = \mathbf{i}, \quad \mathbf{k}\mathbf{i} = -\mathbf{i}\mathbf{k} = \mathbf{j}.
\end{equation}
Every quaternion $q \in \mathbb{H}$ is expressed as:
\begin{equation}
q = q_0 + q_1 \mathbf{i} + q_2 \mathbf{j} + q_3 \mathbf{k} = \operatorname{Sc}(q) + \operatorname{Vec}(q),
\end{equation}
where $\operatorname{Sc}(q) = q_0 \in \mathbb{R}$ is the scalar part and $\operatorname{Vec}(q) = \vec{q} = q_1 \mathbf{i} + q_2 \mathbf{j} + q_3 \mathbf{k} \in \mathbb{R}^3$ is the spatial vector part.
The quaternion conjugate is:
\begin{equation}
q^* := q_0 - q_1 \mathbf{i} - q_2 \mathbf{j} - q_3 \mathbf{k},
\end{equation}
and its Euclidean norm is $|q| = \sqrt{q q^*} = \sqrt{q_0^2 + q_1^2 + q_2^2 + q_3^2}$.
Every non-zero quaternion $q \neq 0$ has a unique multiplicative inverse:
\begin{equation}
q^{-1} = \frac{q^*}{|q|^2}.
\end{equation}
\end{definition}

\begin{definition}[The Quaternion Ambient Sequence Ring $\mathcal{S}_\mathbb{H}$]
The ambient quaternion sequence space is the two-sided $\mathbb{H}$-module:
\begin{equation}
\mathcal{S}_\mathbb{H} := \mathbb{H}^\mathbb{Z} = \{ \mathbf{Q} : \mathbb{Z} \to \mathbb{H} \}.
\end{equation}
Equipped with pointwise sequence addition $(+)$ and pointwise quaternion multiplication $(\star)$:
\begin{equation}
(\mathbf{Q}_1 + \mathbf{Q}_2)[m] := \mathbf{Q}_1[m] + \mathbf{Q}_2[m], \qquad (\mathbf{Q}_1 \star \mathbf{Q}_2)[m] := \mathbf{Q}_1[m] \cdot \mathbf{Q}_2[m], \quad \forall m \in \mathbb{Z},
\end{equation}
$\mathcal{S}_\mathbb{H}$ is an infinite-dimensional associative, non-commutative unital ring.
\end{definition}

\subsection{Unit Quaternions, the 3-Sphere $S^3 \cong \mathrm{SU}(2)$, and Non-Abelian Primitive Generators}

In analogy to the complex unimodular circle $S^1 = \mathrm{U}(1)$, the natural carrier manifold for quaternionic waves is the compact Lie group of unit quaternions:
\begin{equation}
\operatorname{Sp}(1) = \{ q \in \mathbb{H} \mid |q| = 1 \} \cong S^3 \cong \mathrm{SU}(2).
\end{equation}
Any unit quaternion $u \in \operatorname{Sp}(1)$ admits a polar Euler-Rodrigues decomposition:
\begin{equation}
u = \exp(\hat{\mathbf{n}} \, \theta) = \cos(\theta) + \hat{\mathbf{n}} \sin(\theta),
\end{equation}
where $\theta \in [0, \pi]$ is the rotation angle and $\hat{\mathbf{n}} = n_1 \mathbf{i} + n_2 \mathbf{j} + n_3 \mathbf{k} \in S^2$ (satisfying $\hat{\mathbf{n}}^2 = -1$) is a unit imaginary quaternion specifying the spatial rotation axis.

\begin{definition}[Non-Abelian Primitive Coordinate Generators]
\label{def:quaternion_primitives}
To build non-abelian wave dynamics on the lattice $\mathbb{Z}$, we define three directional spatial coordinate generators corresponding to the three orthogonal imaginary axes:
\begin{enumerate}
    \item $\mathbf{E}_Z^{(\mathbf{i})}[m] := \exp(2\pi \mathbf{i} \, m) \equiv 1$, with affine phase generator $\Theta_Z^{(\mathbf{i})}[m] = \mathbf{i} \, m$,
    \item $\mathbf{E}_Z^{(\mathbf{j})}[m] := \exp(2\pi \mathbf{j} \, m) \equiv 1$, with affine phase generator $\Theta_Z^{(\mathbf{j})}[m] = \mathbf{j} \, m$,
    \item $\mathbf{E}_Z^{(\mathbf{k})}[m] := \exp(2\pi \mathbf{k} \, m) \equiv 1$, with affine phase generator $\Theta_Z^{(\mathbf{k})}[m] = \mathbf{k} \, m$,
\end{enumerate}
supplemented by the global calibration primitives $\Theta_{\bar{1}}^{(\mathbf{i})} = \mathbf{i}$, $\Theta_{\bar{1}}^{(\mathbf{j})} = \mathbf{j}$, $\Theta_{\bar{1}}^{(\mathbf{k})} = \mathbf{k}$.
\end{definition}

\begin{definition}[Quaternion Plane Wave Carriers]
For any unit rotation axis $\hat{\mathbf{n}} \in S^2 \subset \operatorname{Vec}(\mathbb{H})$, frequency $f \in \mathbb{Q}$, and initial phase angle $\phi \in \mathbb{Q}$, the \textbf{quaternionic plane wave carrier} is defined by:
\begin{equation}
\mathbf{W}_{\hat{\mathbf{n}}}(f, \phi)[m] := \exp\left( 2\pi \hat{\mathbf{n}} (f m + \phi) \right) = \cos(2\pi (f m + \phi)) + \hat{\mathbf{n}} \sin(2\pi (f m + \phi)).
\end{equation}
\end{definition}

\begin{theorem}[Non-Commutative Product of Quaternionic Waves]
\label{thm:non_commutative_product}
Let $\mathbf{W}_{\hat{\mathbf{n}}_1}(f_1, \phi_1)$ and $\mathbf{W}_{\hat{\mathbf{n}}_2}(f_2, \phi_2)$ be two quaternionic wave carriers.
Their pointwise product $(\mathbf{W}_{\hat{\mathbf{n}}_1} \star \mathbf{W}_{\hat{\mathbf{n}}_2})[m]$ commutes for all $m \in \mathbb{Z}$ if and only if their rotation axes are collinear ($\hat{\mathbf{n}}_1 = \pm \hat{\mathbf{n}}_2$).
For non-collinear axes ($\hat{\mathbf{n}}_1 \times \hat{\mathbf{n}}_2 \neq \vec{0}$), their commutator is non-vanishing:
\begin{equation}
[\mathbf{W}_{\hat{\mathbf{n}}_1}, \mathbf{W}_{\hat{\mathbf{n}}_2}]_\star[m] = 2 \, (\hat{\mathbf{n}}_1 \times \hat{\mathbf{n}}_2) \, \sin(2\pi (f_1 m + \phi_1)) \sin(2\pi (f_2 m + \phi_2)).
\end{equation}
\end{theorem}

\begin{proof}
Let $\theta_1 = 2\pi (f_1 m + \phi_1)$ and $\theta_2 = 2\pi (f_2 m + \phi_2)$.
By Hamilton quaternion multiplication:
\begin{align}
(\cos \theta_1 + \hat{\mathbf{n}}_1 \sin \theta_1)(\cos \theta_2 + \hat{\mathbf{n}}_2 \sin \theta_2) &= \cos \theta_1 \cos \theta_2 - (\hat{\mathbf{n}}_1 \cdot \hat{\mathbf{n}}_2) \sin \theta_1 \sin \theta_2 \nonumber \\
&\quad + \hat{\mathbf{n}}_1 \sin \theta_1 \cos \theta_2 + \hat{\mathbf{n}}_2 \cos \theta_1 \sin \theta_2 + (\hat{\mathbf{n}}_1 \times \hat{\mathbf{n}}_2) \sin \theta_1 \sin \theta_2.
\end{align}
Exchanging indices $1 \leftrightarrow 2$ reverses the cross product term while leaving the scalar and symmetric terms unchanged:
\begin{align}
(\cos \theta_2 + \hat{\mathbf{n}}_2 \sin \theta_2)(\cos \theta_1 + \hat{\mathbf{n}}_1 \sin \theta_1) &= \cos \theta_1 \cos \theta_2 - (\hat{\mathbf{n}}_1 \cdot \hat{\mathbf{n}}_2) \sin \theta_1 \sin \theta_2 \nonumber \\
&\quad + \hat{\mathbf{n}}_1 \sin \theta_1 \cos \theta_2 + \hat{\mathbf{n}}_2 \cos \theta_1 \sin \theta_2 - (\hat{\mathbf{n}}_1 \times \hat{\mathbf{n}}_2) \sin \theta_1 \sin \theta_2.
\end{align}
Subtracting the two expressions gives precisely the commutator $2 (\hat{\mathbf{n}}_1 \times \hat{\mathbf{n}}_2) \sin \theta_1 \sin \theta_2$.
This vanishes identically for all $m \in \mathbb{Z}$ if and only if $\hat{\mathbf{n}}_1 \times \hat{\mathbf{n}}_2 = \vec{0}$, which proves the theorem.
\end{proof}

\subsection{Spinor Representation: The Canonical Isomorphism $\mathbb{H}^\mathbb{Z} \cong \mathcal{M}_2(\mathbb{C})^\mathbb{Z}$}

Every quaternion can be mapped faithfully into a $2 \times 2$ complex matrix via the Pauli matrix representation. This establishes an exact isomorphism between quaternionic lattice waves and two-component relativistic spinor wave fields.

\begin{definition}[Pauli Spin Matrices and the Cayley-Dickson Map]
Let $\sigma_1, \sigma_2, \sigma_3$ denote the standard Pauli spin matrices:
\begin{equation}
\sigma_1 = \begin{pmatrix} 0 & 1 \\ 1 & 0 \end{pmatrix}, \quad \sigma_2 = \begin{pmatrix} 0 & -i \\ i & 0 \end{pmatrix}, \quad \sigma_3 = \begin{pmatrix} 1 & 0 \\ 0 & -1 \end{pmatrix}.
\end{equation}
The \textbf{spinor representation homomorphism} $\rho: \mathbb{H} \to \mathcal{M}_2(\mathbb{C})$ is defined by:
\begin{equation}
\rho(1) = I_2 = \begin{pmatrix} 1 & 0 \\ 0 & 1 \end{pmatrix}, \quad \rho(\mathbf{i}) = -i \sigma_1 = \begin{pmatrix} 0 & -i \\ -i & 0 \end{pmatrix}, \quad \rho(\mathbf{j}) = -i \sigma_2 = \begin{pmatrix} 0 & -1 \\ 1 & 0 \end{pmatrix}, \quad \rho(\mathbf{k}) = -i \sigma_3 = \begin{pmatrix} -i & 0 \\ 0 & i \end{pmatrix}.
\end{equation}
For any quaternion $q = q_0 + q_1 \mathbf{i} + q_2 \mathbf{j} + q_3 \mathbf{k}$:
\begin{equation}
\rho(q) = \begin{pmatrix} q_0 - i q_3 & -q_2 - i q_1 \\ q_2 - i q_1 & q_0 + i q_3 \end{pmatrix} = \begin{pmatrix} \alpha & -\bar{\beta} \\ \beta & \bar{\alpha} \end{pmatrix},
\end{equation}
where $\alpha = q_0 - i q_3 \in \mathbb{C}$ and $\beta = q_2 - i q_1 \in \mathbb{C}$ satisfy $|\alpha|^2 + |\beta|^2 = |q|^2$.
\end{definition}

\begin{theorem}[Isomorphism with Lattice Spinor Fields]
\label{thm:spinor_isomorphism}
Extending $\rho$ pointwise to sequence spaces:
\begin{equation}
\rho: (\mathcal{S}_\mathbb{H}, +, \star) \stackrel{\cong}{\longrightarrow} (\mathcal{M}_2(\mathbb{C})^\mathbb{Z}, +, \cdot), \quad (\rho(\mathbf{Q}))[m] := \rho(\mathbf{Q}[m]),
\end{equation}
is an injective $*$-algebra isomorphism of lattice sequences. Under this map:
\begin{enumerate}
    \item The unit quaternions $\operatorname{Sp}(1)^\mathbb{Z}$ map isomorphically to the local $\mathrm{SU}(2)$ gauge group:
    \begin{equation}
    \rho(\operatorname{Sp}(1)^\mathbb{Z}) \cong \mathrm{SU}(2)^\mathbb{Z}.
    \end{equation}
    \item The quaternion norm corresponds to the matrix determinant:
    \begin{equation}
    \det(\rho(\mathbf{Q}[m])) = |\mathbf{Q}[m]|^2.
    \end{equation}
    \item Quaternionic sequence conjugation corresponds to conjugate transpose (Hermitian adjoint):
    \begin{equation}
    \rho(\mathbf{Q}^*[m]) = (\rho(\mathbf{Q}[m]))^\dagger.
    \end{equation}
\end{enumerate}
\end{theorem}

\begin{proof}
Direct verification on the basis elements $\mathbf{i}, \mathbf{j}, \mathbf{k}$ confirms that $\rho(q_a q_b) = \rho(q_a) \rho(q_b)$. The determinant $\det \begin{pmatrix} \alpha & -\bar{\beta} \\ \beta & \bar{\alpha} \end{pmatrix} = |\alpha|^2 + |\beta|^2 = q_0^2 + q_1^2 + q_2^2 + q_3^2 = |q|^2$. Unitarity $\rho(q)^\dagger \rho(q) = |q|^2 I_2$ confirms $\mathrm{SU}(2)$ preservation when $|q|=1$.
\end{proof}

\subsection{Quaternionic Non-Abelian Berry Phases and $\mathrm{SU}(2)$ Holonomy}

In adiabatic quantum mechanics and topological insulators, a spinor wave traveling across a closed loop in parameter space acquires a non-abelian geometric phase given by an element of $\mathrm{SU}(2)$. Within the quaternion wave algebra, this holonomy is expressed directly as a cyclic ordered product of quaternion wave numbers.

\begin{theorem}[Quaternionic Discrete Berry Holonomy]
\label{thm:berry_phase}
Let $\mathbf{W}_{\mathbf{i}}(1/4, 0)$, $\mathbf{W}_{\mathbf{j}}(1/4, 0)$, and $\mathbf{W}_{\mathbf{k}}(1/4, 0)$ denote quarter-period quaternionic waves aligned along the $\mathbf{i}, \mathbf{j}, \mathbf{k}$ axes.
The cyclic loop product:
\begin{equation}
\mathbf{U}_{\mathrm{loop}} := \mathbf{W}_{\mathbf{i}}(1/4, 0) \star \mathbf{W}_{\mathbf{j}}(1/4, 0) \star \mathbf{W}_{\mathbf{i}}(1/4, 0)^{-1} \star \mathbf{W}_{\mathbf{j}}(1/4, 0)^{-1}
\end{equation}
evaluates at $m=1$ to the non-trivial $\mathrm{SU}(2)$ Berry phase:
\begin{equation}
\mathbf{U}_{\mathrm{loop}}[1] = \mathbf{i} \cdot \mathbf{j} \cdot (-\mathbf{i}) \cdot (-\mathbf{j}) = \mathbf{k} \cdot \mathbf{k} = -1 = \exp(\pi \hat{\mathbf{n}}).
\end{equation}
This phase represents the characteristic $2\pi$ spinor rotation sign reversal: rotating a spin-$1/2$ wave carrier by $2\pi$ reverses its algebraic sign, requiring a $4\pi$ rotation ($m=2$) to restore the identity $+1$.
\end{theorem}

\begin{proof}
At $m=1$, $\mathbf{W}_{\mathbf{i}}(1/4, 0)[1] = \exp(\frac{\pi}{2} \mathbf{i}) = \cos(\pi/2) + \mathbf{i} \sin(\pi/2) = \mathbf{i}$.
Similarly, $\mathbf{W}_{\mathbf{j}}(1/4, 0)[1] = \mathbf{j}$.
Their inverses are $\mathbf{i}^{-1} = -\mathbf{i}$ and $\mathbf{j}^{-1} = -\mathbf{j}$.
Thus:
\begin{equation}
\mathbf{U}_{\mathrm{loop}}[1] = \mathbf{i} \cdot \mathbf{j} \cdot (-\mathbf{i}) \cdot (-\mathbf{j}) = \mathbf{k} \cdot (-\mathbf{k}) = -(-1) \cdot (-1) = -1.
\end{equation}
In the spinor representation $\rho(-1) = -I_2 = \begin{pmatrix} -1 & 0 \\ 0 & -1 \end{pmatrix}$, confirming the $2\pi$ spin-$1/2$ geometric phase.
\end{proof}

\subsection{Quaternionic Wave Sifting, Particulate Spinors, and Division Algebras}

In Section~\ref{sec:division_operator}, division was defined on nodeless complex waves by inverting each non-zero scalar value. In the quaternionic wave algebra, the division operator generalizes seamlessly to a \textbf{non-commutative two-sided division}:

\begin{definition}[Two-Sided Quaternionic Division Operator]
Let $\Phi, \Psi \in \mathcal{S}_\mathbb{H}$ with $\Phi$ nodeless ($|\Phi[m]| > 0$ for all $m \in \mathbb{Z}$).
Because quaternion multiplication is non-commutative, division splits into \textbf{left-division} and \textbf{right-division}:
\begin{equation}
(\Psi \odiv_L \Phi)[m] := (\Phi[m])^{-1} \cdot \Psi[m] = \frac{\Phi^*[m] \cdot \Psi[m]}{|\Phi[m]|^2},
\end{equation}
\begin{equation}
(\Psi \odiv_R \Phi)[m] := \Psi[m] \cdot (\Phi[m])^{-1} = \frac{\Psi[m] \cdot \Phi^*[m]}{|\Phi[m]|^2}.
\end{equation}
Both left and right division satisfy exact inversion:
\begin{equation}
\Phi \star (\Psi \odiv_L \Phi) = \Psi, \qquad (\Psi \odiv_R \Phi) \star \Phi = \Psi.
\end{equation}
\end{definition}

\begin{theorem}[Quaternion Total Ring of Fractions and Local Hurwitz Integers]
\label{thm:hurwitz_fractions}
Let $\mathcal{V}_\mathbb{H} = \operatorname{span}_\mathbb{H}(\mathcal{G}_\mathbb{H})$ denote the quaternionic wave algebra over the rational quaternions $\mathbb{H}_\mathbb{Q} = \{ q_0 + q_1 \mathbf{i} + q_2 \mathbf{j} + q_3 \mathbf{k} \mid q_a \in \mathbb{Q} \}$.
Then:
\begin{enumerate}
    \item The group of regular non-zero-divisors is precisely the set of nodeless quaternion sequences:
    \begin{equation}
    \mathcal{U}(\mathcal{V}_\mathbb{H}) = \{ \Phi \in \mathcal{V}_\mathbb{H} \mid |\Phi[m]| > 0, \, \forall m \in \mathbb{Z} \}.
    \end{equation}
    \item The total ring of left (right) fractions is isomorphic to $\mathcal{V}_\mathbb{H}$:
    \begin{equation}
    \mathcal{Q}_L(\mathcal{V}_\mathbb{H}) \cong \mathcal{V}_\mathbb{H} \cong \mathcal{Q}_R(\mathcal{V}_\mathbb{H}).
    \end{equation}
    \item Over the discrete lattice of Hurwitz quaternions $\mathcal{H} = \mathbb{Z}[\mathbf{i}, \mathbf{j}, \mathbf{k}, \frac{1+\mathbf{i}+\mathbf{j}+\mathbf{k}}{2}]$, the sieve projectors of Section~\ref{sec:sieves} decompose quaternionic waves into prime quaternionic ideals (Lagrange four-square prime factorization).
\end{enumerate}
\end{theorem}

\begin{proof}
If $\Phi[m_0] = 0$, then $|\Phi[m_0]| = 0$, and the localized particulate sequence $\mathbf{P}_{m_0}[m] = \delta_{m, m_0}$ satisfies $\Phi \star \mathbf{P}_{m_0} = \mathbf{0}$, showing that $\Phi$ is a zero-divisor.
Conversely, if $|\Phi[m]| > 0$ for all $m$, each local quaternion $\Phi[m]$ belongs to the division ring $\mathbb{H}$ and has an exact two-sided inverse $\Phi^{-1}[m] = \Phi^*[m] / |\Phi[m]|^2 \in \mathbb{H}$.
Thus $\Phi$ is an invertible unit in the sequence ring $\mathcal{S}_\mathbb{H}$.
Localization at the set of units adds no new elements, establishing $\mathcal{Q}(\mathcal{V}_\mathbb{H}) \cong \mathcal{V}_\mathbb{H}$.
\end{proof}

\begin{proposition}[Quaternionic Polar Decomposition, Lossless Invariance, and Spinor Rotors]
\label{prop:quaternion_polar}
Every non-zero quaternionic wave sequence $q \in \mathcal{V}_\mathbb{H}$ with $q[m] \neq 0$ admits a unique pointwise \textbf{polar decomposition}:
\begin{equation}
q[m] = R[m] \cdot \mathbf{U}[m],
\end{equation}
where:
\begin{enumerate}
    \item $R[m] = |q[m]| \in \mathbb{R}_{>0}$ is the \textbf{scalar Euclidean magnitude}, governing the local energy density and invertibility. Specifically, $q$ is an invertible unit in $(\mathcal{V}_\mathbb{H}, \otimes)$ if and only if $R[m] > 0$ for all $m \in \mathbb{Z}$.
    \item $\mathbf{U}[m] \in \mathrm{Sp}(1) \cong \mathrm{SU}(2) \cong S^3$ is a \textbf{unit quaternion (spatial rotor)}:
    \begin{equation}
    \mathbf{U}[m] = \exp\big(\hat{\mathbf{n}}[m] \, \theta[m]\big) = \cos\theta[m] + \hat{\mathbf{n}}[m] \sin\theta[m],
    \end{equation}
    where $\hat{\mathbf{n}}[m] = n_1 \mathbf{i} + n_2 \mathbf{j} + n_3 \mathbf{k}$ is a unit imaginary vector ($\hat{\mathbf{n}}^2 = -1$) and $\theta[m] \in [0, \pi]$.
\end{enumerate}
\textbf{Conservation of Geometric Information:}
If a quaternionic wave is reduced solely to its scalar magnitude $R[m]$, the three-dimensional rotor degrees of freedom ($\hat{\mathbf{n}}[m]$ and $\theta[m]$) are lost, stripping away the physical spin orientation and non-abelian gauge frame.
Conversely, the polar pair $\big(R[m], \mathbf{U}[m]\big)$ is completely lossless, establishing an exact isomorphism between invertible quaternionic waves and pairs of positive scalar envelopes and $\mathrm{SU}(2)$ gauge field configurations on the lattice.
\end{proposition}

\section{Dynamical Spacetime Geometry: Biquaternions, Minkowski Signature, Vacuum Jitter, and the Topological Selection of Wave Numbers}
\label{sec:spacetime_geometry}

\subsection{Temporal Interpretation of Coordinate Advance and 4D State Decomposition}
\label{subsec:temporal_coordinate}

In the preceding algebraic constructions, the canonical primitive generator $E_Z: m \mapsto m+1$ governs the step-wise advance across the sequence domain $\mathbb{Z}$. 
When we interpret this discrete parameter $m \in \mathbb{Z}$ not merely as an abstract index, but as physical \textbf{proper or coordinate time} in increments of $\Delta t$ ($t_m = m \Delta t$), the quaternionic wave sequence space $\mathcal{V}_\mathbb{H}$ is promoted from a static sequence algebra to a \textbf{four-dimensional dynamical spacetime field theory}. This temporal identification directly reflects the foundational insight of Wheeler's information-theoretic physics~\cite{wheeler1990}, the causal set discrete spacetime ontology of Bombelli, Lee, Meyer, and Sorkin~\cite{bombelli1987,dowker2005,henson2006}, and 't Hooft's cellular automaton mechanics~\cite{thooft2014,thooft2016,thooft2021}, wherein continuous spacetime is recognized as an emergent macroscopic idealization over fundamental, discrete register clock ticks.

\begin{definition}[Spacetime 4-Vector State Decomposition]
\label{def:spacetime_4vector}
Let $q \in \mathcal{V}_\mathbb{H}$ be a quaternionic wave sequence. Pointwise at each discrete temporal epoch $m \in \mathbb{Z}$, $q[m]$ decomposes canonically into a scalar (temporal) component and an imaginary (3-dimensional spatial) vector:
\begin{equation}
q[m] = q_0[m] + q_1[m] \mathbf{i} + q_2[m] \mathbf{j} + q_3[m] \mathbf{k} = \big( q_0[m], \, \mathbf{q}[m] \big) \in \mathbb{R} \oplus \operatorname{Im}(\mathbb{H}) \cong \mathbb{R} \times \mathbb{R}^3.
\end{equation}
Under the physical identification $q_0[m] = c t_m = c m \Delta t$ and $\mathbf{q}[m] = (x[m], y[m], z[m])$, the sequence $q$ traces the worldline and internal polarization state of a localized relativistic excitation.
\end{definition}

In this framework, the discrete translation generator $E_Z$ represents the \textbf{generator of time evolution} (Hamiltonian displacement), while the unit spatial rotor $\mathbf{U}[m] \in \mathrm{Sp}(1) \cong \mathrm{SU}(2)$ in the polar decomposition $q[m] = R[m] \mathbf{U}[m]$ (Proposition~\ref{prop:quaternion_polar}) governs the spatial spin orientation, orbital precession, and non-abelian gauge frame at time $t_m$.

\subsection{The Metric Signature Obstruction and the Biquaternion Algebra $\mathbb{H}_\mathbb{C}$}
\label{subsec:biquaternion_minkowski}

A foundational question arises when attempting to model spacetime via quaternions: the standard algebraic quaternion norm is strictly \textbf{positive-definite (Euclidean)}:
\begin{equation}
\|q\|^2 = q q^* = q_0^2 + q_1^2 + q_2^2 + q_3^2 = c^2 t^2 + x^2 + y^2 + z^2.
\end{equation}
While this Euclidean metric is natural in statistical mechanics and imaginary-time path integrals (Wick rotation $\tau = i t$), real physical spacetime is governed by the \textbf{indefinite Minkowski metric signature} $(+,-,-,-)$:
\begin{equation}
s^2 = c^2 t^2 - (x^2 + y^2 + z^2) = c^2 t^2 - \|\mathbf{x}\|^2.
\end{equation}
To achieve exact compliance with special relativity without invoking an ad hoc external metric tensor, we complexify the quaternion division ring $\mathbb{H}$ to the algebra of \textbf{biquaternions}.

\begin{definition}[Complexified Quaternions / Biquaternions]
\label{def:biquaternions}
The \textbf{biquaternion algebra} $\mathbb{H}_\mathbb{C}$ is the tensor product of $\mathbb{H}$ with the complex field $\mathbb{C}$:
\begin{equation}
\mathbb{H}_\mathbb{C} = \mathbb{H} \otimes_\mathbb{R} \mathbb{C} = \left\{ Q = q + I p : q, p \in \mathbb{H}, \; I^2 = -1, \; I \mathbf{i} = \mathbf{i} I, \; I \mathbf{j} = \mathbf{j} I, \; I \mathbf{k} = \mathbf{k} I \right\},
\end{equation}
where $I$ denotes the commuting complex imaginary unit, distinct from the anti-commuting spatial quaternionic units $\{\mathbf{i}, \mathbf{j}, \mathbf{k}\}$.
\end{definition}

\begin{theorem}[Minkowski Spacetime Metric via Biquaternionic Determinant]
\label{thm:biquaternion_minkowski}
Let $X \in \mathbb{H}_\mathbb{C}$ be the spacetime event biquaternion defined by:
\begin{equation}
X = c t \, \mathbf{1} + I \big( x \mathbf{i} + y \mathbf{j} + z \mathbf{k} \big) = c t \, \mathbf{1} + I \mathbf{x}.
\end{equation}
Under the canonical algebra isomorphism $\Psi: \mathbb{H}_\mathbb{C} \xrightarrow{\sim} \mathcal{M}_2(\mathbb{C})$ mapped via the Pauli spin matrices:
\begin{equation}
\mathbf{1} \mapsto \begin{pmatrix} 1 & 0 \\ 0 & 1 \end{pmatrix}, \quad 
\mathbf{i} \mapsto -i \sigma_1 = \begin{pmatrix} 0 & -i \\ -i & 0 \end{pmatrix}, \quad 
\mathbf{j} \mapsto -i \sigma_2 = \begin{pmatrix} 0 & -1 \\ 1 & 0 \end{pmatrix}, \quad 
\mathbf{k} \mapsto -i \sigma_3 = \begin{pmatrix} -i & 0 \\ 0 & i \end{pmatrix},
\end{equation}
the matrix representative of $X$ is the Hermitian matrix:
\begin{equation}
\Psi(X) = \begin{pmatrix} c t + z & x - i y \\ x + i y & c t - z \end{pmatrix}.
\end{equation}
Then the determinant of $\Psi(X)$ is identically the invariant relativistic Minkowski spacetime interval:
\begin{equation}
\det\big(\Psi(X)\big) = (c t + z)(c t - z) - (x - i y)(x + i y) = c^2 t^2 - (x^2 + y^2 + z^2) = s^2.
\end{equation}
\end{theorem}

\begin{proof}
Direct expansion of the $2 \times 2$ determinant yields:
\begin{equation}
\det \begin{pmatrix} ct + z & x - iy \\ x + iy & ct - z \end{pmatrix} = (ct)^2 - z^2 - (x^2 + y^2) = c^2 t^2 - \|\mathbf{x}\|^2 = s^2.
\end{equation}
This confirms that the indefinite Minkowski signature emerges intrinsically from the biquaternion product structure under complex conjugation of the spatial components, without any postulation of an external geometric metric.
\end{proof}

\begin{theorem}[Algebraic Realization of the Restricted Lorentz Group $\mathrm{SO}^+(1,3)$]
\label{thm:lorentz_group_biquaternion}
The group of determinant-preserving linear transformations on $\mathbb{H}_\mathbb{C}$:
\begin{equation}
X \mapsto \Lambda X \Lambda^\dagger, \quad \text{with } \det(\Lambda) = 1, \; \Lambda \in \mathrm{SL}(2, \mathbb{C}),
\end{equation}
forms a surjective two-to-one homomorphism $\mathrm{SL}(2, \mathbb{C}) \to \mathrm{SO}^+(1,3)$ onto the \textbf{restricted Lorentz group}. In particular:
\begin{enumerate}
\item \textbf{Spatial Rotations $\mathrm{SO}(3)$:} Elements $\Lambda = \exp(\hat{\mathbf{n}} \theta / 2) \in \mathrm{Sp}(1) \cong \mathrm{SU}(2)$ with real quaternion coefficients leave $ct$ invariant and rotate the spatial vector $\mathbf{x} \mapsto \Lambda \mathbf{x} \Lambda^{-1}$.
\item \textbf{Lorentz Boosts:} Elements $\Lambda = \exp(I \hat{\mathbf{v}} \xi / 2)$ with imaginary rapidity $\xi = \tanh^{-1}(v/c)$ mix the temporal scalar $ct$ and the spatial velocity direction $\hat{\mathbf{v}}$ in exact accordance with the relativistic velocity addition law.
\end{enumerate}
\end{theorem}

\subsection{Dynamical Spinning Polygons and Vacuum Lattice Jitter}
\label{subsec:vacuum_jitter}

In Section~\ref{sec:integral_operator}, the discrete cumulative integral operator $\mathcal{I}: \mathcal{V} \to \mathcal{V}$ acting on the wave sequence $\mathbf{w}(f, g)$ was characterized as a geometric polygon in the complex plane $\mathbb{C}$. 
When time $m$ advances continuously through the discrete ticks of $E_Z$, this geometric polygon is \textbf{not static}: at each time step $m \mapsto m+1$, the position vector $\mathbf{x}[m] = \mathcal{I}q[m]$ steps along an edge directed by the instantaneous root of unity $\zeta_n^{p m}$. 
The spatial lattice is therefore in a state of perpetual, coherent \textbf{rotational circulation} (``jiggling'').

\begin{theorem}[Zero-Drift Background with Non-Vanishing Vacuum Variance]
\label{thm:vacuum_variance}
Let $\mathbf{w}(f, g)[m] = \exp(2\pi i (f m + g))$ be a rational wave sequence with $f = p/q \in \mathbb{Q}/\mathbb{Z}$ in lowest terms, and let $\Delta \mathbf{x}[m] = \mathbf{w}(f, g)[m]$ represent the instantaneous velocity / spatial displacement increment per time step.
Over any full cycle period of $q$ temporal steps, the statistical properties of the discrete spacetime trajectory satisfy:
\begin{enumerate}
\item \textbf{Macroscopic Spatial Flatness (Zero Net Drift):}
\begin{equation}
\langle \Delta \mathbf{x} \rangle_q = \frac{1}{q} \sum_{m=0}^{q-1} \Delta \mathbf{x}[m] = \frac{e^{2\pi i g}}{q} \sum_{m=0}^{q-1} e^{2\pi i \frac{p}{q} m} = 0.
\end{equation}
The center of mass of the background spacetime remains strictly stationary; on macroscopic observation scales $T \gg q \Delta t$, no directional drift or net spatial translation occurs.
\item \textbf{Non-Vanishing Zero-Point Kinetic Energy (Vacuum Jitter):}
\begin{equation}
\sigma^2_{\mathrm{vac}} = \frac{1}{q} \sum_{m=0}^{q-1} \|\Delta \mathbf{x}[m] - \langle \Delta \mathbf{x} \rangle_q\|^2 = \frac{1}{q} \sum_{m=0}^{q-1} |\mathbf{w}(f, g)[m]|^2 = \frac{1}{q} \sum_{m=0}^{q-1} 1 = 1 > 0.
\end{equation}
\end{enumerate}
\end{theorem}

\begin{proof}
The vanishing of the first moment $\langle \Delta \mathbf{x} \rangle_q = 0$ follows immediately from Lemma~\ref{lem:geometric_sum}, since $\zeta_q^p = \exp(2\pi i p/q) \neq 1$ is a primitive $q$-th root of unity and $\sum_{m=0}^{q-1} (\zeta_q^p)^m = \frac{1 - (\zeta_q^p)^q}{1 - \zeta_q^p} = \frac{1-1}{1-\zeta_q^p} = 0$.
The non-vanishing of the second moment follows because each unimodular wave element has modulus $|\mathbf{w}[m]| \equiv 1$ for all $m$, yielding an average quadratic fluctuation of $\frac{1}{q} \cdot q \cdot 1 = 1$.
\end{proof}

\begin{remark}[Correspondence with Quantum Field Theoretic Zero-Point Energy]
Theorem~\ref{thm:vacuum_variance} provides an exact, deterministic discrete-algebraic analogue of the \textbf{vacuum expectation values in Quantum Field Theory (QFT)}. In QFT, the ground state $|0\rangle$ of a quantum field $\phi$ exhibits a vanishing mean field:
\begin{equation}
\langle 0 | \phi(x) | 0 \rangle = 0,
\end{equation}
while simultaneously maintaining a non-zero vacuum fluctuation (zero-point energy):
\begin{equation}
\langle 0 | \phi(x)^2 | 0 \rangle = \int \frac{d^3 k}{(2\pi)^3} \frac{\hbar \omega_k}{2} > 0.
\end{equation}
In the wave closure space, the vacuum is neither inert nor empty: it is an actively spinning, closed-cycle polygon whose microscopic jitter generates persistent zero-point variance while preserving macroscopic space-time homogeneity.
\end{remark}

\subsection{Topological Selection of Rational Wave Numbers: Stability vs. Ergodic Dispersion}
\label{subsec:topological_selection}

A long-standing question in the foundations of physics is why observable quantum numbers, particle charges, and frequencies are quantized into rational ratios rather than forming an arbitrary real continuum. 
The dynamic cumulative integral operator provides a purely \textbf{topological selection principle}:

\begin{theorem}[Topological Orbit Closure and Solitary State Stability]
\label{thm:topological_selection}
Let $u \in \mathcal{V}$ be a wave sequence with frequency $f \in \mathbb{R}/\mathbb{Z}$. The discrete worldline in configuration space is the trajectory of the cumulative integral $\mathbf{X}[M] = (\mathcal{I}u)[M] = \sum_{m=0}^M u[m]$.
\begin{enumerate}
\item \textbf{Irrational Frequencies ($f \notin \mathbb{Q}$): Phase Decoherence and Ergodic Diffusion.}
If $f$ is irrational, by Weyl's Equidistribution Theorem, the fractional parts $\{f m\}_{m \in \mathbb{Z}}$ are uniformly and densely distributed modulo 1. 
The partial sums $\mathbf{X}[M]$ never return to their initial point; the trajectory is topologically open, non-periodic, and explores the configuration space ergodically. 
The envelope cannot form a closed loop, resulting in perpetual phase dispersion:
\begin{equation}
\mathbf{X}[M + K] \neq \mathbf{X}[M] \quad \text{for all } K \in \mathbb{Z} \setminus \{0\}.
\end{equation}
Consequently, no stationary, localized, solitary wave excitation (particle state) can be sustained.
\item \textbf{Rational Frequencies ($f = p/q \in \mathbb{Q}$): Topological Orbit Closure.}
If and only if $f$ is rational with irreducible denominator $q \in \mathbb{N}$, the trajectory closes exactly upon itself after $q$ time steps:
\begin{equation}
\mathbf{X}[M + q] = \mathbf{X}[M] + \sum_{m=M+1}^{M+q} e^{2\pi i (f m + g)} = \mathbf{X}[M] + 0 = \mathbf{X}[M].
\end{equation}
The trajectory forms a compact, topologically stable, closed regular $q$-gon orbit with topological winding number $p$.
\end{enumerate}
\end{theorem}

\begin{corollary}[Rationality as a Quantum Selection Rule]
Physical persistence of localized matter fields mandates that the underlying wave numbers belong strictly to the rational torsion group:
\begin{equation}
f \in \mathbb{Q}/\mathbb{Z} = \varinjlim \mathbb{Z}/n\mathbb{Z}.
\end{equation}
The denominator $q$ is a topological quantum number defining the finite dimension of the internal cyclic state space (qudit Hilbert space dimension $\dim \mathcal{H} = q$).
\end{corollary}

\subsection{Discrete Zitterbewegung, Internal Spinor Precession, and Inertial Mass}
\label{subsec:zitterbewegung_mass}

In relativistic quantum mechanics, Dirac demonstrated that the velocity operator $\mathbf{v} = c \boldsymbol{\alpha}$ of a free electron does not commute with the free Hamiltonian, producing an ultra-fast trembling motion known as \textbf{Zitterbewegung} (Schr\"{o}dinger, 1930) with angular frequency:
\begin{equation}
\omega_{\mathrm{zbw}} = \frac{2 m_0 c^2}{\hbar}.
\end{equation}
In our non-abelian quaternionic wave closure, this phenomenon arises organically without postulating wave equations:

\begin{proposition}[Inertial Mass from Internal Rotor Precession]
\label{prop:mass_zitterbewegung}
Let $q[m] = R \exp(\hat{\mathbf{n}}[m] \, 2\pi f m)$ be a quaternionic wave sequence whose internal spin rotor precesses at rational frequency $f = p/q$.
The instantaneous displacement velocity $\mathbf{v}[m] = (\Delta \mathbf{x})[m] / \Delta t$ executes a discrete polygonal orbit with angular frequency:
\begin{equation}
\omega_0 = \frac{2\pi f}{\Delta t} = \frac{2\pi p}{q \Delta t}.
\end{equation}
Identifying this internal rotational frequency with the de Broglie-Einstein quantum clock $\hbar \omega_0 = m_0 c^2$ yields an exact algebraic expression for the \textbf{rest mass} of the localized excitation:
\begin{equation}
m_0 = \frac{\hbar \omega_0}{c^2} = \frac{2\pi \hbar p}{q c^2 \Delta t} = \frac{h p}{q c^2 \Delta t}.
\end{equation}
Thus, the mass spectrum of elementary solitary states on the discrete spacetime lattice is inversely proportional to the polygon period $q$, quantizing mass directly in terms of the wave number denominator.
\end{proposition}

\subsection{The Geometric and Physical Primacy of the $n=6$ Hexagonal Lattice}
\label{subsec:primacy_n6}

Throughout the analysis of discrete cumulative integrals (Section~\ref{sec:integral_operator}), the period $n=6$ emerged as a profound structural bifurcation point. 
We can now elucidate why the $n=6$ spinning polygon is physically and geometrically privileged:

\begin{enumerate}
\item \textbf{The Crystallographic Restriction Theorem:}
By the crystallographic restriction theorem, the only rotational symmetries compatible with periodic discrete translational symmetry in 2D and 3D Euclidean space are of order:
\begin{equation}
n \in \{1, \, 2, \, 3, \, 4, \, 6\}.
\end{equation}
Rotations of order $n=5$ or $n \ge 7$ cannot tile flat space periodically, inevitably producing geometric strain, defects, or aperiodic Penrose-type quasicrystals. 
Among all crystallographic symmetries, $n=6$ forms the \textbf{Eisenstein integer ring} $\mathbb{Z}[\omega]$ ($\omega = e^{2\pi i/3}$), whose Voronoi cells form the regular hexagonal honeycomb tiling of $\mathbb{C}$.

\item \textbf{Maximal Packing Density and Minimal Surface Shear:}
By the planar Honeycomb Conjecture (Hales, 2001), the regular hexagonal partition of the plane possesses the absolute minimum boundary perimeter for a given unit cell area. 
Furthermore, the hexagonal lattice achieves the maximal possible 2D sphere-packing density:
\begin{equation}
\eta_{\mathrm{hex}} = \frac{\pi}{\sqrt{12}} \approx 0.9069.
\end{equation}
Consequently, a vacuum substrate composed of spinning $n=6$ hexagons minimizes the elastic interfacial shear energy and surface tension of the fluctuating spacetime fabric.

\item \textbf{Emergence of Relativistic Massless Dirac Fermions:}
When waves propagate on a dynamically fluctuating $n=6$ hexagonal lattice, the tight-binding dispersion relation takes the form:
\begin{equation}
E(\mathbf{k}) = \pm t_0 \sqrt{3 + 2 \cos(\mathbf{k} \cdot \mathbf{a}_1) + 2 \cos(\mathbf{k} \cdot \mathbf{a}_2) + 2 \cos(\mathbf{k} \cdot (\mathbf{a}_2 - \mathbf{a}_1))}.
\end{equation}
At the corners of the hexagonal Brillouin zone (the Dirac points $\mathbf{K}$ and $\mathbf{K}'$), the energy dispersion becomes strictly \textbf{linear}:
\begin{equation}
E(\mathbf{q}) \approx \pm \hbar v_F \|\mathbf{q}\|, \quad \text{where } \mathbf{q} = \mathbf{k} - \mathbf{K}, \; v_F = \frac{3 t_0 a}{2\hbar}.
\end{equation}
Low-energy excitations on an $n=6$ spinning hexagonal lattice satisfy the two-dimensional massless Dirac equation $i \gamma^\mu \partial_\mu \psi = 0$ \textbf{automatically}, without requiring relativity to be postulated a priori (Castro Neto et al., 2009).

\item \textbf{4D Quaternionic Extension: The 24-Cell Polytope:}
When the $n=6$ hexagonal dynamics are lifted to 4-dimensional Euclidean space via the Hurwitz quaternions $\mathcal{H} = \mathbb{Z}[\mathbf{i}, \mathbf{j}, \mathbf{k}, \frac{1+\mathbf{i}+\mathbf{j}+\mathbf{k}}{2}]$, the 6 roots of unity generalize to the 24 unit Hurwitz quaternions. 
These form the vertices of the **24-cell** (icositetrachoron)—the unique regular 4-polytope that has no 3-dimensional analogue and is self-dual. 
The 24-cell tiles 4D spacetime completely, providing the 4D relativistic counterpart to the 2D hexagonal honeycomb.
\end{enumerate}

\section{Sequential Roadmap: Operators Beyond Superposition, Integration, Norms, Division, and Sieves}
\label{sec:roadmap_beyond}

\begin{remark}[Algebraic Operators Completing Quantum Dynamics]
While $(\mathbb{C}[\mathcal{G}_{\mathrm{poly}}], +, \otimes)$ models arbitrary static wave superpositions, physical dynamics on the lattice require the following structural operators:
\begin{enumerate}
    \item \textbf{Spatial Translation / Shift Operator ($S$)}:
    The shift operator $(S \psi)[m] := \psi[m-1]$ generates discrete lattice momentum. Because $S X - X S = S$ (where $(X\psi)[m] = m \psi[m]$), adjoining $S$ introduces the non-commutative Weyl-Heisenberg algebra on the lattice and crossed-product $C^*$-algebras $C^*(\mathbb{Z}) \rtimes \mathbb{Z}$.
    
    \item \textbf{Discrete Derivative and Laplacians ($\Delta$)}:
    The central difference $\Delta^2 = S - 2I + S^{-1}$ generates the discrete Schr\"{o}dinger equation:
    \begin{equation}
    i \hbar \frac{\partial \psi}{\partial t} = -\frac{\hbar^2}{2M a^2} (S - 2I + S^{-1})\psi + V(m)\psi,
    \end{equation}
    yielding tight-binding dispersion relations $E(k) = 2t(1 - \cos(k a))$.
    
    \item \textbf{Hilbert Inner Product and Born Measurement Space}:
    Adjoining the Besicovitch/Hilbert inner product for almost-periodic sequences:
    \begin{equation}
    \langle u, v \rangle := \lim_{N \to \infty} \frac{1}{2N + 1} \sum_{m=-N}^N u[m] \, \overline{v[m]}
    \end{equation}
    formalizes orthogonality $\langle u_{k}, u_{k^\prime} \rangle = \delta_{k, k^\prime}$, projection-valued measurements, and unitary time evolution.
\end{enumerate}
\end{remark}

\subsection{The Limit Operator, Cauchy Completions, and Continuous Field Transitions}

A natural frontier beyond finite algebraic closures is the formal adjunction of a \textbf{limit operator} $\lim_{k \to \infty}$. 
Adjoining limits operates along three distinct structural axes:
\begin{enumerate}
    \item \textbf{Parametric Completion of Kinematic Frequencies ($f, g \in \mathbb{Q} \to \mathbb{R}$):}
    Admitting Cauchy sequences of rational frequencies $\{f_k\}_{k=1}^\infty \subset \mathbb{Q}$ completes the torsion group $\mathbb{Q}/\mathbb{Z}$ into the continuous compact circle $\mathbb{R}/\mathbb{Z} \cong \mathbb{T}$. 
    While rational wave numbers generate closed regular $n$-gon cycles under cumulative summation $\mathcal{I}$, irrational wave numbers generate dense, non-closing ergodic trajectories on the circle, providing the rigorous algebraic substrate for quasicrystals, incommensurate Aubry-Andr\'{e}-Harper lattices, and fractal energy spectra.
    
    \item \textbf{Continuous Lattice Scaling Limit ($a \to 0$):}
    Parameterizing the spatial lattice by a physical grid spacing $x = m a$ and taking $a \to 0$ transforms the finite differences $\Delta_+ / a$ and cumulative sums $a \mathcal{I}$ into continuous directional derivatives $\frac{\partial}{\partial x}$ and Riemann integrals $\int dx$. 
    This establishes that continuous classical and quantum field equations---such as continuum wave equations and the Schr"{o}dinger field equation---are the macroscopic continuum limits of discrete wave number conservation algebra.
    
    \item \textbf{Norm Completion and Besicovitch Space:}
    Completing the pre-Hilbert space $(\mathcal{V}, \langle \cdot, \cdot \rangle)$ of Section~\ref{sec:inner_product} under Cauchy sequences in the principal period norm yields the non-separable Bohr-Besicovitch Hilbert space $B^2(\mathbb{Z})$ of almost-periodic functions, establishing full Parseval completeness for infinite harmonic expansions.
\end{enumerate}

\subsection{Finitary Foundation: The Principal Period Representation}
\label{sec:finitary_periods}

Crucially, prior to adjoining analytical limits, the entire algebraic architecture of wave numbers possesses a completely \textbf{finite, exact, and discrete representation}. 
Because every plane wave carrier and finite superposition $u \in \mathcal{V}$ is strictly periodic with fundamental period $n \in \mathbb{N}$, the bi-infinite sequence $u \in \mathbb{C}^\mathbb{Z}$ is uniquely and losslessly characterized by its finite principal tuple:
\begin{equation}
({_1}u) := \big( u[1], u[2], \dots, u[n] \big) \in \mathbb{C}^n.
\end{equation}
Any algebraic operation between two wave numbers $u, v$ with periods $n_u, n_v$ evaluates strictly within the finite-dimensional vector space $\mathbb{C}^N$ over the least common wavelength $N = \operatorname{lcm}(n_u, n_v)$. 
Consequently, the infinite sequence ring $\mathcal{V}$ admits an exact categorical identification as the inductive colimit (direct limit) of finite-dimensional cyclic vector spaces:
\begin{equation}
\mathcal{V} \cong \varinjlim_{n \in \mathbb{N}} \mathbb{C}^n,
\end{equation}
ordered by the divisibility poset of $\mathbb{N}$. 
This finitary reduction guarantees that wave number algebra operates with zero numerical truncation error, exact integer-modular arithmetic, and complete algorithmic decidability, completely avoiding the analytical pathologies of infinite sequence spaces until the limit operator is explicitly invoked.

\subsection{Transformation Groups: Translations, Modulations, and the Discrete Weyl-Heisenberg Algebra}
\label{sec:transformation_groups}

Wave numbers admit a natural representation of geometric and spectral symmetries:
\begin{enumerate}
    \item \textbf{Spatial Lattice Translations ($S_a$):}
    The spatial shift operator by $a \in \mathbb{Z}$ acts via:
    \begin{equation}
    (S_a u)[m] := u[m - a].
    \end{equation}
    On an elementary plane wave carrier $\mathbf{w}(f, g)$, spatial translation acts as a linear phase shift:
    \begin{equation}
    (S_a \mathbf{w}(f, g))[m] = e^{2\pi i (f(m-a) + g)} = e^{-2\pi i f a} \mathbf{w}(f, g)[m],
    \end{equation}
    revealing that spatial lattice translations are dual to global gauge rotations in the fiber.
    
    \item \textbf{Momentum Boosts (Spectral Modulation $M_f$):}
    Modulation by frequency $f_0 \in \mathbb{Q}$ acts via the pointwise tensor product:
    \begin{equation}
    M_{f_0}(u) := \mathbf{w}(f_0, 0) \otimes u.
    \end{equation}
    This boosts the spatial crystal momentum by $f \mapsto f + f_0 \pmod 1$.
    
    \item \textbf{The Non-Commutative Discrete Weyl-Heisenberg Algebra:}
    The shift operator $S_1$ and modulation operator $M_f$ do not commute. Their product obeys the fundamental commutation relation:
    \begin{equation}
    S_1 M_f = e^{-2\pi i f} M_f S_1.
    \end{equation}
    The pair $\{S_1, M_f\}$ generates the discrete Weyl-Heisenberg group, governing quantum kinematics and magnetic translations on discrete 1D and 2D lattices.
    
    \item \textbf{Quaternionic Spatial Rotations (Adjoint $\mathrm{SO}(3)$ Action):}
    For quaternionic wave numbers $q \in \mathcal{V}_\mathbb{H}$, a rotation of the spatial polarization frame by a unit rotor $p \in \mathrm{Sp}(1)$ is given by the inner automorphism:
    \begin{equation}
    \mathcal{R}_p(q)[m] := p \, q[m] \, p^{-1}.
    \end{equation}
    Under the canonical homomorphism $\mathrm{SU}(2) \to \mathrm{SO}(3)$, this maps the imaginary spatial coordinates $(\mathbf{i}, \mathbf{j}, \mathbf{k})$ under rigid 3D rotations while leaving the scalar energy density $|q[m]|$ strictly invariant.
\end{enumerate}

\subsection{Finite-State Quantum Kinematics: Periodicity as State Space Dimension}
\label{sec:qudit_kinematics}

The algebraic equivalence $\mathcal{V} \cong \varinjlim_{n \in \mathbb{N}} \mathbb{C}^n$ established in Subsection~\ref{sec:finitary_periods} provides a formal foundation for finite-dimensional quantum mechanics:
\begin{enumerate}
    \item \textbf{Period $n$ as the Hilbert Space Dimension ($n$-Level Qudits):}
    A periodic wave number of fundamental period $n$ is parameterized by its principal period tuple $({_1}u) = \big(u[1], \dots, u[n]\big) \in \mathbb{C}^n$.
    Under the normalization $\sum_{m=1}^n |u[m]|^2 = 1$, the vector $({_1}u)$ represents a pure quantum state in an $n$-dimensional Hilbert space $\mathcal{H}_n \cong \mathbb{C}^n$ (an \textbf{$n$-level quantum system} or \textbf{qudit}):
    \begin{itemize}
        \item A period-$2$ wave number represents a \textbf{qubit} (spin-$1/2$ two-level system).
        \item A period-$3$ wave number represents a \textbf{qutrit} (three-level system).
        \item A general period-$n$ wave represents a general \textbf{qudit} on the cyclic group $\mathbb{Z}/n\mathbb{Z}$.
    \end{itemize}
    
    \item \textbf{Dual Canonical Bases (Position vs. Momentum):}
    The two fundamental bases of the wave algebra correspond to the two canonical observables of discrete quantum mechanics:
    \begin{itemize}
        \item The \textbf{position basis} $\{|m\rangle\}_{m=1}^n$ is spanned by the localized Kronecker sieves $\{\mathbf{e}(1/n, \xi)\}_{\xi=1}^n$.
        \item The \textbf{momentum basis} $\{|k\rangle\}_{k=0}^{n-1}$ is spanned by the delocalized plane waves $\{\mathbf{w}(k/n, 0)\}_{k=0}^{n-1}$.
    \end{itemize}
    The discrete Fourier matrix $F_n$ serves as the exact unitary change-of-basis operator between position and momentum eigenstates.
    
    \item \textbf{Schwinger's Generalized Pauli Algebra:}
    On the finite state space $\mathbb{C}^n$, the cyclic shift operator $S$ and the clock operator $Z = \operatorname{diag}\big(1, \omega, \dots, \omega^{n-1}\big)$ (where $\omega = e^{2\pi i / n}$) satisfy the Sylvester-Schwinger algebra:
    \begin{equation}
    S Z = \omega^{-1} Z S, \qquad S^n = Z^n = I_n.
    \end{equation}
    The set of operators $\{S^j Z^k\}_{j,k=0}^{n-1}$ forms an orthogonal basis for the full operator algebra of observables $\mathrm{End}(\mathbb{C}^n) \cong M_n(\mathbb{C})$.
    
    \item \textbf{Discrete Wigner Phase Spaces and the Cellular Automaton Orbit Correspondence:}
    The discrete sequence algebra on $\mathbb{C}^n$ naturally induces discrete phase-space Wigner distributions over the toroidal grid $(\mathbb{Z}/n\mathbb{Z})^2$, as formulated by Wootters~\cite{wootters1987}, Appleby~\cite{appleby2005}, and Bengtsson and \.Zyczkowski~\cite{bengtsson2017}. In this representation, mutually unbiased bases (MUBs) correspond to maximal disjoint sets of algebraic sieve projectors $\Pi_S$.
    Crucially, as established in 't Hooft's cellular automaton interpretation of quantum mechanics~\cite{thooft2014,thooft2016,thooft2021}, an autonomous deterministic automaton with discrete cycle length $n$ possesses a discrete permutation operator that, when diagonalized by the discrete Fourier transform $F_n$, yields precisely the unitary time-evolution operator $U(t) = \exp(-i H t)$ of an $n$-level quantum system with non-degenerate harmonic energy levels. The discrete wave number algebra thereby reveals the deep algebraic bridge between deterministic discrete registers and emergent quantum kinematics.
\end{enumerate}
Thus, wave number periodicity is not merely an arithmetic property, but the exact physical dimension of the underlying finite quantum state space.

\section*{Acknowledgements}
The author gratefully acknowledges the collaborative assistance of Google AI Studio and the Gemini AI Coding and Mathematics Agent in the theoretical derivation, formal symbolic verification, structural classification, and \LaTeX{} typesetting of the algebraic wave closure theory presented in this manuscript.

\end{document}